\documentclass[a4paper,10pt]{article}

\usepackage[T1]{fontenc}
\usepackage[english]{babel}
\usepackage[babel]{csquotes}%pour les guillemets
\MakeAutoQuote{«}{»}%idem
\usepackage{fullpage}
\usepackage{graphicx}
\usepackage{caption}
\usepackage{subcaption}
\usepackage{xcolor}
\usepackage{wrapfig}
\usepackage{tikz}
\usepackage{hyperref}
\usepackage{float}
\usepackage{amsthm}
\usepackage{amsmath}
\usepackage{amssymb}
\usepackage{amsfonts}
\usepackage{mathrsfs}
\usepackage{bbm}
\usepackage[
    backend=biber,
    style=numeric,
    giveninits=true,
    doi=false,
    url=false,
  maxbibnames=99
]{biblatex}
\renewbibmacro{in:}{}
\usepackage{indentfirst}
\fontsize{30pt}{18pt}
\hypersetup{colorlinks=true,linkcolor = black, anchorcolor = red, citecolor = blue,filecolor = red, urlcolor = blue}

\title{Long-time dynamics of partially dissipative hyperbolic systems with non-autonomous coefficients}
\author{Timothée Crin-Barat, Ling-Yun Shou and Qimeng Zhu}
\date{}
\newcommand{\A}{{\mathbb A}}
\newcommand{\B}{{\mathbb B}}
\newcommand{\C}{{\mathbb C}}

\newcommand{\dd}{\mathrm d}

\renewcommand{\L}{{\mathbb L}}

\newcommand{\N}{{\mathbb N}}

\newcommand{\R}{{\mathbb R}}
\newcommand{\SSS}{{\mathbb S}}

\newcommand{\Z}{{\mathbb Z}}

\newcommand{\sC}{{\mathcal C}}

\newcommand{\sH}{{\mathcal H}}
\newcommand{\sI}{{\mathcal I}}

\newcommand{\sL}{{\mathcal L}}

\newcommand{\sO}{{\mathcal O}}

\newtheorem{rem}{Remark}
\newtheorem{defi}{Definition}
\newtheorem{lem}{Lemma}
\newtheorem{prop}{Proposition}
\newtheorem{thm}{Theorem}

\def\ddj{\dot\Delta_j}

\def\dz1{\nabla Z_1}
\def\div{ \hbox{\rm div}\,  }

\newcommand{\e}{\exists}

\newcommand{\mc}[1]{\left\|#1\right\|_{L^2}}%mochang
\newcommand{\mcbe}[2]{\left\|#1\right\|_{\dot{\B}_{2,1}^{#2}}}%mochang

\begin{document}
\allowdisplaybreaks[4]
\def \with {\quad\!\hbox{with}\!\quad}
\def \andf {\quad\!\hbox{and}\!\quad}
\newcommand{\cR} {\mathcal{R}} 
\newcommand{\norme}[1]{\left\Vert #1\right\Vert}
\newcommand{\cA} {\mathcal{A}} 
\newcommand{\cB} {\mathcal{B}}    
\newcommand{\cC} {\mathcal{C}}     
\newcommand{\cD} {\mathcal{D}}   
\newcommand{\cF} {\mathcal{F}}     
\newcommand{\cH} {\mathcal{H}}     
\newcommand{\cI} {\mathcal{I}}     
\newcommand{\cJ} {\mathcal{J}}     
\newcommand{\cL} {\mathcal{L}}      
\newcommand{\cM} {\mathcal{M}}  
\newcommand{\cN} {\mathcal{N}}  
\newcommand{\cO} {\mathcal{O}}         
\newcommand{\cS} {\mathcal{S}}      
\newcommand{\cW} {\mathcal{W}}      
\newcommand{\cZ} {\mathcal{Z}}    
\def\dn{\delta\!n}
\def\du{\delta\!u}
\def\dv{\delta\!v}
\def\dD{\delta\!D}
\def\dZ_1{\delta\!Z_1}
\def\dV{\delta\!V}
\def\e{\varepsilon}
\def\d{\partial}
\def\l{\alpha}
\def\wh{\widehat}
\def\wt{\widetilde}
\def\pa{\partial}

\maketitle

\begin{abstract}
We study quasilinear symmetrizable partially dissipative hyperbolic systems with non-autonomous relaxation coefficients in $\mathbb{R}^d$ ($d\geq1$). The existence of global strong solutions is established in a critical regularity setting for systems satisfying the so-called Shizuta-Kawashima (SK) and entropy conditions. When the initial data are additionally bounded in a lower-regularity norm, we prove that the corresponding solutions converge to equilibrium at optimal algebraic decay rates. Furthermore, we show that the conservative part of the solution behaves asymptotically as the solution
of a non-autonomous parabolic equation. Our results apply to the compressible Euler system with the time-dependent damping coefficient $\frac{K}{(1+t)^{\alpha}}$ ($\alpha<1$, $K>0$ or $\alpha=1$, $K\gg 1$) in the velocity equation.

%   In particular, the presence of non-autonomous relaxation coefficients leads to coefficient-dependent behavior, and the weakening of the damping coefficient may result in
% faster decay to equilibrium.

The natural low/high-frequency splitting of the autonomous theory persists in the non-autonomous setting, but with a frequency-threshold that evolves in time. To handle this moving frequency structure, we introduce a new class of hybrid Besov spaces adapted to time-dependent thresholds and derive hypocoercive estimates in each frequency  regime. Our results reveal the qualitative and quantitative effects of general time-dependent relaxation coefficients on dissipation and large-time dynamics.
\end{abstract}

% \tableofcontents

\section{Introduction}
\subsection{Presentation of the system}
We are concerned with first-order $n$-component hyperbolic systems of the form:
\begin{equation}
\partial_{t} V+ \sum_{k=1}^d \partial_{x_{k}} F^{k}(V)=\frac{G(V)}{b(t)}, \label{system}
\end{equation}
where the unknown $V=V(t,x)$ depends on the time variable $t\in\R_+$ and the spatial variable $x\in\R^d\ (d\geqslant 1)$, the functions $F^k(V) (k=1,...,d)$ are smooth vector-valued functions defined on an open subset $\sO_V$ of $\R^n$, $G(V)$ is a vector-valued function defined on $\sO_V$, and $b(t)$ is a positive time-dependent function.  System \eqref{system} is supplemented with the initial condition 
\begin{align} 
V(0,x)=V_0(x).\label{d}
\end{align}

Systems of the form \eqref{system} constitute a class of hyperbolic balance laws with relaxation. The coefficient $1/b(t)$ may be interpreted as a time-dependent friction force exerted by the surrounding medium, which allows one to describe damping mechanisms whose strength varies in time, and leads to new difficulties in the description of the large-time dynamics.
Such coefficients arise in evolution equations posed on expanding backgrounds, that is, media whose spatial scale increases with time, where the expansion induces a friction-like term, commonly referred to in cosmology as \emph{Hubble friction} \cite{Ryden2017}.

In this paper, we focus on the existence and the asymptotic behavior of global-in-time strong solutions to the Cauchy problem when the initial data $V_0$ is close to some constant state $\bar{V}$ that satisfies $G(\bar{V})=0$. 

\medbreak
In the \textit{no-damping} scenario, i.e., when $G \equiv 0$, it is well known that systems of conservation laws, supplemented with initial data in Sobolev spaces $H^s$ with $s>d/2+1$, admit local-in-time strong solutions, see \cite{Benzoni-Gavage2006}. However, even for small initial data, solutions may develop singularities in finite time, see \cite{Majdalocal,Serre}. On the other hand, in the autonomous fully dissipative case, i.e.,  if $b\equiv1$ and all the eigenvalues of $DG(\bar{V})$ have negative real parts, then small perturbations of $\bar{V}$ give rise to global-in-time solutions that tend exponentially fast to $\bar{V}$; see Li \cite{LiTT}.

The situation of interest here, which arises in many physical models described by systems of the form \eqref{system}, involves only partial dissipation; namely, the term $G(V)$ acts only on a subset of the components of the solution. Typically, this occurs in gas dynamics, where the mass density and entropy equations are conservative, while the momentum equation can include dissipative effects (diffusion or friction). Here, we will assume that the linearization of $G$ around the equilibrium gives rise to a dissipative operator with dissipative components in some parts of the solution\footnote{We do not address the case where $G$ has dissipative nonlinear terms. The Fourier-based analysis that we employ is not well-suited to deal with this. We refer to \cite{MochizukiMotai,CBLSZ25} and the references therein for the analysis of systems with nonlinear damping terms.}. A representative example is the governing compressible Euler equations for barotropic gas with time-dependent damping:
\begin{equation}\label{eq:euler}
\left\{\begin{aligned}&\partial_t\rho +{\rm div}(\rho v)=0,\\
&\partial_t(\rho v)+{\rm div}(\rho v\otimes v) +\nabla P(\rho) = -\frac{K\rho v}{(1+t)^{\alpha}},
\end{aligned}\right.\end{equation}
where $\rho$ corresponds to the density of a fluid, $v$ denotes the velocity, and $P(\rho)$ represents the pressure. When $\alpha=0$, the system \eqref{eq:euler} is the classical damped compressible Euler system, which describes the motion of compressible fluids in porous media.

\subsection{Aims of the paper}

Numerous works have been dedicated to the study of partially dissipative hyperbolic systems. In the time-independent case, namely when $b\equiv 1$, the global-in-time existence of small strong solutions holds under the Shizuta-Kawashima condition. In that setting, the damping mechanism is time-independent, and the dissipation transferred from the damped components to the conservative ones can be described through a fixed frequency decomposition.

In the non-autonomous case, although the damping strength $1/b(t)$ may decay in time, the effective
parabolic dynamics involves the diffusion coefficient $b(t)$, which may grow
and hence lead to faster decay. A typical example is often assumed to take the form
\begin{align}\label{explicitb} b(t)=\dfrac{(1+t)^\alpha}{K}
\end{align}
for some constants $K>0$ and $\alpha\in \R$. As we will describe in more detail below, the corresponding behaviors are known to depend on the parameters $\alpha$ and $K$, as well as on the spatial dimension. When $\alpha<1$, global-in-time solutions near equilibrium can be constructed, while in the critical case $\alpha=1$, the behavior of the solution depends on the value of $K$, and global solutions exist when $K$ is large enough. In general, when $\alpha>1$, the damping mechanism becomes too weak to prevent the blow-up of solutions, cf. \cite{pan2}.

The purpose of the present paper is to investigate how general time-dependent coefficients $b(t)$ affect the dissipation structure and the large-time behavior of partially dissipative hyperbolic systems. More precisely, we establish a global well-posedness theory near equilibrium in a critical regularity framework and derive optimal decay estimates whose rates are explicitly determined by the behavior of $b(t)$. One feature is that, although the damping strength $\frac{1}{b(t)}$ becomes weaker as time evolves, the corresponding diffusive effect becomes stronger, leading to faster decay rates. We also highlight the related diffusion phenomenon: the conservative part of the solution behaves asymptotically like the solution of a parabolic equation whose diffusion coefficient is precisely given by $b(t)$.

\subsection{Literature on partially dissipative systems with constant damping coefficients}

Here, we recall some results for the autonomous case \(b\equiv 1\). Due to the partial nature of the dissipation, one of the fundamental issues is to identify structural conditions on the source term $G(V)$ that are strong enough to prevent finite-time breakdown of small smooth solutions. For constant damping coefficients, this question has been extensively investigated. A first important step was made by Chen, Levermore and Liu \cite{chen1994}, who introduced a dissipative entropy framework extending the classical entropy theory of Godunov \cite{Godunov2} and Friedrichs--Lax \cite{FL} for conservation laws. This condition, however, does not by itself yield a complete global existence theory.

The compressible Euler system with constant damping has been the subject of
extensive investigations in the literature. Hsiao and Liu~\cite{HsiaoL:1992CMP} first observed that solutions to the one-dimensional damped Euler equations asymptotically approach the self–similar profile of the corresponding nonlinear porous-medium equation, the so-called diffusion wave. 
Subsequently, Nishihara~\cite{Nishihara:1996JDE}, Nishihara et al.~\cite{NishiharaWY:2000JDE}, and Mei~\cite{Mei:2010SIAM} quantified the convergence rates toward such diffusion waves in various functional frameworks. For initial data that are small perturbations in Sobolev spaces $H^{s}(\mathbb{R}^{d})$ ($s\geq [\frac{d}{2}]+2$), the global well-posedness and asymptotics of classical solutions have been studied by Wang and Yang in \cite{Wang}, as well as Sideris, Thomases and Wang in \cite{Sideris}.

More generally, global existence theories for \eqref{system} with constant relaxation coefficients were further developed by combining entropy dissipation methods with the Shizuta-Kawashima condition \cite{SK}. Yong \cite{Yong} proved the global existence of classical solutions near a constant equilibrium $\bar V$ satisfying $G(\bar V)=0$, under an additional technical assumption on the entropy. A related one-dimensional result had previously been obtained by Hanouzet and Natalini \cite{HanouzetNatalini}. Later on, Kawashima and Yong \cite{KY} refined the entropy formulation and removed that extra technical assumption, thereby obtaining a global existence result in the classical Sobolev framework. The large-time behavior was then studied by Bianchini, Hanouzet and Natalini \cite{BHN}, who proved that smooth solutions converge to the equilibrium state $\bar V$ in $L^p$ with the rate $O(t^{-\frac d2(1-\frac1p)})$, for $p\in[\min\{d,2\},\infty]$, by means of Duhamel's formula and sharp Green kernel estimates for the linearized system. Kawashima and Yong \cite{KYDecay} employed a time-weighted energy method and established asymptotic decay estimates similar to those in \cite{BHN}.

More recently, Beauchard and Zuazua \cite{BZ} revisited the Shizuta-Kawashima condition from the viewpoint of hypocoercivity and proved its equivalence with the Kalman rank condition from control theory. Kawashima and Xu \cite{XK1,XK2,XK1D} extended the well-posedness theory to critical non-homogeneous Besov spaces $B^{d/2+1}_{2,1}$, a regularity framework in which the classical Sobolev theory of Kato \cite{Katolocal} and Majda \cite{Majdalocal} do not directly apply. Crin-Barat and Danchin \cite{c2,c1,c3} justified the global well-posedness in the hybrid Besov space $\B^{d/2,d/2+1}_{2,1}$ with different regularities in the low and high-frequency regimes. This approach allows us to formulate sharper smallness assumptions and to obtain a more precise description of the qualitative and quantitative properties of global solutions to general partially dissipative systems. Crin-Barat, Shou and Zuazua \cite{MR4951948} recently developed a ``physical space version" of the hyperbolic hypocoercivity approach introduced in \cite{BZ} and obtained time-decay estimates in one dimension. The authors in \cite{SXZ} considered different critical regularities on the conservative and dissipative components and provided sufficient and necessary conditions for achieving optimal upper and lower bounds of time decay rates.

\subsection{Hyperbolic systems with time-dependent damping coefficients}

%\subsubsection{General systems}
%In \cite{Wirth15}, the author investigates diffusion phenomena for partially dissipative hyperbolic systems with time-dependent coefficients under a uniform Kalman rank condition. The author proves that, under suitable assumptions, solutions behave asymptotically like those of an associated parabolic equation, combining a low-frequency WKB analysis with exponential stability estimates for high frequencies.

%\textcolor{red}{Are there other refs?}
% \subsubsection{The compressible Euler system}

\subsubsection{Link with the damped wave equation} System \eqref{eq:euler} can be rewritten as a nonlinear wave equation with time-dependent damping, which naturally connects our problem to the classical diffusion phenomenon for damped wave equations. Indeed, when a solution to \eqref{eq:euler} is viewed as a perturbation of a non-vacuum constant equilibrium $(\bar{\rho},0)$, the wave formulation of the linearized equation for the modified density perturbation $n:=\rho-\bar\rho$  (see \eqref{eulerre}) reads
\begin{equation}\label{linearwave}
\partial_t^2 n-P'(\bar{\rho})\Delta n+\frac{1}{b(t)}\partial_t n=0.
\end{equation}
Assuming that the damping coefficient satisfies \eqref{explicitb}, the
analysis of Wirth
\cite{Wirth:2004MMAS,Wirth:2006JDE,Wirth:2007JDE} shows that the
large-time behavior of \eqref{linearwave} changes according to the
strength of the damping. In the effective range $-1\leq\alpha<1$,
solutions exhibit a diffusion phenomenon and are asymptotically
described by a heat equation with a time-dependent diffusion
coefficient. In contrast, when $\alpha>1$, the damping becomes
non-effective, and the dynamics are asymptotically governed by the free
wave equation
\begin{equation*}
\partial_t^2 n-P'(\bar{\rho})\Delta n=0.
\end{equation*}

The case $\alpha=1$ is critical and corresponds to the scale-invariant damping coefficient $\frac{K}{1+t}$. In this regime, the linear
large-time behavior depends essentially on $K$, with $K=2$ marking a transition in the energy decay; see \cite{Wirth:2004MMAS}.  This
transition is consistent with the distinguished role of the condition $K>2$ in the global existence theory for the corresponding nonlinear
Euler equations, although the nonlinear threshold does not follow
simply from the time integrability of a linear decay rate. This
suggests that, at least heuristically, global small-data solutions to
\eqref{eq:euler} should be expected when $\alpha<1$, or in the
critical case $\alpha=1$ with $K>2$.

\subsubsection{Link with nonlinear damping}
Nonlinear damping laws arise naturally in stabilization problems and
feedback control, cf. \cite{ChitourLiuSontag1995,ChitourMarxPrieur2020,LiuChitourSontag1996}.
They may also be interpreted as time-dependent linear damping along the
trajectories of a solution. A simple connection is already visible for
the damped Burgers equation %Nonlinear damping can also be related to time-dependent linear damping along the trajectories of a solution. A simple connection is already visible for the damped Burgers equation
\begin{equation*}
\partial_t u+u\partial_xu+u|u|=0.
\end{equation*}
As long as the solution remains smooth, let \(X(t;x_0)\) denote the characteristic starting from \(x_0\), defined by
\begin{equation*}
\frac{\mathrm{d}}{\mathrm{d}t}X(t;x_0)
=
u\bigl(t,X(t;x_0)\bigr),
\qquad
X(0;x_0)=x_0,
\end{equation*}
and set
\begin{equation*}
U(t;x_0):=u\bigl(t,X(t;x_0)\bigr).
\end{equation*}
Then \(U\) satisfies
\begin{equation*}
\frac{\mathrm{d}}{\mathrm{d}t}U+|U|U=0,
\end{equation*}
which may be rewritten as the linear equation
\begin{equation*}
\frac{\mathrm{d}}{\mathrm{d}t}U+a_{x_0}(t)U=0,
\qquad
a_{x_0}(t):=|U(t;x_0)|.
\end{equation*}
Thus, along each characteristic, the nonlinear friction becomes linear damping with a time-dependent coefficient determined by the solution itself. In fact,
\begin{equation*}
U(t;x_0)
=
\frac{u_0(x_0)}
{1+t|u_0(x_0)|},
\qquad
a_{x_0}(t)
=
\frac{|u_0(x_0)|}
{1+t|u_0(x_0)|}.
\end{equation*}
Consequently, whenever \(u_0(x_0)\neq 0\), the effective damping coefficient behaves asymptotically as
\begin{equation*}
a_{x_0}(t)\sim \frac{1}{1+t}.
\end{equation*}
This toy model shows how a quadratic damping law naturally generates, along the flow, the critical time-dependent damping scale \((1+t)^{-1}\). %Related connections between nonlinear feedback laws and non-autonomous damping also arise in the control theory of nonlinearly damped wave equations; see, for instance, \cite{ChitourLiuSontag1995,LiuChitourSontag1996,ChitourMarxPrieur2020}.
For the compressible Euler system with nonlinear damping, a similar Lagrangian interpretation is formally available. However, the coupling with the density and pressure generates additional nonlinear terms that fall outside the framework developed in the present paper.

\subsubsection{Existing results for hyperbolic systems with time-dependent damping}

There has been a substantial body of work on the time-dependent damped Euler system \eqref{eq:euler}. In the one-dimensional case, Pan~\cite{pan1,pan2} proved that if $\alpha \in (0,1)$ with $K > 0$ or $\alpha = 1$ with $K > 2$, and if the initial data are small and smooth perturbations of a non-vacuum constant state, then the corresponding classical solution exists globally in time.
Chen \emph{et al.}~\cite{chen1} subsequently extended the global existence result to certain classes of large initial data.
When $\alpha > 1$ and $K > 0$ or $\alpha = 1, K \le 2$, the $C^1$ solution blows up in finite time; the blow-up mechanism was investigated by Sugiyama~\cite{su1}.
Convergence toward the diffusion wave profile was independently established by Cui \emph{et al.}~\cite{CuiYZZ} and Li \emph{et al.}~\cite{lht1,lht2}, where the asymptotic states at spatial infinity are distinct.
In the critical damping case, Geng \emph{et al.}~\cite{geng1} further proved convergence toward the asymptotic profile with an explicit rate depending on the physical parameter $K$.

For higher-dimensional cases, Hou and Yin~\cite{hou1} and Hou \emph{et al.}~\cite{hou2} showed that when $\alpha \in (0,1)$ with $K > 0$ or $\alpha = 1$ with $K > 3-d$, the time-dependent damped Euler system admits global smooth solutions, provided that the initial perturbation is small, curl-free, compactly supported and smooth around a non-vacuum equilibrium.
In contrast, when $\alpha > 1, K > 0$ or $\alpha = 1, K\le 3-d$, the solution blows up in finite time.
The decay rates for multidimensional solutions in the range $\alpha \in (0,1)$ were first obtained by Pan~\cite{pan3} and later refined by Ji and Mei~\cite{ji1,ji2}.
The $L^1$-weak convergence to the generalized Barenblatt self-similar solution was established by Geng \emph{et al.}~\cite{geng2}, while strong convergence in the vacuum regime toward the generalized Barenblatt profile was rigorously justified by Pan~\cite{pan4,pan5} in the one-dimensional and spherically symmetric three-dimensional settings. Liu, Sheng and Lai \cite{MR4770303} considered the global existence of small solutions in Sobolev spaces for a class of hyperbolic systems with the time-dependent damping coefficient of the form \eqref{explicitb} with $0<\alpha\leq 1$.

Recently, in \cite{CBPSZ}, the authors and Pan established the global well-posedness in hybrid critical Besov spaces and justified the global strong convergence of the relaxation limit toward a porous medium system with time-dependent diffusion coefficients for ill-prepared data.

% However, as far as we know, for general nonlinear hyperbolic systems, very few works address the influence of non-autonomous relaxation coefficients on the regularity, long-time behavior, and qualitative and quantitative properties of solutions.

To the best of our knowledge, no previous result has provided a systematic treatment of nonlinear hyperbolic systems with general non-autonomous relaxation coefficients.

\subsection{Outline of the paper}
The paper is organized as follows. In Section~\ref{sec:main}, we reformulate the system, introduce the structural assumptions, and state the main results. Section~\ref{sec:LP} recalls the Littlewood--Paley decomposition and the functional spaces used throughout the paper. In Section~\ref{sec:linear}, we derive the key estimates for the linearized system with time-dependent coefficients. Section~\ref{sec:global} is devoted to the proof of the global existence result. Section~\ref{sec:asymptotics} establishes the optimal decay estimates and the diffusion phenomenon. Finally, in Section~\ref{sect:7}, we apply the abstract theory to the compressible Euler system with time-dependent damping. In the appendix, we collect some useful tools in Besov spaces.

\section{Main results}\label{sec:main}

\subsection{Reformulation of the system and Shizuta-Kawashima condition}

Define
\begin{equation}\label{Mset}
\begin{aligned}
&\mathcal{M}:=\Big\{\psi\in\mathbb{R}^{n}~:~\langle\psi,G(V)\rangle=0\quad~\text{for any}~V\in \mathcal{O}_{V}\Big\}.
\end{aligned}
\end{equation}
Here, $\mathcal{M}$ is a subset of $\mathbb{R}^{n}$ with ${\rm{dim}}~\mathcal{M}=n_{1}$.
It follows from the definition of $\mathcal{M}$ that $G(V)\in \mathcal{M}^{\perp}$ (the orthogonal complement of $\mathcal{M}$) holds for any $V\in \mathcal{O}_{V}$.
Accordingly, one has the orthogonal decomposition $\mathbb{R}^{n}=\mathcal{M}\oplus\mathcal{M}^{\perp}$, and we write $V\in \mathbb{R}^{n}$
as
\[
V=\left(
    \begin{array}{c}
      V_{1} \\
      V_{2} \\
    \end{array}
  \right)
\]
such that $V\in\mathcal{M}$ holds if and only if $V_{2}=0$. Moreover,
we denote the set of equilibria for \eqref{system} by:
\begin{align}
&\mathcal{E}=\{V\in \mathcal{O}_{V}~:~G(V)=0\}.
\end{align}
Following \cite{KY,SK}, we define the notion of dissipative entropy for \eqref{system} to symmetrize the system.

\begin{defi}\label{defnentropy}
    Let $\eta=\eta(V)$ be a smooth function defined on a convex open set $\mathcal{O}_{V}\subset \mathbb{R}^{n}$. Then $\eta=\eta(V)$ is called an entropy for \eqref{system} if the following statements hold{\rm:}
    \begin{itemize}
    \item [$(\bullet)$]$\eta=\eta(V)$ is strictly convex in $\mathcal{O}_{V}$ in the sense that the Hessian $D^2_{V}\eta(V)$ is positive definite for $V\in \mathcal{O}_{V} ${\rm;}
    \item [$(\bullet)$]$D_{V}F^{i}(V)\big( D^2_{V}\eta(V)\big)^{-1}$ is  symmetric for $i=1,\dots,d$ and $V\in \mathcal{O}_{V} $;
    \item [$(\bullet)$] $V\in \mathcal{E}$ if and only if $\big( D_{V}\eta(V) \big)^{\top}\in \mathcal{M}${\rm;}
    \item [$(\bullet)$] For $V\in \mathcal{E}$, the matrix $D_{V}G(V)\big( D^2_{V}\eta(V)\big)^{-1}$ is symmetric and non-positive definite, and its null space coincides with $\mathcal{M}$.
    \end{itemize}
\end{defi}

In the rest of the manuscript, we assume that the system \eqref{system} admits an entropy $\eta=\eta(V)$ in the sense of Definition \ref{defnentropy}.
Introducing the unknown
\begin{equation}\label{w}
   U=U(V):=  (D_{V}\eta(V))^{\top}.
\end{equation}
Due to the convexity of $\eta(V)$, the mapping $U=U(V)$ is a diffeomorphism from $\mathcal{O}_{V}$ onto its range $\mathcal{O}_{U}$.
Let $V=V(U)$ be the inverse mapping, which is also a diffeomorphism from $\mathcal{O}_{U}$ onto its range $\mathcal{O}_{V}$.
Then \eqref{system} rewrites as
\begin{align}
\widetilde{A}^{0}(U)\partial_{t}U+\sum_{k=1}^{d}\widetilde{A}^{k}(U)\partial_{x_{k}}U=\frac{\widetilde{H}(U)}{b(t)}
\label{entropeq}
\end{align}
with
\[
\begin{aligned}
&\widetilde{A}^{0}(U)=D_{U}V(U),\\
&\widetilde{A}^{i}(U)=D_{U}F^{i}(V(U))=D_{V}F^{i}(V(U))D_{U}V(U),\\
&\widetilde{H}(U)=G(V(U)).
\end{aligned}
\]
Define
\begin{equation}\label{tildeL}
\begin{aligned}
&\widetilde{B}(U)=-D_{U}\widetilde{H}(U)=-(D_{V}G)(V(U))D_{U}V(U).
\end{aligned}
\end{equation}
We see that $\widetilde{H}(U)\in \mathcal{M}^{\perp}$ for any $U\in \mathcal{O}_{U}$.
By virtue of \eqref{w}, we have
\[
D_{U}V(U)=\big(D^2_{V}\eta(V(U))\big)^{-1}.
\]
Let $\bar{V}\in \mathcal{E}$ be an equilibrium state associated with $V$,
satisfying $G(\bar{V})=0$. Then there exists an equilibrium state
$\bar{U}$ associated with $U$.
From the definition above, we know that $\widetilde{H}(\bar{U})=G(\bar{V})=0$, and the source term $\widetilde{H}(U)$ in \eqref{entropeq} can be expressed by
\begin{equation}\label{tildeH}
\begin{aligned}
\widetilde{H}(U)=-\widetilde{B}(\bar{U})(U-\bar{U})+\mathbf{r}(U),
\end{aligned}
\end{equation}
where 
\begin{align}
\mathbf{r}(U):=   \widetilde{H}(U)-\widetilde{H}(\bar{U})-D_{U}\widetilde{H}(\bar{U})(U-\bar{U}),\label{rtildeU}
\end{align}
which satisfies $\mathbf{r}(U)\in \mathcal{M}^{\perp}$ for all $U\in \mathcal{O}_{U}$.

The entropy formulation provides a nonlinear symmetrization that is used to prevent loss of derivatives in the high-frequency analysis. However, when applied at low frequencies, the fully nonlinear entropy symmetrization tends to obscure structural features of the underlying conservation laws. Indeed, the nonlinear terms in the original system naturally appear in divergence form, which provides greater flexibility when using product laws.

% As a consequence, the nonlinear dynamics at low frequencies are asymptotically governed by the diffusive behavior of the linearized system.

Therefore, in order to maintain the natural divergence structure of the nonlinear terms, we symmetrize only the linearized system by freezing the entropy symmetrizer at the equilibrium state, while keeping the nonlinear terms in their original divergence form. Precisely, as in \cite{BHN,SXZ}, we can find a function $Z=Z(V)$ depending on $V$ smoothly and satisfying $Z(\bar{V})=0$, such that the system \eqref{entropeq} can be rewritten as the normal symmetric dissipative system
\begin{equation}\label{eq0}
\left\{
    \begin{aligned}
&\partial_{t}Z+\sum_{k=1}^{d}A^{k}\partial_{x_{k}}Z+\frac{BZ}{b(t)}=\frac{\mathbf{Q}(Z)}{b(t)}+\sum_{k=1}^{d}\partial_{x_{k}}\mathbf{N}^{k}(Z),\\
&Z(0,x)=Z_0(x).
\end{aligned}
\right.
\end{equation}
Here, $A^{k}$ ($k=1,\dots,d$) are symmetric constant matrices, and $B$ is the constant symmetric matrix given by
\begin{align}\label{partdissip}
B=\left(
\begin{matrix}
    0 & 0\\
    0 & L_2
\end{matrix}
\right)
\quad \text{and} \quad
    L_2 z \cdot z \geqslant \kappa |z|^2,\quad z\in \R^{n_2},
\end{align}
for some constant $\kappa>0$. In addition, the nonlinear terms $\mathbf{Q}$ and $\mathbf{N}^k$ are given by the \emph{quadratic} terms
\begin{equation}\label{R0Ri}
\left\{
    \begin{aligned}
    &\mathbf{Q}(Z):=  \mathbb{A}^0\Big(G(V)-G(\bar{V})-D_{V}G(\bar{V})(V-\bar{V})\Big),\\
    &\mathbf{N}^{i}(Z):=  \mathbb{A}^i \Big( F^i(V)-F^i(\bar{V})-D_{V}F^i(\bar{V})(V-\bar{V})\Big),\quad i=1,\dots,d,
\end{aligned}
\right.
\end{equation}
where $\mathbb{A}^i$ ($i=0,\dots,d$) are some constant matrices. The derivation of \eqref{eq0}--\eqref{R0Ri} can be found in \ref{AppendixNF}.

For the general system, additional structural assumptions are required
to guarantee sufficient decay of global solutions. In this regard, we
introduce the so-called Shizuta--Kawashima {\rm [SK]} condition.

\begin{defi}[{\rm (SK) condition}]
System \eqref{eq0} is said to satisfy the {\rm (SK)} condition if,
for every $\omega\in\mathbb S^{d-1}$, with $A_\omega:=\sum_{k=1}^d \omega_k A^k$,
the conditions
\[
B\phi=0
\quad\text{and}\quad
(\lambda {\rm Id} - A_\omega)\phi=0
\ \text{for some }\lambda\in\R
\]
imply $\phi=0_{\R^n}$.
\end{defi}

The {\rm (SK)} condition is closely related to the classical Kalman rank
condition in control theory, and characterizes how the dissipation induced by
the source term propagates through the hyperbolic coupling. 
We recall the following equivalence statement from \cite{BZ}.

\begin{prop}
\label{Propositionkalman}
Let $A^1,\dots,A^d$ and $B$ be real $n\times n$ matrices, and define $A_\omega:=\sum_{k=1}^d \omega_k A^k$ for $\omega\in\SSS^{d-1}$.
Then, for every $\omega\in\SSS^{d-1}$, the following assertions are equivalent:
\begin{enumerate}
\item The {\rm (SK)} condition holds;

\item The Kalman rank condition holds:
\[
\operatorname{rank}\,(B,\, BA_\omega,\, \dots,\, BA_\omega^{n-1}) = n.
\]

\item For any $\varepsilon_0,\dots,\varepsilon_{n-1}>0$,
the mapping
\[
y\longmapsto
|y|_{K}:=\left(\sum_{i=0}^{n-1}\varepsilon_i\,\big|B A_\omega^{\,i} y\big|^2\right)^{\frac12}
\]
defines a norm on $\R^n$. In particular, there exists a constant $C>1$, depending on $B$, $A^i_\omega$ and $\varepsilon_i$, such that $\frac{1}{C}|y|\leq |y|_{K}\leq C|y|$.
\end{enumerate}
\end{prop}

Based on the partially dissipative nature of the system \eqref{partdissip}, we define
\[
Z=\begin{pmatrix}
Z_1\\Z_2
\end{pmatrix},\quad
BZ=\begin{pmatrix}
0\\L_2Z_2
\end{pmatrix},\quad
A^k=\begin{pmatrix}
A^k_{1,1}&A^k_{1,2}\\A^k_{2,1}&A^k_{2,2}
\end{pmatrix},
\]
\[
\mathbf{Q}(Z)=\begin{pmatrix}
0 \\ Q(Z)
\end{pmatrix}
\quad{\rm and}\quad
\mathbf{N}^k(Z)=
\begin{pmatrix}
N_1^k(Z) \\ N_2^k(Z)
\end{pmatrix}.
\]
The Cauchy problem \eqref{eq0} takes the following form:
\begin{equation}\label{sys}
\left\{
\begin{aligned}
&\partial_t Z_1+\sum_{k=1}^d\Big(A^k_{1,1}\partial_k Z_1+A^k_{1,2}\partial_k Z_2\Big)
=\sum_{k=1}^d\partial_k N^k_1(Z), \\
&\partial_t Z_2+\sum_{k=1}^d\Big(A^k_{2,1}\partial_k Z_1+A^k_{2,2}\partial_k Z_2\Big)
+\frac{L_2 Z_2}{b(t)}
=\sum_{k=1}^d\partial_k N^k_2(Z) +\frac{Q(Z)}{b(t)},\\
&(Z_1,Z_2)(0)=(Z_{1,0},Z_{2,0}).
\end{aligned}
\right.
\end{equation}
In addition to the entropy and (SK) condition, to carry out our computations, we impose the following structural assumptions for all $Z_1$ sufficiently close to $0$:
\begin{align}
A_{1,1}^k=0,\quad Q(Z_1,0)=0\quad \text{and}\quad  N_1^k(Z_1,0)=0,\qquad k=1,\dots,d.
\label{conditionAVbar}
\end{align}

%Z\mapsto A^k_{1,1}(\Bar{V}+Z)\quad{\rm is \ linear\ with\ respect\ to\ } Z_2$
% \item[(4)] %$Q(Z_1,0)=0$ for $Z_1$ close to 0, and thus $Q(Z_1,Z_2)$ is at least linear with respect to $Z_2$.
\subsection{Main results}

We begin by stating our global-in-time well-posedness result.  For $t>0$, define the energy norm
\begin{equation}\label{X}
\begin{aligned}
\|Z\|_{\mathcal{E}_t}:={}&
\|Z\|_{\widetilde{L}_t^\infty(\dot{\B}_{2,1}^{\frac d2})}
+\|b(\tau)Z\|_{\widetilde{L}_t^\infty(\dot{\B}_{2,1}^{\frac d2+1})}\\
&
+\|b(\tau)^{\frac12}Z\|_{\widetilde{L}_t^2(\dot{\B}_{2,1}^{\frac d2+1})}+\|b(\tau)Z_1\|^\ell_{L_t^1(\dot{\B}_{2,1}^{\frac d2+2})}+\|Z\|^h_{L_t^1(\dot{\B}_{2,1}^{\frac d2+1})}\\
&+\|Z_2\|_{L_t^1(\dot{\B}_{2,1}^{\frac d2+1})}+\|b(\tau)^{-\frac12}Z_2\|_{\widetilde{L}_t^2(\dot{\B}_{2,1}^{\frac d2})}+\|\partial_t Z\|_{L_t^1(\dot{\B}_{2,1}^{\frac d2})}+\|b(\tau)^{-1}W\|_{L_t^1(\dot{\B}_{2,1}^{\frac d2})}.
%+\|b(\tau)^{-\frac12}Z_2\|^\ell_{\widetilde{L}_t^2(\dot{\B}_{2,1}^{\frac d2})}.
\end{aligned}
\end{equation}

\begin{thm}\label{thm0}
Let $d\geq 1$. Assume that \eqref{partdissip}, \eqref{conditionAVbar}, and the
{\rm (SK)} condition hold.
Moreover, suppose that the $\sC^\infty$ time-dependent coefficient $b(t)$ satisfies one of the following assumptions:
\begin{itemize}
\item \textbf{Case 1 (Constant-damping):} $b(t)\equiv b(0)>0$;
\item \textbf{Case 2 (Over-damping):} $b(t)>0$ and $b'(t)< 0$;
\item \textbf{Case 3 (Under-damping):} $b(t)>0$, $b'(t)>0$, and $\limsup\limits_{t\to\infty} b'(t)\leq c^*$
for some sufficiently small $c^*>0$ depending only on $d$, $\bar{V}$ and the matrices $A^k$ and $B$.
\end{itemize}
Then, there exists a constant $\beta>0$ such that for any initial data
$Z_0\in \dot{\B}_{2,1}^{\frac{d}{2}}\cap \dot{\B}_{2,1}^{\frac{d}{2}+1}$ satisfying
\[
\|Z_0\|_{\dot{\B}_{2,1}^{\frac{d}{2}}\cap \dot{\B}_{2,1}^{\frac{d}{2}+1}}
\leq \beta,
\]
the system \eqref{eq0} admits a unique global solution $Z$ such that
\begin{align*}
Z 
&\in \mathcal{C}_b(\mathbb{R}_+;\dot{\B}_{2,1}^{\frac{d}{2}}), 
& b(t)Z 
&\in \mathcal{C}_b(\mathbb{R}_+;\dot{\B}_{2,1}^{\frac{d}{2}+1}), 
& \sqrt{b(t)}\, Z 
&\in L^2(\mathbb{R}_+;\dot{\B}_{2,1}^{\frac{d}{2}+1}),\\
b(t)Z_1^\ell 
&\in L^1(\mathbb{R}_+;\dot{\B}_{2,1}^{\frac{d}{2}+2}), 
& Z_2 
&\in L^1(\mathbb{R}_+;\dot{\B}_{2,1}^{\frac{d}{2}+1}), 
& Z^h 
&\in L^1(\mathbb{R}_+;\dot{\B}_{2,1}^{\frac{d}{2}+1}),\\
\frac{Z_2}{\sqrt{b(t)}} 
&\in L^2(\mathbb{R}_+;\dot{\B}_{2,1}^{\frac{d}{2}}), 
& \partial_t Z 
&\in L^1(\mathbb{R}_+;\dot{\B}_{2,1}^{\frac{d}{2}}), 
& \frac{W}{b(t)} 
&\in L^1(\mathbb{R}_+;\dot{\B}_{2,1}^{\frac{d}{2}})
\end{align*}
where the damped mode $W$ is defined in \eqref{dampdef}.
Moreover, there exists a constant $C>0$ such that, for all $t\ge0$,
\begin{align}
\|Z\|_{\mathcal{E}_t}\le C \|Z_0\|_{\dot{\B}_{2,1}^{\frac{d}{2}}\cap \dot{\B}_{2,1}^{\frac{d}{2}+1}}.\label{Xbound}
\end{align}
\end{thm}
% \begin{rem}
% We observe that the condition over $b$ can be concluded in $b>0$ and $\limsup_{t\rightarrow \infty}b'(t)\leqslant c^*$
% \end{rem}

\begin{rem}\normalfont
% Our result seems to be the first one dealing with general non-autonomous
% coefficients $b(t)$ for partially dissipative hyperbolic systems.
The condition imposed on $b$ will be discussed in detail in Section \ref{condb}. Here are some examples of time-dependent coefficients satisfying the conditions of Theorem \ref{thm0}:
\begin{itemize}
\item \emph{Polynomial damping:} Let
\[
b(t)=\frac{(1+t)^\alpha}{K},\qquad K>0,\quad  \alpha\leq 1.
\]
When $\alpha=0$, our result recovers that of Crin-Barat and Danchin \cite{c2} in the same regularity setting, while allowing for more general $F^k(V)$ and $G(V)$ in \eqref{system}. If $\alpha<0$, then $b'(t)<0$, and the damping coefficient
$1/b(t)$ increases in time, which corresponds to the over-damping regime. If $0<\alpha<1$, then $b'(t)>0$, so the damping
coefficient $1/b(t)$ decreases in time. Finally, in the critical case $\alpha=1$, the condition $\limsup\limits_{t\to\infty}b'(t)\leq c^*$ is satisfied, provided
that $K\geq 1/c^*$. For related classifications, we refer to \cite{ji1,ji2} and the references therein.

\item \emph{Logarithmically critical damping:} Let
\[
b(t)=\frac{(1+t)(\log(e+t))^{-\theta}}{K},
\qquad K>0,\quad 0<\theta\leq 1.
\]
Then $b(t)$ is increasing for large time, while $b'(t)\rightarrow 0$ as $t\rightarrow \infty$. In this case, one does not need $K$ to be sufficiently large.

\item \emph{Logarithmic damping:}
Consider the weak damping coefficient described by
\[
b(t)=\frac{(\log(e+t))^\theta}{K},
\qquad K>0,\quad \theta>0.
\]
In this choice, Case 3 is well satisfied.

\item \emph{Sub-exponential damping:}
The over-damping regime also allows coefficients such as
\[
b(t)=\frac{1}{K}e^{-(1+t)^\theta},
\qquad K>0,\quad 0<\theta<1.
\]
Then $b'(t)<0$, and the damping coefficient $\frac1{b(t)}=K e^{(1+t)^\theta}$ grows sub-exponentially in time.
\end{itemize}
\end{rem}

Next, we consider a damping coefficient $b(t)$ which behaves asymptotically as $(1+t)^{\alpha}$. We establish optimal decay estimates and provide a quantitative description of how the decay rates depend on the parameter $\alpha$ and the regularity considered.

\begin{thm}\label{thm2}
Let $Z$ be the global solution to the system \eqref{eq0} with initial data $Z_0$ obtained in Theorem~\ref{thm0}.
Suppose, in addition, that there exist constants $t_b\gg 1$ and $0<c_b\leq C_b$ such that, for all $t\geq t_b$, \begin{align}\label{mu:infer} \frac{1}{C_b} \leq \frac{b(t)}{(1+t)^{\alpha}} \leq \frac{1}{c_b}. \end{align}
\iffalse\begin{align}
0<\lim_{t\rightarrow \infty} \frac{b(t)}{(1+t)^{\alpha}}<\infty,\qquad -1\leq \alpha \leq 1,\label{b}
\end{align}\fi
and
\begin{align}
Z_0\in \dot{\B}^{\sigma_1}_{2,\infty},
\qquad
-\frac{d}{2}\leq \sigma_1<\frac{d}{2}.\label{21}
\end{align}
Then, there exists a uniform constant $C>0$ such that the solution $Z$ satisfies the following decay estimates:
\begin{itemize}
\item \textbf{Case 1: $-1<\alpha\leq 1$.}
For any $\sigma_1<\sigma\leq \frac{d}{2}$, the solution $Z$ satisfies the algebraic
 decay estimate
\begin{align}
\|Z(t)\|_{\dot{\B}^{\sigma}_{2,1}}
&\leq C(1+t)^{-\frac{1+\alpha}{2}(\sigma-\sigma_1)}.\label{decay1}
% Furthermore, we have the following improved stability properties: For all $-\frac{d}{2}<\sigma_1<\frac{d}{2}$ when $0\leq \alpha<1$ and  $-\frac{d}{2}<\sigma_1<\frac{d}{2}-\frac{2\alpha}{1+\alpha}$ when $-1<\alpha<0$, it holds for $\sigma_1<\sigma\leq \frac{d}{2}-1$ that
\end{align}
When $\sigma_1<\frac{d}{2}-1$, for any $\sigma_1<\sigma'\leq \frac{d}{2}-1$, $Z_2$ has the improved decay estimate
\begin{align}
\|Z_2(t)\|_{\dot{\B}^{\sigma'}_{2,1}}&\leq C(1+t)^{-\frac{1+\alpha}{2}(\sigma'-\sigma_1)-\frac{1-\alpha}{2}}.\label{decayimprove1}
\end{align}

\item \textbf{Case 2: $\alpha=-1$.}
For any $\sigma_1<\sigma\leq \frac{d}{2}$, the following logarithmic decay estimate holds:
\begin{align}
\|Z(t)\|_{\dot{\B}^{\sigma}_{2,1}}
&\leq C (\log(e+t))^{-\frac{1}{2}(\sigma-\sigma_1)}.\label{decay2}
\end{align}
When $\sigma_1<\frac{d}{2}-1$, for any $\sigma_1<\sigma'\leq \frac{d}{2}-1$, it holds
\begin{align}
\|Z_2(t)\|_{\dot{\B}^{\sigma'}_{2,1}}&\leq C (\log(e+t))^{-\frac12(\sigma'-\sigma_1)-\frac{1}{2}}(1+t)^{-1}.\label{decay2improve}
\end{align}
\end{itemize}
%Here $C>0$ denotes a constant independent of time.
\end{thm}

%\subsection{The Shizuta-Kawashima and Kalman rank condition}

Finally, we show that the solutions of the partially dissipative system \eqref{eq0}
behave asymptotically like those of a linear diffusion equation with a time-dependent coefficient:
\begin{align}
\partial_t\mathcal{N}-b(t)\sum_{k=1}^d\sum_{l=1}^dA^k_{1,2}L_2^{-1}A^l_{2,1}\partial_k\partial_l \mathcal{N}=0,\quad \mathcal{N}|_{t=0}=Z_{1,0}. \label{ZL}
\end{align}

\begin{thm}\label{thm3}
Let $Z$ be the global solution to the system \eqref{eq0} with initial data $Z_0$ obtained in Theorem~\ref{thm0}. Suppose additionally that \eqref{21} and \eqref{mu:infer} hold with $-1\leq \alpha< 1$ and that the structural assumptions
\begin{align}\label{structural2}
N_{1}^k(Z)=0,\quad N_2^k(Z_1,0)=0,\quad k=1,\dots,d,\quad Q(Z)=0
\end{align}
hold. Then, for all $\sigma_1<\sigma\leq \frac{d}{2}-1$, we have
\begin{align}
\|(Z_{1}-\mathcal{N})(t)\|_{\dot{\B}^{\sigma}_{2,1}}
&\leq C
\begin{cases}
(1+t)^{-\frac{1+\alpha}{2}(\sigma-\sigma_1)
-\frac{1-\alpha}{2}},
& 0<\alpha <1,\\[0.5em]
(1+t)^{-\frac{1+\alpha}{2}(\sigma-\sigma_1)-\frac{1+\alpha}{2}},
& -1<\alpha \leq 0,\\[0.5em]
(\log(e+t))^{-\frac12(\sigma-\sigma_1+1)},
& \alpha=-1.
\end{cases}
\label{decayimprove2}
\end{align}
where the profile $\mathcal{N}$ is the solution to \eqref{ZL}. 
\end{thm}

\begin{rem}\normalfont
The decay estimates for the solution $Z$ in
\eqref{decay1} and \eqref{decay2} are optimal under the assumption
\eqref{21}, as they agree with the decay rates of the diffusion profile
$\mathcal N$. Indeed, Theorem~\ref{thm3} shows that the conservative
component of the solution to the non-autonomous partially dissipative
hyperbolic system \eqref{system} is asymptotically equivalent to the
diffusion profile $\mathcal{N}$, the solution to \eqref{ZL}.
Moreover, the improved decay of the dissipative component $Z_2$ may be understood from the asymptotic of Darcy's law:
\[
Z_2\simeq 
-b(t)L_2^{-1}\sum_{k=1}^d A_{2,1}^k\partial_k Z_1\quad \text{as}\quad t\rightarrow \infty.
\]
\end{rem}

\begin{rem}\normalfont
Theorems \ref{thm0}, \ref{thm2} and \ref{thm3} can be applied to the compressible Euler system with time-dependent damping \eqref{eq:euler}. 
In particular, when $\sigma_1=-\frac{d}{2}$, the embedding
$L^1\hookrightarrow \dot{\B}^{-\frac{d}{2}}_{2,\infty}$ implies that the decay rates we obtained for $\rho-\bar{\rho}$ and $u$ are consistent with those obtained in \cite{ji1,ji2}, while the additional smallness assumption in $L^1$ is no longer required. 
Moreover, as far as we know, the stability of the diffusion profile had previously been established only in one space dimension; our result provides its first higher-dimensional counterpart. 
The reader can refer to Section~\ref{sect:7} for more details.
\end{rem}

\begin{rem}\normalfont
The decay rates that we obtained in Theorem~\ref{thm3} follow from the divergence structure
of the error equation \eqref{errorZ}, which contains one additional spatial
derivative. Since the diffusion multiplier satisfies
\[
|\xi|e^{-c|\xi|^2\int_0^t b(\tau)\,d\tau}
\lesssim
\left(\int_0^t b(\tau)\,d\tau\right)^{-\frac12},
\]
this yields the additional decay $(1+t)^{-\frac{1+\alpha}{2}}$ for
$-1<\alpha\leq0$, and $(\log(e+t))^{-1/2}$ for $\alpha=-1$.
When $0<\alpha\leq1$, however, the source term in \eqref{errorZ}
contains $W$, whereas the weighted estimate \eqref{weight:ee} controls $W/b(t)$. Estimating the original error therefore costs one factor $b(t)$, and the gain becomes $(1+t)^{-\frac{1-\alpha}{2}}$. In particular, no strict
decay improvement occurs when $\alpha=1$.
\end{rem}
%\textcolor{blue}{T. So we cannot say that $\mathcal{N}$ is a suitable diffusion profile if there are no gain of decay? (For $\alpha=1$)} {\red{Yes, so we just consider $\alpha<1$?}}

\begin{rem}\normalfont
The structural conditions \eqref{conditionAVbar} and \eqref{structural2} are technical assumptions that ensure the relevant nonlinear terms are not quadratic in the slow component $Z_1$. With suitable adjustments to the regularity assumptions and decay rates, it is possible to relax these technical restrictions; see \cite{c2,SXZ} for the case of constant relaxation coefficients.
\end{rem}

\begin{rem}\normalfont
In the present setting, only the dissipation coefficient depends on time, whereas the hyperbolic matrices and the dissipative subspace remain fixed. The classical Shizuta–Kawashima condition therefore provides a natural and convenient structural framework. It would be interesting to extend the argument to systems
with time-dependent hyperbolic matrices, %A natural extension of the present analysis would be to allow the hyperbolic matrices themselves to vary in time, 
replacing the classical SK condition with an appropriate non-autonomous analogue based on a differential rank condition; see \cite[Theorem~1.18 and Proposition~1.19]{Coron}.
\end{rem}

\subsection{Strategy of proof}

% \subsubsection{Difficulties.}

\subsubsection{An inspiration from spectral analysis}

To highlight the impact of the time-dependent damping coefficient $b(t)$ on the dynamics of general systems, we perform a spectral analysis on a simplified model, namely the linearization of the compressible Euler equations around the equilibrium state $(\bar{\rho},0)$ with $\bar{\rho}>0$. For the sake of simplicity, let $\bar{\rho}=P'(\bar{\rho})=1$ and define $n=\rho-1$. Then, the linearized compressible Euler equations with time-depending damping read:
\begin{equation}\label{euler-lin}
\left\{
    \begin{aligned}
    &\partial_{t}n+\div m=0,\\
    &\partial_{t}m+
      \nabla n+\frac{m}{b(t)}=0.
    \end{aligned}
\right.
\end{equation}
Using the Hodge decomposition, we introduce the compressible component
\[
\mathfrak m:=\Lambda^{-1}\div m,
\qquad
\omega:=\Lambda^{-1}\nabla\times m,
\]
where
\[
\Lambda^{\sigma}:=\mathcal{F}^{-1}\big(|\xi|^{\sigma}\mathcal{F}(\cdot)\big).
\]
The linearized system then reduces to
\begin{equation}\label{lin-hodge}
\partial_{t}
\binom{n}{\mathfrak m}
=
\mathbb{A}(t)
\binom{n}{\mathfrak m},
\qquad
\mathbb{A}(t):=
\begin{pmatrix}
0 & -\Lambda\\
\Lambda & -\dfrac{1}{b(t)}
\end{pmatrix},
\qquad
\partial_{t}\omega+\frac{\omega}{b(t)}=0.
\end{equation}
For each fixed $t>0$, the Fourier symbol $\widehat{\mathbb{A}}(\xi,t)$ admits
two eigenvalues
\begin{equation}\label{eigs}
\lambda_{\pm}(\xi,t)
=
-\frac{1}{2b(t)}
\pm
\frac{1}{2}\sqrt{\frac{1}{b^{2}(t)}-4|\xi|^{2}}.
\end{equation}
As a consequence,
\begin{itemize}
\item in the low-frequency regime $|\xi|\ll \frac{1}{b(t)}$, the eigenvalues
are real and satisfy
\[
\lambda_{+}(\xi,t)\sim -b(t)|\xi|^{2},
\qquad
\lambda_{-}(\xi,t)\sim -\frac{1}{b(t)};
\]
\item in the high-frequency regime $|\xi|\gg \frac{1}{b(t)}$, the eigenvalues
form a complex conjugate pair and satisfy
\[
\lambda_{\pm}(\xi,t)\sim -\frac{1}{2b(t)}\pm {\rm i}\,|\xi|.
\]
\end{itemize}
This spectral analysis suggests that one should choose a frequency threshold
$J_{t}\sim \log_{2}\frac{1}{b(t)}$, which allows us to decompose the frequency
space into the low-frequency region $|\xi|\lesssim \frac{1}{b(t)}$ and the
high-frequency region $|\xi|\gtrsim \frac{1}{b(t)}$, thereby capturing the
distinct qualitative behaviors of solutions to the system \eqref{euler-lin}.

Under the Shizuta-Kawashima (SK) condition, the dissipative coupling structure ensures that the qualitative spectral behavior of the general system is expected to be similar to that of the linearized compressible Euler equations with time-dependent damping, providing a reliable guide for understanding the dynamics of general systems.

\subsubsection{Damped mode}
The above spectral analysis reveals that, at low frequencies, two distinct regimes emerge: a parabolic behavior and a hyperbolic one. In order to capture the damped component (corresponding to the eigenvalue $\lambda_-$), we introduce the following damped mode:
\begin{align}
    \label{dampdef}
    W\triangleq Z_2+L_2^{-1}\left(b(t)\sum_{k=1}^d(A_{2,1}^k\partial_k Z_1+A_{2,2}^k\partial_k Z_2)-b(t)\sum_{k=1}^d\partial_k N^k_2(Z) -Q(Z)\right).
\end{align}
Then $\partial_t Z_2=-\frac{L_2 W}{b(t)}$. Moreover, $W$ satisfies
 \begin{align}
    \partial_t W+\frac{L_2 W}{b(t)}=l_1+f_1,
    \label{weq}
\end{align}
where the linear remainder reads
\begin{equation}\label{l1}
 \begin{aligned}
 l_1&=L_2^{-1}\left(b(t) \sum_{k=1}^d (A_{2,1}^k\partial_k \partial_t Z_1-b(t)^{-1}A_{2,2}^k\partial_k L_2 W)+b'(t) \sum_{k=1}^d (A_{2,1}^k\partial_k Z_1+A_{2,2}^k\partial_k Z_2)\right),
 \end{aligned}
  \end{equation}
 and the nonlinear term
 \begin{equation}
\begin{aligned}\label{h1h2}
 f_1&=-L_2^{-1}\left(\partial_t \big( b(t)\sum_{k=1}^d\partial_k N^k_2(Z)\big)+\partial_tQ(Z)\right).
    % f_2&=b(t) L_2^{-1}\sum_{k=1}^d \partial_t(\widetilde{A}_{2,1}^k(Z)\partial_k Z_1+\widetilde{A}_{2,2}^k(Z)\partial_k Z_2),\\
    % h_2&=b'(t) L_2^{-1}\sum_{k=1}^d (\widetilde{A}_{2,1}^k(Z)\partial_k Z_1+\widetilde{A}_{2,2}^k(Z)\partial_k Z_2),\\
    % h_3&=\partial_t Q(Z),
 \end{aligned}
 \end{equation}
Due to the special structure of the equation satisfied by $W$, it enjoys faster decay rates than both components $Z_1$ and $Z_2$.

\vspace{2mm}

\noindent
\subsubsection{Leading diffusion profile}

In order to describe the time-dependent diffusion profile, it is convenient to work directly with the evolution equations for $Z_1$, formulated in terms of the damped mode $W$. %Let us denote for all $k\in\{1,...,d\}$ and $p,m\in\{1,2\}$
 % \begin{align*}
      %A_{p,m}^k=A_{p,m}^k(\Bar{V})\quad{\rm and}\quad \widetilde{A}_{p,m}^k(Z)=A_{p,m}^k(\Bar{V}+Z)-A_{p,m}^k(\Bar{V}).
  %\end{align*}
 % In particular, by the structural condition before, we may write $\widetilde{A}_{1,1}^k(Z)$ as $\widetilde{A}_{1,1}^k(Z_2)$.
Substituting 
\begin{align}\label{Z2}
Z_2=W-L_2^{-1}\left(b(t)\sum_{k=1}^d(A_{2,1}^k\partial_k Z_1+A_{2,2}^k\partial_k Z_2)-b(t)\sum_{k=1}^d\partial_k N^k_2(Z) -Q(Z)\right)
\end{align}
into $\eqref{sys}_1$, the equation of $Z_1$ rewrites as
\begin{align}
    \label{eqz1}
    \partial_tZ_1-b(t)\sum_{k=1}^d\sum_{l=1}^dA^k_{1,2}L_2^{-1}A^l_{2,1}\partial_k\partial_l Z_1=l_2+f_2,
\end{align}
with the linear remainder
\begin{align}\label{l2}
l_2=-\sum_{k=1}^dA^k_{1,2}\partial_k W+b(t)\sum_{k=1}^d\sum_{l=1}^dA^k_{1,2}L^{-1}_2A^l_{2,2}\partial_k\partial_l Z_2
\end{align}
and the nonlinear term
\begin{equation}\label{f1f5}
\begin{aligned}
    f_2= \sum_{k=1}^d\partial_k N^k_1(Z)-\left(b(t)\sum_{k=1}^d\sum_{l=1}^d A^k_{1,2}L_2^{-1}\partial_k\partial_l N^l_2(Z)+\sum_{k=1}^d A^k_{1,2}L_2^{-1} \partial_k Q(Z)\right).
%&f_1=b(t)\sum_{k=1}^d\sum_{l=1}^d\left(A^k_{1,2}(V)\partial_k(L^{-1}_2A^l_{2,2}(V)\partial_l Z_2)-A^k_{1,2}(\bar{V})L^{-1}_2A^l_{2,2}(\bar{V})\partial_k\partial_l Z_2\right),\\
   % &f_2=-\sum_{k=1}^d \widetilde{A}^k_{1,2}(Z)\partial_k W,\\
   % &f_3=b(t)\sum_{k=1}^d\sum_{l=1}^dA^k_{1,2}(V)\partial_k(L^{-1}_2\widetilde{A}^l_{2,1}(Z))\partial_l Z_1,\\
    %&f_4=b(t)\sum_{k=1}^d\sum_{l=1}^d\widetilde{A}^k_{1,2}(Z)L^{-1}_2A^l_{2,1}\partial_k\partial_l Z_1,\\
   % &f_5=-\sum_{k=1}^d(\widetilde{A}^k_{1,1}(Z)\partial_k Z_1-A_{1,2}^k\partial_k Q(Z)).
\end{aligned}
\end{equation}
%Here we used 
%\begin{align}
%Z_2=\mathbf{L}\Big(W-b(t)L_2^{-1}\sum_{k=1}^d A_{2,1}^k\partial_k Z_1+b(t)\sum_{k=1}^d\partial_k N^k_2(Z) +Q(Z) \Big),\label{Z2}
%\end{align}
%where the operator $\mathbf{L}$ is defined by the Fourier multiplier with the symbol  
%\begin{align}\label{L2}
%\Big(\mathbf{I}_n+b(t) L_2^{-1} \sum_{k=1}^dA_{2,2}^k {\rm i} \xi_k \Big)^{-1},
%\end{align}
%which is bounded in $L^p$ ($1\leq p\leq \infty$) restricted in the annulus $2^j\mathcal{C}$ with $2^j<<\frac{1}{b(t)}$. 
Note that a second-order operator appears in the equation for $Z_1$. In order to ensure that $Z_1$ exhibits the parabolic behavior predicted by the spectral analysis, we will make use of the facts:
\begin{align}
    \forall k\in\{1,...,d\},\quad A^k_{1,1}=0\quad{\rm and}\quad-\sum_{k=1}^d\sum_{l=1}^d  A^k_{1,2}L_2^{-1}A^l_{2,1}\partial_k\partial_l\quad{\rm is\ strongly\ elliptic.}
    \label{ellip}
\end{align}
In fact, if the first condition of \eqref{ellip} is satisfied, then the strong ellipticity property in \eqref{ellip} is equivalent to the (SK) condition (cf. \cite{BHN,c3,Yong}).

With this reformulated system in terms of $Z_1$ and $W$, we recover a structure consistent with the spectral analysis at low frequencies. In the high-frequency regime, we perform hypocoercive estimates to make the dissipative effects on $Z_1$ explicit.

The strategy is then to derive a priori estimates based on the linear structure and to close the argument via a standard bootstrap procedure.

{\subsubsection{Conditions on the damping coefficient \texorpdfstring{$b$}{b}}
\label{condb}

We explain how the additional growth condition on $b$ arises from the
nonlinear estimates. The linear analysis accommodates a broader class
of coefficients under the monotonicity assumptions described in
Section~\ref{sec:linear}. In contrast, closing the nonlinear
high-frequency estimates requires quantitative control of the growth
of $b$.

At high frequencies, the hypocoercive energy estimate takes the
schematic form
\[
\frac{\dd}{\dd t}X+\frac{\kappa_{\rm hf}}{b(t)}X
\leq C X^2,
\]
where $X$ is a suitable high-frequency energy norm and
$\kappa_{\rm hf}>0$ depends on the matrices $A^k$ and $B$ and on the
choice of the hypocoercive Lyapunov functional. To absorb the nonlinear
term into the dissipation, it suffices to keep $b(t)X(t)$ sufficiently
small. Introducing $Y(t):=b(t)X(t)$, we obtain
\[
\frac{\dd}{\dd t}Y
+\frac{\kappa_{\rm hf}-b'(t)}{b(t)}Y
\leq \frac{C}{b(t)}Y^2.
\]
This naturally leads to the assumption $
\limsup_{t\to\infty}b'(t)<\kappa_{\rm hf}$. Indeed, under this condition, there exist $T_0\geq0$ and $\delta>0$
such that $\kappa_{\rm hf}-b'(t)\geq\delta$ for all $t\geq T_0$. 
Consequently, it holds that
\[
\frac{\dd}{\dd t}Y+\frac{\delta}{b(t)}Y
\leq\frac{C}{b(t)}Y^2,
\qquad t\geq T_0.
\]
If $Y(T_0)<\delta/(2C)$, a standard bootstrap argument yields
\[
Y(t)\leq Y(T_0)
\qquad\text{for all }t\geq T_0,
\]
and hence
\[
X(t)\leq \frac{b(T_0)X(T_0)}{b(t)}.
\]
The required smallness at time $T_0$ follows from the finite-time
estimates on $[0,T_0]$, provided the initial perturbation is sufficiently
small. The smallness threshold may therefore depend on the behavior
of $b$ on this initial interval.

The same condition can be understood by comparing the linear decay
factor
\[
\exp\left(-\kappa_{\rm hf}\int_0^t\frac{\dd\tau}{b(\tau)}\right)
\]
with the weight $1/b(t)$. Indeed, we have
\[
\frac{\dd}{\dd t}
\left[
b(t)\exp\left(
-\kappa_{\rm hf}\int_0^t\frac{\dd\tau}{b(\tau)}
\right)
\right]
=
\bigl(b'(t)-\kappa_{\rm hf}\bigr)
\exp\left(
-\kappa_{\rm hf}\int_0^t\frac{\dd\tau}{b(\tau)}
\right).
\]
Thus, the strict gap between $b'$ and $\kappa_{\rm hf}$ provides the
margin needed to propagate the weighted smallness in the nonlinear
problem.

For the Euler system with damping coefficient
$\mu(1+t)^{-\lambda}$, our convention gives
\[
\frac{1}{b(t)}=\frac{\mu}{(1+t)^\lambda},
\qquad
b(t)=\frac{(1+t)^\lambda}{\mu},
\qquad
b'(t)=\frac{\lambda}{\mu}(1+t)^{\lambda-1}.
\]
If $\lambda<1$, then $b'(t)\to0$, so the above condition is satisfied
for every $\mu>0$. In the critical case $\lambda=1$, it becomes
\[
\frac{1}{\mu}<\kappa_{\rm hf},
\qquad\text{or equivalently}\qquad
\mu>\frac{1}{\kappa_{\rm hf}}.
\]
The abstract hypocoercive argument therefore yields a sufficient
lower bound on the damping strength in the critical case. This bound
depends on the system and the chosen energy functional and does not need to
be optimal.}
\section{Littlewood-Paley decomposition and functional spaces}\label{sec:LP}

% Before proving the main results, we introduce a few notations.

We fix a homogeneous Littlewood--Paley decomposition
$(\dot{\Delta}_j)_{j\in \mathbb{Z}}$ by setting
\begin{align*}
    \dot{\Delta}_j \triangleq \varphi(2^{-j}D)
    \quad {\rm with}\quad
    \varphi(\xi)=\chi(\xi/2)-\chi(\xi),
\end{align*}
where $\chi$ stands for a smooth function taking values in $[0,1]$, supported
in the open ball $B(0,4/3)$ and such that $\chi\equiv 1$ on the closed ball
$\overline{B}(0,3/4)$. Furthermore, we set
\begin{align*}
    \dot{S}_j \triangleq \chi(2^{-j}D)
    \quad {\rm for\ all}\quad j\in \mathbb{Z},
\end{align*}
and define $\mathcal{S}'_h$ to be the set of tempered distributions $z$ such that
\begin{align*}
    \lim_{j\to -\infty}\|\dot{S}_j z\|_{L^\infty}=0.
\end{align*}
Following \cite{Bahouri2011}, we introduce the homogeneous Besov semi-norms:
\begin{align*}
    \|z\|_{\dot{\B}_{p,r}^s}
    \triangleq
    \big\|2^{js}\|\dot{\Delta}_j z\|_{L^p(\mathbb{R}^d)}\big\|_{\ell^r(\mathbb{Z})},
\end{align*}
and define the homogeneous Besov spaces $\dot{\B}_{p,r}^s$ (for any $s\in\mathbb{R}$
and $(p,r)\in[1,\infty]^2$) as the subset of those $z$ in $\mathcal{S}_h'$ such that
$\|z\|_{\dot{\B}_{p,r}^s}$ is finite.

% Similarly, for the nonhomogeneous Besov spaces, we first introduce the nonhomogeneous dyadic blocks $\Delta_j$:
% \begin{align*}
%     \Delta_j\triangleq \dot{\Delta}_j\quad{\rm if}\quad j\geqslant 0,\qquad
%     \Delta_{-1}=\dot{S}_0,\qquad
%     \Delta_j=0\quad{\rm if}\quad j\leqslant -2,
% \end{align*}
% and the nonhomogeneous Besov semi-norms:
% \begin{align*}
%     \|z\|_{{\B}_{p,r}^s}
%     \triangleq
%     \big\|2^{js}\|{\Delta}_j z\|_{L^p(\mathbb{R}^d)}\big\|_{\ell^r(\mathbb{Z})}.
% \end{align*}
% Then the definition of nonhomogeneous Besov spaces follows. 

%\subsection{Functional framework}

According to the calculations above, it is most suitable to separate low and high frequencies in a time-dependent manner.
Precisely, for $t\geq 0$, we set the thresholds
\begin{align}
    &J_{t}:=\left\lfloor\log_{2}\frac{1}{b(t)}\right\rfloor-k_{0}\quad\text{and}\quad
J_0:= \left\lfloor\log_2 \frac{1}{b(0)}\right\rfloor-k_0,
    \label{J}
\end{align}
for some fixed integer $k_{0}\in \mathbb{Z}$. If $b(t)\equiv b_0$, then $J_t\equiv J_0$, and the above decomposition
reduces to the usual fixed low--high frequency splitting. In this case,
no transition time $t_j$ is needed. The quantities $t_j$ introduced
below are used only when $b$ is strictly monotone.

\begin{rem}
    As we shall see, the definition of $k_0$ depends only on the matrices $A^k$ and $B$.
    In particular, it is independent of $b(t)$. Hence we may impose conditions on $b(t)$
    depending on $k_0$ without any circular dependence.
\end{rem}

In the case $b'(t)\ne 0$, the frequency threshold $J_t$ depends on time, which introduces substantial technical difficulties compared with the constant-damping case. To quantify the interplay between time and frequency, we define the time threshold $t_j$ as follows:
\begin{equation*}
\text{if} \quad  b'(t)>0,\quad t_j=\left\{
\begin{aligned}
& b^{-1}(2^{-(j+k_0)}),\quad  j\leq J_0,\\
& 0, \qquad\qquad\qquad\ \  j> J_0,
\end{aligned}
\right.
\end{equation*}
and
\begin{equation*}
\text{if} \quad b'(t)< 0,\quad  t_j=\left\{
\begin{aligned}
& b^{-1}(2^{-(j+k_0)}), \quad  j\geq J_0+1,\\
& 0, \qquad\qquad\qquad\ \  j< J_0.
\end{aligned}
\right.
\end{equation*}
Whenever $t_j>0$, we have
\begin{equation}\label{relationtj}
2^{j}= (2^{k_0} b(t_j))^{-1}.
\end{equation}

We define the Besov semi-norms associated with a general threshold $J\in \mathbb{Z}$ by
\begin{align*}
    \|u\|_{\dot{\B}_{p,r}^s}^{\ell,J}
    :=\big\|\{2^{js}\|\dot{\Delta}_j u\|_{L^p}\}_{j\leqslant J}\big\|_{\ell^r},
    \qquad
    \|u\|_{\dot{\B}_{p,r}^s}^{h,J}
    :=\big\|\{2^{js}\|\dot{\Delta}_j u\|_{L^p}\}_{j\geqslant J+1}\big\|_{\ell^r}.
\end{align*}

Then, for fixed $j$ and $t$, we define the time intervals
\begin{align*}
    I_{j,t}^\ell:=\{0\leq\tau\leq t\mid j\leqslant J_\tau\},
    \qquad
    I_{j,t}^h:=\{0\leq\tau\leq t\mid j\geqslant J_\tau+1\}.
\end{align*}
For simplicity, we shall denote $I_{j}^\ell:=I_{j,\infty}^\ell$ and  $I_{j}^h:=I_{j,\infty}^h$. 

\begin{rem}\label{remark11}
Assume that $b$ is monotone. Then, for any $j\in\mathbb{Z}$ and $t\ge0$, whenever there exists $t_j\ge0$ such that $2^{-(j+k_0)}=b(t_j)$, 
the sets $I_{j,t}^\ell$ and $I_{j,t}^h$ are intervals. More precisely,
\begin{itemize}
\item If $b'(t)>0$, then
\[
I_{j,t}^\ell=[0,\min\{t,t_j\}],\qquad
I_{j,t}^h=
\begin{cases}
[t_j,t], & t\ge t_j,\\
\emptyset, & t<t_j;
\end{cases}
\]
\item If $b'(t)<0$, then
\[
I_{j,t}^\ell=
\begin{cases}
[t_j,t], & t\ge t_j,\\
\emptyset, & t<t_j.
\end{cases}, \qquad I_{j,t}^h=[0,\min\{t,t_j\}].
\]
\end{itemize}
\end{rem}

\begin{figure}[!ht]
    \centering
    \begin{minipage}[t]{0.48\textwidth}
        \centering
        \begin{tikzpicture}[xscale=0.85,yscale=0.85, thick]
            % axes
            \draw[-latex] (-0.2,0) -- (6.9,0);
            \draw[-latex] (0,-3.4) -- (0,3.4);

            % axis labels
            \node[left] at (0,3.2) {$j$};
            \node[below left] at (0,0) {$0$};
            \node[below] at (6.55,0) {$\infty$};

            % maximal time t
            \def\tmax{5.0}
            \pgfmathsetmacro{\Jt}{2.4-1.9*ln(1+\tmax)}

            % low-frequency shaded region, only up to t
            \fill[gray!12,opacity=0.32]
                plot[smooth,domain=0.03:\tmax,samples=120] (\x,{2.4-1.9*ln(1+\x)})
                -- (\tmax,-3.1) -- (0.03,-3.1) -- cycle;

            % threshold curve: ends at (t,J_t)
            \draw[line width=1.15pt]
                plot[smooth,domain=0:\tmax,samples=120] (\x,{2.4-1.9*ln(1+\x)});

            % dashed guides
            \draw[dashed] (\tmax,-3.1) -- (\tmax,3.1);
            \draw[dashed] (0,\Jt) -- (\tmax,\Jt);
            \fill (\tmax,\Jt) circle (1.5pt);

            % labels
            \node[below right] at (\tmax,0) {$t$};
            \node[left] at (0,\Jt) {$J_t$};

            % frequency-region labels
            \node at (2.3,-1.35) {Low frequencies};
            \node at (2.1,2.1) {High frequencies};
        \end{tikzpicture}

        \vspace{0.3em}
        \small (a) Underdamped case $b'(t)>0$
    \end{minipage}
    \hfill
    \begin{minipage}[t]{0.48\textwidth}
        \centering
        \begin{tikzpicture}[xscale=0.85,yscale=0.85, thick]
            % axes
            \draw[-latex] (-0.2,0) -- (6.9,0);
            \draw[-latex] (0,-3.4) -- (0,3.4);

            % axis labels
            \node[left] at (0,3.2) {$j$};
            \node[below left] at (0,0) {$0$};
            \node[below] at (6.55,0) {$\infty$};

            % maximal time t
            \def\tmax{5.0}
            \pgfmathsetmacro{\Jt}{-2.2+1.45*ln(1+\tmax)}

            % low-frequency shaded region, only up to t
            \fill[gray!12,opacity=0.32]
                plot[smooth,domain=0.03:\tmax,samples=120] (\x,{-2.2+1.45*ln(1+\x)})
                -- (\tmax,-3.1) -- (0.03,-3.1) -- cycle;

            % threshold curve: ends at (t,J_t)
            \draw[line width=1.15pt]
                plot[smooth,domain=0:\tmax,samples=120] (\x,{-2.2+1.45*ln(1+\x)});

            % dashed guides
            \draw[dashed] (\tmax,-3.1) -- (\tmax,3.1);
            \draw[dashed] (0,\Jt) -- (\tmax,\Jt);
            \fill (\tmax,\Jt) circle (1.5pt);

            % labels
            \node[below right] at (\tmax,0) {$t$};
            \node[left] at (0,\Jt) {$J_t$};

            % frequency-region labels
            \node at (2.2,-2.1) {Low frequencies};
            \node at (2.2,1.8) {High frequencies};
        \end{tikzpicture}

        \vspace{0.3em}
        \small (b) Overdamped case $b'(t)<0$
    \end{minipage}
    \caption{Frequency-time partition}
    \label{fig:frequency-splitting}
\end{figure}
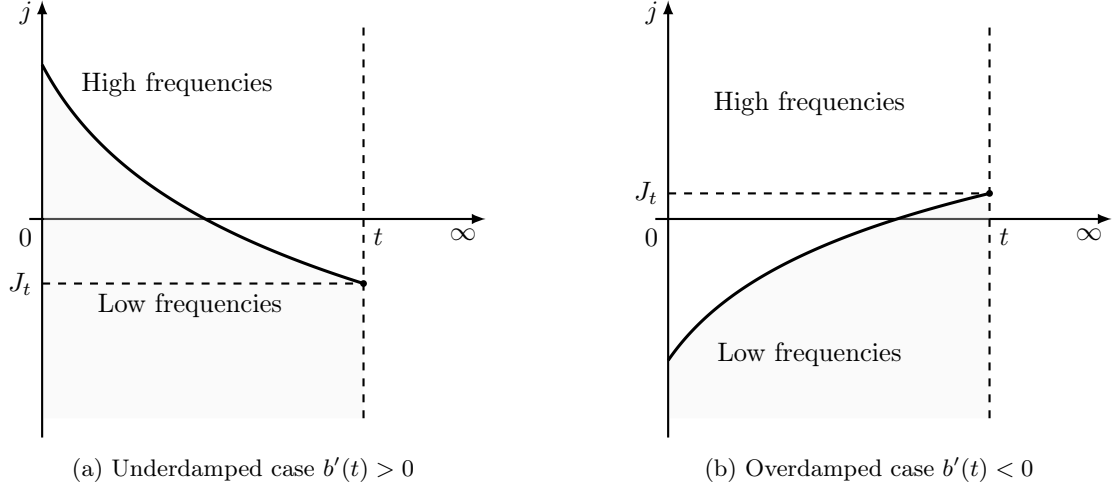

For $\varrho\geqslant1$, we adapt the Chemin-Lerner norms to more general spaces involving time-dependent coefficients:
\begin{align*}
\|u\|_{\widetilde{L}^{\varrho}_{t}(\dot{\B}^{s}_{p,1})}^{\ell}
&:=\sum_{j\in\Z}2^{js}
\Big(\int_{I_{j,t}^\ell}\|\ddj u(\tau)\|^\varrho_{L^p}\,d\tau\Big)^{\frac{1}{\varrho}},\quad \|u\|_{\widetilde{L}^{\varrho}_{t}(\dot{\B}^{s}_{p,1})}^{h}:=\sum_{j\in\Z}2^{js}
\Big(\int_{I_{j,t}^h}\|\ddj u(\tau)\|^\varrho_{L^p}\,d\tau\Big)^{\frac{1}{\varrho}}.
\end{align*}
where, for $\varrho=+\infty$, the usual convention (involving the essential supremum $\sup_{[a,b]}f(\tau)$) is adopted. When $b$ is constant, then the above definition reduces to the classical Chemin-Lerner norms
\begin{align*}
    &\:\|u\|_{\widetilde{L}^{\varrho}_{t}(\dot{\B}^{s}_{p,1})}^{\ell}=\sum_{\substack{j\in\Z\\ j\leq J_0}}2^{js}\Big(\int^t_0\|\ddj u(\tau)\|^\varrho_{L^p}d\tau\Big)^\frac{1}{\varrho},\quad \|u\|_{\widetilde{L}^{\varrho}_{t}(\dot{\B}^{s}_{p,1})}^{h}=\sum_{\substack{j\in\Z\\ j\geq J_0+1}}2^{js}\Big(\int_0^t\|\ddj u(\tau)\|^\varrho_{L^p}d\tau\Big)^\frac{1}{\varrho}.
\end{align*}
 By Fubini's theorem, we observe that, for $\varrho=1$,
\begin{align*}
 \text{For } b'(t)>0,&\quad   \|u\|_{\widetilde{L}^{1}_{t}(\dot{\B}^{s}_{p,1})}^{\ell}=\sum_{\substack{j\in\Z\\ t_j>0}}2^{js}\int_{I_{j,t}^\ell}\|\ddj u(\tau)\|_{L^p}d\tau=\int_0^t\|u(\tau)\|_{\dot{\B}^{s}_{p,1}}^{\ell}d\tau,\\
 & \quad   \|u\|_{\widetilde{L}^{1}_{t}(\dot{\B}^{s}_{p,1})}^{h}=\sum_{\substack{j\in\Z\\t_j<t}}2^{js}\int_{I_{j,t}^h}\|\ddj u(\tau)\|_{L^p}d\tau=\int_{0}^t\|u(\tau)\|_{\dot{\B}^{s}_{p,1}}^{h}d\tau,
 % \\
 % \text{for } \alpha=0,&\quad   \|u\|_{\widetilde{L}^{1}_{t}(\dot{\B}^{s}_{p,1})}^{\ell}=\sum_{j\leq J_0}2^{js}\int^t_0\|\ddj u(\tau)\|_{L^p}d\tau=\int_0^t\|u(\tau)\|_{\dot{\B}^{s}_{p,1}}^{\ell,J_0}d\tau,\\
 % & \quad   \|u\|_{\widetilde{L}^{1}_{t}(\dot{\B}^{s}_{p,1})}^{h}=\sum_{j\geq J_0+1}2^{js}\int^t_0\|\ddj u(\tau)\|_{L^p}d\tau=\int_{0}^t\|u(\tau)\|_{\dot{\B}^{s}_{p,1}}^{h,J_0}d\tau,
 \\
  \text{for } b'(t)<0,&\quad   \|u\|_{\widetilde{L}^{1}_{t}(\dot{\B}^{s}_{p,1})}^{\ell}=\sum_{\substack{j\in\Z\\ t_j<t}}2^{js}\int_{I_{j,t}^\ell}\|\ddj u(\tau)\|_{L^p}d\tau=\int_{0}^t\|u(\tau)\|_{\dot{\B}^{s}_{p,1}}^{\ell}d\tau,\\
  &\quad   \|u\|_{\widetilde{L}^{1}_{t}(\dot{\B}^{s}_{p,1})}^{h}=\sum_{\substack{j\in\Z\\ t_j>0}}2^{js}\int_{I_{j,t}^h}\|\ddj u(\tau)\|_{L^p}d\tau=\int_0^t\|u(\tau)\|_{\dot{\B}^{s}_{p,1}}^{h}d\tau.
\end{align*}

Our analysis is based on a refined frequency decomposition and a careful low/high-frequency analysis within the Littlewood–Paley framework. More precisely, the low-frequency part is estimated in the critical space $\dot{\B}^{d/2}_{2,1}$, while the high-frequency part is treated in $\dot{\B}^{d/2+1}_{2,1}$. For each dyadic block, the energy estimates must be performed separately on the time intervals $I_{j,t}^\ell$ and $I_{j,t}^h$. Moreover, once the analysis is split at the transition time $t_j$, the estimates on these two intervals have to be patched together rather than treated independently. As a consequence, the low- and high-frequency analyses are no longer decoupled (see Figures \ref{fig:frequency-splitting} and \ref{fig:time-splitting}).
\begin{figure}[!ht]
    \centering

    \begin{minipage}{0.95\textwidth}
        \centering
        \begin{tikzpicture}[xscale=0.48,yscale=0.48, thick]
            % axis
            \draw[-latex] (-1,0) -- (21,0);
            \draw (21,0) node[below] {$t$};

            % ticks
            \draw (-1,0) node {$|$};
            \draw (10,0) node {$|$};

            \draw (-1,-0.25) node[below] {$0$};
            \draw (10,-0.25) node[below] {$t_j$};

            % labels
            \draw (4.5,0.0) node[above] {Low frequencies};
            \draw (4.5,-0.75) node {$\dot{\B}^{\frac{d}{2}}_{2,1}$};

            \draw (15.5,0.0) node[above] {High frequencies};
            \draw (15.5,-0.75) node {$\dot{\B}^{\frac{d}{2}+1}_{2,1}$};
        \end{tikzpicture}

        \vspace{0.3em}
        \small (a) Underdamped case $b'(t)>0$
    \end{minipage}

    \vspace{0.8em}

    \begin{minipage}{0.95\textwidth}
        \centering
        \begin{tikzpicture}[xscale=0.48,yscale=0.48, thick]
            % axis
            \draw[-latex] (-1,0) -- (21,0);
            \draw (21,0) node[below] {$t$};

            % ticks
            \draw (-1,0) node {$|$};
            \draw (10,0) node {$|$};

            \draw (-1,-0.25) node[below] {$0$};
            \draw (10,-0.25) node[below] {$t_j$};

            % labels
            \draw (4.5,0.0) node[above] {High frequencies};
            \draw (4.5,-0.75) node {$\dot{\B}^{\frac{d}{2}+1}_{2,1}$};

            \draw (15.5,0.0) node[above] {Low frequencies};
            \draw (15.5,-0.75) node {$\dot{\B}^{\frac{d}{2}}_{2,1}$};
        \end{tikzpicture}

        \vspace{0.3em}
        \small (b) Overdamped case $b'(t)<0$
    \end{minipage}

    \caption{Frequency-dependent time splitting}
    \label{fig:time-splitting}
\end{figure}
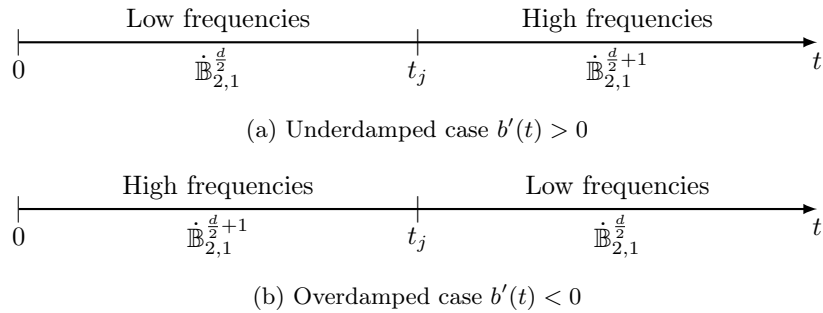

%\begin{rem}
 %   By the definitions above, we observe that the roles of increasing and decreasing $b$ are reversed.
  %  Hence, for most arguments, the proofs are essentially symmetric. In what follows, we shall always suppose that $b$ is increasing.
%\end{rem}

By the definition of the time threshold and Bernstein's inequality, we have the following lemma which will be used repeatedly.

\begin{lem}\label{lemma51}
For any function $f=f(x)$, $s\geqslant0$, and $r\in\mathbb{R}$, we have
\begin{align}
     &\|b^{r}(\tau)f\|_{\dot{\B}^{s}_{2,1}}^\ell
    \lesssim 2^{-k_0}\|b^{r-1}(\tau)f\|_{\dot{\B}^{s-1}_{2,\infty}}^{\ell}
    \lesssim 2^{-k_0}\|b^{r-1}(\tau)f\|_{\dot{\B}^{s-1}_{2,1}}^{\ell},
    \label{timelow}\\
    &\|b^{r}(\tau)f\|_{\dot{\B}^{s}_{2,1}}^h
    \lesssim 2^{k_0}\|b^{1+r}(\tau)f\|_{\dot{\B}^{s+1}_{2,1}}^{h}.
    \label{timehigh}
\end{align}
\end{lem}

\begin{proof}
We prove the first inequality in \eqref{timelow}; the proof of \eqref{timehigh} is similar.
By the definition of the threshold, for $j\leqslant J_\tau$ we have
\[
2^j \leqslant \frac{2^{-k_0}}{b(\tau)}.
\]
Hence, by Bernstein's inequality,
\[
\forall j\leqslant J_\tau,\qquad
\|b^{r}(\tau)\dot{\Delta}_j f\|_{L^2}
= b^r(\tau)\|\dot{\Delta}_j f\|_{L^2}
\leqslant 2^{-k_0}2^{-j}\, b^{r-1}(\tau)\|\dot{\Delta}_j f\|_{L^2}.
\]
Multiplying by $2^{js}$ and summing over $j\leqslant J_\tau$, we get
\[
\|b^r(\tau)f\|_{\dot{\B}^{s}_{2,1}}^\ell
\leqslant 2^{-k_0}\sum_{j\leqslant J_\tau}2^{j(s-1)}\|b^{r-1}(\tau)\dot{\Delta}_j f\|_{L^2}
\leqslant 2^{-k_0}\|b^{r-1}(\tau)f\|_{\dot{\B}^{s-1}_{2,1}}^\ell.
\]
For the bound with $\dot{\B}^{s-1}_{2,\infty}$, we observe that 
\[
\|b^r(\tau)f\|_{\dot{\B}^{s}_{2,1}}^\ell
= b^r(\tau)\sum_{j\leqslant J_\tau}2^{js}\|\dot{\Delta}_j f\|_{L^2}
\leqslant b^r(\tau)\sum_{j\leqslant J_\tau}2^{j}\|f\|_{\dot{\B}^{s-1}_{2,\infty}}
\leqslant 2^{-k_0}\|b^{r-1}(\tau)f\|_{\dot{\B}^{s-1}_{2,\infty}}^\ell,
\]
since $\sum_{j\leqslant J_\tau}2^{j}\lesssim 2^{J_\tau}$. For $j\ge J_\tau+1$, we have $2^j\geq \frac{2^{-k_0}}{b(\tau)}$. Thus $b^\gamma(\tau)\|\dot\Delta_j f\|_{L^2}
\leq 2^{k_0}2^j b^{1+\gamma}(\tau)\|\dot\Delta_j f\|_{L^2}$ and summing over $j\ge J_\tau+1$ gives \eqref{timehigh}.
\end{proof}

Similarly, we obtain the time-dependent versions:
\begin{align}
\|b^{r}(\tau)f\|_{\widetilde{L}_T^\rho(\dot{\B}^{s}_{2,1})}^\ell
\leqslant 2^{-k_0}\|b^{r-1}(\tau)f\|_{\widetilde{L}_T^\rho(\dot{\B}^{s-1}_{2,1})}^{\ell},
\qquad
\|b^{r}(\tau)f\|_{\widetilde{L}_T^\rho(\dot{\B}^{s}_{2,1})}^h
\leqslant 2^{k_0}\|b^{1+r}(\tau)f\|_{\widetilde{L}_T^\rho(\dot{\B}^{s+1}_{2,1})}^{h}.
\end{align}

Then we introduce a lemma which will be used repeatedly.

\begin{lem}\label{b'lem}
Assume that $b$ is monotone.
For any $s\in\mathbb{R}$ and any function $f\in \widetilde{L}^\infty_T(\dot\B_{p,1}^s)$, we have
\begin{align*}
\|b'(\tau) f\|^\ell_{\widetilde{L}^1_T(\dot\B_{p,1}^{s+1})}
\leqslant 2^{-k_0}\|f\|^\ell_{\widetilde{L}^\infty_T(\dot\B_{p,1}^s)},
\qquad
\left\|\left(\frac{1}{b(\tau)}\right)' f\right\|^h_{\widetilde{L}^1_T(\dot\B_{p,1}^{s-1})}
\leqslant 2^{k_0}\|f\|^h_{\widetilde{L}^\infty_T(\dot\B_{p,1}^s)}.
\end{align*}
\end{lem}

\begin{proof}
We have
\begin{equation*}
\begin{aligned}
\int_{0}^{t}\|b'(\tau) f(\tau)\|_{\dot{\B}^{s+1}_{p,1}}^{\ell}\,d\tau
&=\int_{0}^{t}|b'(\tau)|\sum_{j\leqslant J_\tau}2^{j(s+1)} \|\dot{\Delta}_j f(\tau)\|_{L^p}\,d\tau\\
&=\sum_{j\in\mathbb{Z}}2^{j(s+1)} \int_{I_{j,t}^{\ell}}|b'(\tau)|\|\dot{\Delta}_j f(\tau)\|_{L^p}\,d\tau\\
&\leqslant\sum_{j\in\mathbb{Z}}2^{j(s+1)}
\sup_{\tau\in I_{j,t}^{\ell}}  \|\dot{\Delta}_j f(\tau)\|_{L^p}
\int_{I_{j,t}^\ell}|b'(\tau)|\,d\tau.
\end{aligned}
\end{equation*}
It remains to estimate the integral.
Since $b$ is monotone and, by the definition of $I_{j,t}^{\ell}$, for the relevant intervals we have
\[
\forall \tau\in I_{j,t}^\ell,\qquad 0<b(\tau)\leqslant b(t_j),
\]
it follows that
\[
\int_{I_{j,t}^\ell}|b'(\tau)|\,d\tau\leqslant b(t_j).
\]
Therefore, using \eqref{relationtj}, we obtain
\begin{align*}
\int_{0}^{t}\|b'(\tau) f(\tau)\|_{\dot{\B}^{s+1}_{p,1}}^{\ell}\,d\tau
&\leqslant\sum_{j\in \mathbb{Z}}2^{js} \big(2^j b(t_j)\big)\,
\sup_{\tau\in I_{j,t}^\ell}\|\dot{\Delta}_j f(\tau)\|_{L^p}\\
&=2^{-k_0}\sum_{j\in \mathbb{Z}}2^{js}\sup_{\tau\in I_{j,t}^\ell}\|\dot{\Delta}_j f(\tau)\|_{L^p}
=2^{-k_0}\| f\|_{\widetilde{L}^\infty_{t}(\dot{\B}^{s}_{p,1})}^{\ell}.
\end{align*}

For the high-frequency part, similarly,
\begin{align*}
\int_{0}^{t}\left\|\left(\frac{1}{b(\tau)}\right)' f(\tau)\right\|_{\dot{\B}^{s-1}_{p,1}}^{h}\,d\tau
&=\int_{0}^{t}\left|\left(\frac{1}{b(\tau)}\right)'\right|
\sum_{j\geqslant J_\tau+1}2^{j(s-1)} \|\dot{\Delta}_j f(\tau)\|_{L^p}\,d\tau\\
&=\sum_{j\in\mathbb{Z}}2^{j(s-1)} \int_{I_{j,t}^{h}}
\left|\left(\frac{1}{b(\tau)}\right)'\right|\|\dot{\Delta}_j f(\tau)\|_{L^p}\,d\tau\\
&\leqslant \sum_{j\in \mathbb{Z}}2^{j(s-1)}
\sup_{\tau\in I_{j,t}^h}\|\dot{\Delta}_j f(\tau)\|_{L^p}
\int_{I_{j,t}^h}\left|\left(\frac{1}{b(\tau)}\right)'\right|\,d\tau.
\end{align*}
Since $b$ is monotone and, by definition of $I_{j,t}^h$, we have $b(\tau)\ge b(t_j)$ for $\tau\in I_{j,t}^h$,
it follows that
\[
\int_{I_{j,t}^h}\left|\left(\frac{1}{b(\tau)}\right)'\right|\,d\tau
\le \frac{1}{b(t_j)}.
\]
Hence, using again \eqref{relationtj},
\begin{align*}
\int_{0}^{t}\left\|\left(\frac{1}{b(\tau)}\right)' f(\tau)\right\|_{\dot{\B}^{s-1}_{p,1}}^{h}\,d\tau
&\leqslant \sum_{j\in \mathbb{Z}}2^{js}\big(2^{-j}b(t_j)^{-1}\big)
\sup_{\tau\in I_{j,t}^h}\|\dot{\Delta}_j f(\tau)\|_{L^p}\\
&=2^{k_0}\sum_{j\in \mathbb{Z}}2^{js}\sup_{\tau\in I_{j,t}^h}\|\dot{\Delta}_j f(\tau)\|_{L^p}
=2^{k_0}\| f\|_{\widetilde{L}^\infty_{t}(\dot{\B}^{s}_{p,1})}^{h}.
\end{align*}
This completes the proof.
\end{proof}

\section{Analysis of the linear system}\label{sec:linear}

In this section, we perform the linear analysis in the case where $b(t)>0$ is monotone.
The linearized system reads
\begin{equation}
     \left\{
     \begin{aligned}
        & \partial_t Z_1+\sum_{k=1}^d A^k_{1,2}\partial_k Z_2=0, \\
        & \partial_t Z_2+\sum_{k=1}^d\Big(A^k_{2,1}\partial_k Z_1
         +A^k_{2,2}\partial_k Z_2\Big)
         +\frac{L_2 Z_2}{b(t)}=0.
     \end{aligned}
     \right.
     \label{syslin}
\end{equation}

\begin{prop}
Let $s\in\mathbb{R}$. For the linear system \eqref{syslin}, the following estimate holds:
\begin{equation}
\begin{aligned}\label{eqlqlin}
&\|Z\|_{\widetilde{L}_t^\infty(\dot{\B}^{s}_{2,1})}
+\|b(\tau) Z_1\|_{L_t^1(\dot{\B}^{s+2}_{2,1})}^\ell +\|Z_2\|_{L_t^1(\dot{\B}^{s+1}_{2,1})}^\ell+
\left\|\frac{W}{b(\tau)}\right\|_{L_t^1(\dot{\B}^{s}_{2,1})}^\ell
+\Big\|\frac{Z}{b(\tau)}\Big\|_{L^1_t(\dot\B_{2,1}^{s})}^h\lesssim\|Z_0\|_{\dot{\B}^s_{2,1}}.
\end{aligned}
\end{equation}

Furthermore, if $b$ satisfies the condition in Theorem~\ref{thm0}, then we have

\begin{align}
\|b(\tau)Z\|_{\widetilde{L}^\infty_t(\dot\B_{2,1}^{s+1})}^h
+\|Z\|_{L^1_t(\dot\B_{2,1}^{s+1})}^h\lesssim\|Z_0\|_{\dot{\B}^{s}_{2,1}}+\|Z_0\|_{\dot{\B}^{s+1}_{2,1}}^h.\label{estnlhigh}
\end{align}

\end{prop}

\begin{rem}
Compared with the nonlinear a priori estimates established in Theorem~\ref{thm0},
the present linear analysis allows us to handle a more general class of damping coefficients,
namely $b(t)>0$ with $b$ monotone.
In particular, we may take $b(t)=\frac{(1+t)^{\alpha}}{K}$ for all $\alpha\in \mathbb{R}$ and  $K>0$.

However, for the nonlinear system, \eqref{estnlhigh} is required to control the nonlinear terms and to close the uniform
$L^\infty$ estimates.
This, in turn, imposes an extra condition $\limsup_{t\to\infty} b'(t)\ll 1$.
\end{rem}

In this linear setting, we introduce the damped variable
\begin{align}
\label{dampdeflin}
W
:=Z_2+L_2^{-1}b(t)\sum_{k=1}^d\Big(A^k_{2,1}\partial_k Z_1
+A^k_{2,2}\partial_k Z_2\Big).
\end{align}
Using the second equation of \eqref{syslin}, this definition is equivalent to $
W=-L_2^{-1}b(t)\partial_t Z_2$. Then, the system satisfied by $(W,Z_1)$ can be decoupled into a damped equation and a diffusion equation, both with time-dependent coefficients:
\begin{equation}
     \left\{
     \begin{aligned}
        & \partial_t W+\frac{L_2 W}{b(t)}=l_1, \\
       &  \partial_t Z_1
         -b(t)\sum_{k,l=1}^d A^k_{1,2}L_2^{-1}A^l_{2,1}\partial_k\partial_l Z_1
         =l_2,
     \end{aligned}
     \right.
     \label{syslin'}
\end{equation}
with the linear remainder terms
 \begin{equation*}
 \begin{aligned}
l_1&=-b(t)L_2^{-1}\sum_{k=1}^d A_{2,1}^k\partial_k
\Big(\sum_{l=1}^d A_{1,2}^l\partial_l Z_2\Big)
-L_2^{-1}\sum_{k=1}^d A_{2,2}^k\partial_k L_2 W 
+b'(t)L_2^{-1}\sum_{k=1}^d
\left(A_{2,1}^k\partial_k Z_1+A_{2,2}^k\partial_k Z_2\right),\\
 l_2&=-\sum_{k=1}^dA^k_{1,2}\partial_k W+b(t)\sum_{k=1}^d\sum_{l=1}^dA^k_{1,2}L^{-1}_2A^l_{2,2}\partial_k\partial_l Z_2.
 \end{aligned}
  \end{equation*}
  
  \iffalse
We first record estimates for the time derivatives.

\begin{lem}\label{timeestlin}
For all $\sigma\in(-\frac d2,\frac d2]$, we have
\begin{align*}
\|\partial_t Z_1\|_{\dot{\B}_{2,1}^\sigma}
&\lesssim \|\nabla Z_2\|_{\dot{\B}_{2,1}^\sigma},\quad \|\partial_t Z_2\|_{\dot{\B}_{2,1}^\sigma}
\lesssim \left\|\frac{W}{b(t)}\right\|_{\dot{\B}_{2,1}^\sigma}.
\end{align*}
\end{lem}

\begin{proof}
The first estimate follows directly from the first equation of \eqref{syslin}:
\[
\partial_t Z_1=-\sum_{k=1}^d A^k_{1,2}\partial_k Z_2.
\]
The second estimate is an immediate consequence of
\[
\partial_t Z_2=-\frac{L_2 W}{b(t)},
\]
which follows from \eqref{dampdeflin}.
\end{proof}
\fi

%For simplicity, we restrict ourselves to the case where $b$ is increasing; the decreasing case can be treated in a similar way.

\subsection{Low frequencies: the under-damped case}\label{sect:linear:low}
We shall first deal with the under-damped case, namely $b'(t)>0$. We first exploit sharp dissipation estimates in the low-frequency regime. To this end, we refine the decoupling energy argument introduced in \cite{c3} to handle a general time-dependent damping coefficient $b(t)$. The main difficulty arises from the intrinsic coupling between time, frequencies and the coefficient in the system. To overcome this difficulty, we introduce a time-dependent frequency threshold and perform the energy estimates on dynamically evolving low frequencies $b(t)2^j \lesssim 1$, while absorbing the higher-order terms through a careful treatment of the time integrations.

%to capture the maximal regularity gain of solutions for partiallydissipative hyperbolic systems with constant relaxation.Here, we adapt this argument to the case of time-dependent damping.

Fix $T>0$. Since $b$ is increasing, we have $I_{j,t}^\ell=[0,\min\{t,t_j\}]$ whenever $t_j>0$. Let us write $(W_{j},Z_{1,j},Z_{2,j})=\dot{\Delta}_j (W,Z_1,Z_2)$ and $(l_{1,j},l_{2,j})=\dot{\Delta}_j(l_1,l_2)$. 
Applying $\dot{\Delta}_j$ to the first equation of \eqref{syslin'} and taking
the $L^2$ inner product with $W_j$ yields, for $\tau\in I_{j,t}^\ell$,
\begin{align}
\frac{1}{2}\frac{d}{d\tau}\|W_j(\tau)\|_{L^2}^2
+\frac{\kappa}{b(\tau)}\|W_j(\tau)\|_{L^2}^2
\le \|l_{1,j}(\tau)\|_{L^2}\|W_j(\tau)\|_{L^2}.
\end{align}
Similarly, from the second equation of \eqref{syslin'} and the strong ellipticity of the operator
\(
-\sum_{k,l=1}^d A^k_{1,2}L_2^{-1}A^l_{2,1}\partial_k\partial_l,
\)
due to the {\rm (SK)} condition, we deduce that, for all $\tau\in I_{j,t}^\ell$,
 %, for some harmless constant $C>0$,
\begin{align*}
\frac{1}{2}\frac{d}{d\tau}\|Z_{1,j}(\tau)\|_{L^2}^2
+\kappa_0 b(\tau)2^{2j}\|Z_{1,j}(\tau)\|_{L^2}^2
\le \|l_{2,j}(\tau)\|_{L^2}\|Z_{1,j}(\tau)\|_{L^2}.
\end{align*}
Applying Lemma~\ref{lem2} and noting $$
\partial_t Z_1=l_2+b(t)\sum_{k,l=1}^d A^k_{1,2}L_2^{-1}A^l_{2,1}\partial_k\partial_l Z_1,\quad \partial_t Z_2=-\frac{L_2 W}{b(t)},
$$
we infer that, for some uniform constants $\kappa_1, \kappa_2>0$,
\begin{equation}
\begin{aligned}\label{west:lin1}
&\quad\sup_{\tau\in  I_{j,t}^\ell}\|W_j(\tau)\|_{L^2}
+\kappa_1\int_{I_{j,t}^\ell}\frac{1}{b(\tau)}\|W_j(\tau)\|_{L^2}\,d\tau
\le \|W_j(0)\|_{L^2}
+C\int_{I_{j,t}^\ell}
\big\|l_{1,j}(\tau)\big\|_{L^2}\,d\tau,
\end{aligned}
\end{equation}
and
\begin{equation}
\begin{aligned}\label{est:Z1lin1}
&\quad \sup_{\tau\in  I_{j,t}^\ell} \|Z_{1,j}(\tau)\|_{L^2}
+\kappa_2\int_{I_{j,t}^\ell}2^{2j}\|b(\tau)Z_{1,j}(\tau)\|_{L^2}\,d\tau\le \|Z_{1,j}(0)\|_{L^2}
+C\int_{I_{j,t}^\ell}\big\|l_{2,j}(\tau)\big\|_{L^2}\,d\tau.
\end{aligned}
\end{equation}
Here and in what follows, the constant $C$ stands for a harmless universal constant
that may change from line to line, and it depends only on the matrices $A^k$ and $L_2$
(in particular, it is independent of $b(t)$, $j$ and time).

We will verify that the higher-order linear terms on the right-hand sides of \eqref{west:lin1}--\eqref{est:Z1lin1} can be absorbed by the dissipation terms on the left-hand side. In fact, using the explicit form of $l_1, l_2$ and Bernstein's inequality, we have
\begin{align*}
\|l_{1,j}\|_{L^2} &\lesssim  2^{2j} \|b(\tau) Z_{2,j}(\tau)\|_{L^2}+ 2^j \|W_j(\tau)\|_{L^2}+|b'(\tau)|2^j \|Z_j(\tau)\|_{L^2},\\
\|l_{2,j}\|_{L^2} &\lesssim 2^j\|W_j(\tau)\|_{L^2}
+2^{2j}\|b(\tau)Z_{2,j}(\tau)\|_{L^2},
\end{align*}
where we have used the expression of $\partial_t Z_{1,j}$ to derive the estimate of $l_{1,j}$.
In the low-frequency regime, namely for all $\tau\in I_{j,t}^\ell$ such that $
2^j b(\tau)\le 2^{-k_0}$,
the higher-order terms on the right-hand sides of
\eqref{west:lin1}--\eqref{est:Z1lin1} can be controlled.
Indeed, by the definition of $W$ and Bernstein's inequality, there exists a constant
$C>0$ such that
\begin{equation}\label{Z2Wpre}
\|Z_{2,j}\|_{L^2}
\le C\big(\|W_j\|_{L^2}+2^j b(\tau)\|(Z_{1,j},Z_{2,j})\|_{L^2}\big).
\end{equation}
 We now fix
\begin{equation}\label{k0choice}
k_0:=\left\lceil \log_2\!\left(\frac{4C}{\min\{1,\kappa,\kappa_0\}}\right)\right\rceil.
\end{equation}
This implies $C\,2^{-k_0}\le \tfrac14\min\{1,\kappa,\kappa_0\}$.
Therefore, the last term on the right-hand side of \eqref{Z2Wpre} can be absorbed,
and we obtain, for all $\tau\in I_{j,t}^\ell$ in the low-frequency regime,
\begin{equation}\label{Z2Wabs}
\|Z_{2,j}(\tau)\|_{L^2}
\le C\big(\|W_j\|_{L^2}+2^j b(\tau)\|Z_{1,j}\|_{L^2}\big)\leqslant C\big(\|W_j\|_{L^2}+2^{-k_0}\|Z_{1,j}\|_{L^2}\big).
\end{equation}
Using arguments similar to those in Lemma~\ref{b'lem}, we have
\begin{equation}
\begin{aligned}
\label{b'argumentlow}
\int_{I_{j,t}^\ell}|b'(\tau)|2^j \|Z_j(\tau)\|_{L^2}
\leqslant &\int_{I_{j,t}^\ell}|b'(\tau)|\,d\tau\, 2^j \|Z_{j}\|_{L^{\infty}(I_{j,t}^\ell;L^2)}\\
\lesssim &2^{-k_0} \|(Z_{1,j},W_j)\|_{L^{\infty}(I_{j,t}^\ell;L^2)}.
\end{aligned}
\end{equation}
Consequently, due to the choice \eqref{k0choice}
we have 
\begin{equation}\label{l:es}
\begin{aligned}
&\quad C\int_{I_{j,t}^\ell}
\left(\|l_{1,j}(\tau)\|_{L^2}+\|l_{2,j}(\tau)\|_{L^2}\right)\,d\tau\\
&\leq C 2^{-k_0}\int_{I_{j,t}^\ell}\Big(2^{2j}\|b(\tau) Z_{1,j}(\tau)\|_{L^2}+\frac{1}{b(\tau)} \|W_j(\tau)\|_{L^2}\Big)\,d\tau +C2^{-k_0} \|(Z_{1,j},W_j)\|_{L^{\infty}(I_{j,t}^\ell;L^2)},
\end{aligned}
\end{equation}
where we used \eqref{Z2Wabs}.  
Thus, substituting \eqref{l:es} into   \eqref{west:lin1}--\eqref{est:Z1lin1}  and  using the smallness of $2^{-k_0}$, one can absorb the higher-order linear terms. Moreover, combining the estimates of $Z_{1,j}$ and $W_{j}$ and using \eqref{Z2Wpre}, we can recover the desired bound for $Z_{2,j}$, while the initial datum $W_j(0)$ can be controlled in terms of $Z_j(0)$. Precisely, we arrive at
\begin{equation}
\begin{aligned}\label{eqlqloclin}
&\sup_{\tau\in I_{j,t}^\ell}\|Z_j(\tau)\|_{L^2}
+\int_{I_{j,t}^\ell}
\left\|\left(\frac{W_j}{b(\tau)},2^{2j}b(\tau)Z_{1,j},2^jZ_{2,j}\right)\right\|_{L^2}
\,d\tau
\lesssim \|Z_j(0)\|_{L^2}.
\end{aligned}
\end{equation}
Multiplying by $2^{js}$ and summing over the low-frequency dyadic blocks, we obtain the low-frequency part of \eqref{eqlqlin}.

\begin{rem}
The choice of $k_0$ depends only on the universal constants appearing in the dissipation estimates,
and therefore is independent of the explicit form of $b(t)$.
As a consequence, we may assume without loss of generality that $b$ satisfies
the smallness condition associated with $k_0$.
\end{rem}

\begin{rem}
The above result can be extended to the case where $b'(t)$
changes sign only a finite number of times.
In that situation, the time interval can be decomposed into finitely many
subintervals on which $b'(t)$ has a fixed sign, and the analysis may be
performed on each subinterval separately.
As the proof is similar, we omit the details for the sake of
clarity.
\end{rem}

\subsection{High frequencies: the under-damped case}

%We shall construct a Lyapunov functional for the entropy variable $U=U(V)$
%\eqref{entropeq}--\eqref{rtildeU},
%$\widetilde{U}:=U-\bar{U}$ satisfies the system
%\begin{align}
%\widetilde{A}^{0}(U)\partial_{t}\widetilde{U}
%+\sum_{k=1}^{d}\widetilde{A}^{k}(U)\partial_{x_{k}}\widetilde{U}
%+\frac{\widetilde{B}\widetilde{U}}{b(t)}
%=\frac{\mathbf{r}(U)}{b(t)},
%\label{entropeq:1}
%\end{align}
%with initial datum
%\begin{align}\label{entropeq:2}
%\widetilde{U}(0,x)=\widetilde{U}_0(x):=U(V_0)-\bar{U}.
%\end{align}
%That is, by combining the corrector with the classical {\rm (SK)} condition,
%we aim to capture the full dissipation mechanism and establish
%uniform bounds for $U-\bar{U}$, in the spirit of hypocoercivity.
%Since $U=U(V)$ depends smoothly on $V$ and the entropy $\eta(V)$ is convex,
%the desired estimates for $V$ can be recovered from those for $U$.

%\medskip

We follow the same strategy as in the time-independent case
developed by Beauchard and Zuazua \cite{BZ}.
We consider the linearized system associated with \eqref{entropeq}. Writing the linearized system in the Fourier variables yields
\begin{align*}
\bar{A}^{0}\partial_{t}z+{\rm i}\sum_{k=1}^{d}\bar{A}^{k}\xi_k z
+\frac{\bar{B}z}{b(t)}=0,
\end{align*}
where $\bar{A}^k=\widetilde{A}^k(\bar{U})$ for $k=0,\dots,d$, $\bar{B}=\widetilde{B}(\bar{U})$,
$z=\widehat{U-\bar{U}}$, and $\xi$ denotes the Fourier variable.
Let $\xi=r\omega$ with $\omega\in\SSS^{d-1}$ and $r=|\xi|$.
The system can then be rewritten as
\begin{align}
\bar{A}^{0}\partial_t z+r\bar{A}_\omega z+\frac{\bar{B} z}{b(t)}=0,
\qquad
\bar{A}_\omega:= {\rm i} \sum_{k=1}^d \bar{A}^k\omega_k.
\label{eqlin}
\end{align}
It is clear that $\bar{A}_\omega$ is skew-Hermitian.
Let 
$$
A:=(\bar{A}^0)^{-1}\bar{A}_\omega,\quad B:=(\bar{A}^0)^{-1}\bar{B}.
$$
Note that $A$ depends on $\omega$, but we drop the index for simplicity.

We now introduce the corrector $\mathcal{I}$.
Fix $n-1$ small positive parameters $\varepsilon_1,\dots,\varepsilon_{n-1}$ and define
\begin{align}
\mathcal{I}
:=\Re\sum_{k=1}^{n-1}\varepsilon_k(BA^{k-1}z\cdot BA^k z),
\label{corrector}
\end{align}
where $(\,\cdot\,)$ denotes the Hermitian inner product in $\C^n$.
Taking the time derivative of $\mathcal{I}$ and using \eqref{eqlin} yields
\begin{align}
\frac{d}{dt}\mathcal{I}
+\sum_{k=1}^{n-1}\varepsilon_k r|BA^kz|^2
=&-\Re\sum_{k=1}^{n-1}\varepsilon_k r (BA^{k-1}z\cdot BA^{k+1}z) \notag\\
&-\frac{1}{b(t)}\Im\Big(\sum_{k=1}^{n-1}\varepsilon_k(
BA^{k-1}Bz\cdot BA^kz
+BA^{k-1}z\cdot BA^kBz)\Big).
\label{eqcorrec}
\end{align}

Choosing the parameters $\varepsilon_k$ appropriately
(details are given in Appendix~\ref{sec:appendixBZ}), we obtain
\begin{align}
\frac{d}{dt}\sI
+\frac{r}{2}\sum_{k=1}^{n-1}\varepsilon_k|BA^{k}z|^2
\leqslant
\Big(\frac{1}{4b^2(t)r}+\frac{r}{4}+\frac{1}{4b(t)}\Big)|Bz|^2.
\label{xianxing}
\end{align}
Then, using the structural assumption on $\bar{B}$ together with
a classical energy estimate, we obtain
\begin{gather}
\frac{1}{2}\frac{d}{dt}( \bar{A}^0 z\cdot z)
+\frac{ c'|B z|^2}{b(t)}\leqslant 0.
\label{eneeq}
\end{gather}

We now define the Lyapunov functional
\[
L_{r,\omega}(z,t)
:= \left( \bar{A}^0 z\cdot z\right)
+\frac{1}{2^{2k_0}b(t)r}
\sI,
\]
which is well-defined in the high-frequency regime $b(t)r\geqslant 2^{-k_0}$.
Combining \eqref{xianxing} and \eqref{eneeq}, we obtain
\begin{multline*}
\frac{d}{dt}L_{r,\omega}(z,t)
+\frac{2^{-2k_0-1}}{b(t)}\sum_{k=1}^{n-1}\varepsilon_k|B A^k z|^2
+\frac{c'|Bz|^2}{b(t)}\\
\leqslant
2^{-2k_0}\Big(\frac{1}{4b^3(t)r^2}
+\frac{1}{4b(t)}
+\frac{1}{4b^2(t)r}\Big)|Bz|^2
+\frac1{2^{2k_0}r}\Big(\frac{1}{b(t)}\Big)'
\sI.
\end{multline*}
Since $A$ and $B$ are bounded, we may assume without loss of generality that
\begin{align}
\sum_{k=1}^{n-1}\varepsilon_k
\big|\Re(BA^{k-1} z\cdot BA^k z)\big|
<\frac{|z|^2}{8},
\qquad \forall z\in\C^n.
\label{bdd}
\end{align}
In the high-frequency regime, this implies that
$L_{r,\omega}(z,t)$ is equivalent to $|z|^2$.

Using repeatedly the condition $b(t)r\geqslant 2^{-k_0}$ and \eqref{bdd},
we deduce
\begin{align*}
\frac{d}{dt}L_{r,\omega}(z,t)
+\frac{2^{-k_0-1}}{b(t)}\sum_{k=0}^{n-1}\varepsilon_k|BA^k z|^2
\leqslant
\frac1{2^{2k_0+3}r}\Big(\frac{1}{b(t)}\Big)'|z|^2,
\end{align*}
where we have chosen $\varepsilon_0=c'$.

If the system \eqref{eq0} satisfies the {\rm (SK)} condition, then by
Proposition~\ref{Propositionkalman} and the compactness of $\SSS^{d-1}$,
there exists $\widetilde{c}>0$ such that
\begin{align}
\sum_{i=0}^{n-1}\varepsilon_i|BA^i y|^2
\geqslant \widetilde{c}|y|^2,
\qquad \forall y\in\C^n.
\label{xiajie}
\end{align}
Hence,
\[
\frac{d}{dt}L_{r,\omega}(z,t)
+\frac{2^{-k_0-1}\widetilde{c}}{b(t)}|z|^2
\leqslant
\frac1{2^{2k_0+3}r}\Big(\frac{1}{b(t)}\Big)'|z|^2.
\]

\medskip

In the spirit of the above analysis, we define
\begin{align*}
\mathcal{L}_{j}(t)
:=\langle\bar{A}^0\widehat{U}_j,\widehat{U}_j\rangle_{L^2}
+\frac{\delta}{2^{2k_0}b(t)}
\sum_{k=1}^{n-1}\varepsilon_k
\Re\!\int_{\R^d}
(BA^{k-1}\widehat{U}_j(\xi))
\cdot(BA^{k}\widehat{U}_j(\xi))
|\xi|^{-1}\,d\xi.
\end{align*}
Using similar arguments, we obtain for all $\tau\in I_{j,t}^h$,
\begin{align*}
\frac{d}{d\tau}\mathcal{L}_{j}(\tau)
+\frac{2^{-k_0-2}\widetilde{c}\delta}{b(\tau)}
\|\widehat{U}_j(\tau)\|_{L^2}^2
\leqslant
\frac{\delta}{2^{2k_0+2+j}}
\Big(\frac{1}{b(\tau)}\Big)'
\|\widehat{U}_j(\tau)\|_{L^2}^2.
\end{align*}

Applying Lemma~\ref{lem2} and Plancherel's theorem yields
\begin{align*}
\sqrt{\mathcal{L}_{j}(t)}
+\int_{t_j}^t
\frac{2^{-k_0-2}\widetilde{c}\delta}{b(\tau)}
\|U_j\|_{L^2}\,d\tau
\lesssim
\sqrt{\mathcal{L}_{j}(t_j)}
+\int_{I_{j,t}^h}
\frac{\delta}{2^{2k_0+2+j}}
\left|\Big(\frac{1}{b(\tau)}\Big)'\right|
\|U_j\|_{L^2}\,d\tau.
\end{align*}
For the right-hand side, arguing as in Lemma~\ref{b'lem}, we obtain
\begin{align}
\int_{I_{j,t}^h}
\frac{\delta}{2^{2k_0+2+j}}
\Big(\frac{1}{b(\tau)}\Big)'
\|U_j\|_{L^2}\,d\tau
\leqslant
\frac{\delta}{2^{k_0+2}}
\sup_{I_{j,t}^h}\|U_j\|_{L^2}.
\label{b'argumenthigh}
\end{align}
Since we are in the linear framework, the transformation from $U$ to $Z$ is indeed a linear function. Then using the fact that $\delta$ is small, we arrive at
\begin{align*}
\sup_{I_{j,t}^h}\|U_j(\tau)\|_{L^2}
+\int_{I_{j,t}^h}
\frac{2^{-k_0-2}\widetilde{c}\delta}{b(\tau)}
\|U_j\|_{L^2}\,d\tau
\lesssim \sqrt{\sL_{j}(t_j)} \lesssim\|U_j(t_j)\|_{L^2}\lesssim\|Z_j(t_j)\|_{L^2}.
\end{align*}
As in the low-frequency case, multiplying by $2^{js}$, summing over all dyadic blocks and using again the linear relation between $Z$ and $U$ yields
\begin{align*}
\|Z\|_{\widetilde{L}^\infty_t(\dot\B_{2,1}^{s})}^h
+\Big\|\frac{Z}{b(\tau)}\Big\|_{L^1_t(\dot\B_{2,1}^{s})}^h
&\lesssim
\sum_{j\in\Z,\;0<t_j\le t}2^{js}\|Z_j(t_j)\|_{L^2}
+\sum_{j\in\Z,\; t_j=0}2^{js}\|Z_j(0)\|_{L^2}\\
&\lesssim
\|Z\|_{\widetilde{L}^\infty_t(\dot\B_{2,1}^{s})}^\ell+\|Z_0\|_{\dot{\B}^s_{2,1}}^h\lesssim\|Z_0\|_{\dot{\B}^s_{2,1}}.
\end{align*}

For the proof of \eqref{estnlhigh}, by the additional condition on $b$, there exists $t^*>0$ such that  
\begin{align}
    \forall t>t^*,\quad b'(t)<2c^*.\label{b'condi}
\end{align}
Then, for all $\tau\in I_{j,t}^h$, we have
\begin{align*}
\frac{d}{d\tau}\big(b(\tau+t^*)\mathcal{L}_{j}(\tau)\big)
+c'\mathcal{L}_{j}(\tau)-b'(\tau+t^*)\mathcal{L}_{j}(\tau)
\leqslant 0.
\end{align*}
Hence, if $c^*$ is small enough, then
\[
c'-b'(\tau+t^*)\geqslant c'-2c^*\geqslant \frac{c'}{2}.
\]
Therefore, 
\begin{align*}
\sup_{I_{j,t}^h}\|b(\tau+t^*)U_j(\tau)\|_{L^2}
+\int_{I_{j,t}^h}
\|U_j\|_{L^2}\,d\tau
\lesssim\|b(t_j+t^*)U_j(t_j)\|_{L^2}\lesssim\|b(t_j)Z_j(t_j)\|_{L^2},
\end{align*}
where we used, for  $\bar M:=\sup_{t\ge0} b'(t)<\infty$,
\[
\forall t>0,\quad
\frac{b(t+t^*)}{b(t)}
\leqslant \frac{b(t)+\bar Mt^*}{b(t)}
\leqslant 1+\frac{\bar Mt^*}{b(0)}
\leqslant C.
\]
Finally, multiplying by $2^{j(s+1)}$ and summing up yields
\begin{align*}
\|b(\tau)Z\|_{\widetilde{L}^\infty_t(\dot\B_{2,1}^{s+1})}^h
+\Big\|Z\Big\|_{L^1_t(\dot\B_{2,1}^{s+1})}^h
&\lesssim
\sum_{j\in\Z,\;0<t_j\le t}2^{j(s+1)}\|b(t_j)Z_j(t_j)\|_{L^2}
+\sum_{j\in\Z,\; t_j=0}2^{j(s+1)}\|b(0)Z_j(0)\|_{L^2}\\
&\lesssim
\|b(\tau)Z\|_{\widetilde{L}^\infty_t(\dot\B_{2,1}^{s+1})}^\ell
+\|Z_0\|_{\dot{\B}^{s+1}_{2,1}}^h\\
&\lesssim
\|Z\|_{\widetilde{L}^\infty_t(\dot\B_{2,1}^{s})}^\ell
+\|Z_0\|_{\dot{\B}^{s+1}_{2,1}}^h\\
&\lesssim
\|Z_0\|_{\dot{\B}^{s}_{2,1}}
+\|Z_0\|_{\dot{\B}^{s+1}_{2,1}}^h.
\end{align*}

\subsection{The over-damped case in low and high frequencies}
In this case, we have $I_{j,t}^\ell=[t_j,\infty)\cap [0,t]$ and $I_{j,t}^h=[0,t_j)\cap [0,t]$. Therefore, the main difference from the under-damped case lies in the time at which each dyadic block enters the low- or high-frequency regime, and hence in the way the initial data are used. In particular, we notice that \eqref{b'argumentlow} and \eqref{b'argumenthigh} are still valid as in Lemma \ref{b'lem}. For the high frequency part, we have 
\begin{align*}
\sup_{I_{j,t}^h}\|U_j(\tau)\|_{L^2}
+\int_{I_{j,t}^h}
\frac{\|U_j\|_{L^2}}{b(
\tau)}\,d\tau
\lesssim\|U_j(0)\|_{L^2},
\end{align*}
which leads to
\begin{align*}
\|Z\|_{\widetilde{L}^\infty_t(\dot\B_{2,1}^{s})}^h
+\Big\|\frac{Z}{b(\tau)}\Big\|_{L^1_t(\dot\B_{2,1}^{s})}^h
&\lesssim
\sum_{j\in\Z,\;t_j>0}2^{js}\|Z_j(0)\|_{L^2}
\lesssim\|Z_0\|_{\dot{\B}^s_{2,1}}^h.\end{align*}
Since $b'(\tau)<0$, we directly have, for all $\tau\in I_{j,t}^h$,
\begin{align*}
\frac{d}{d\tau}(b(\tau)\mathcal{L}_{j}(\tau))
+c'\mathcal{L}_{j}(\tau)\leqslant\frac{d}{d\tau}(b(\tau)\mathcal{L}_{j}(\tau))
+c'\mathcal{L}_{j}(\tau)-b'(\tau)\mathcal{L}_{j}(\tau)
\leqslant
0,
\end{align*}
and then immediately 
\begin{align*}
\|b(\tau)Z\|_{\widetilde{L}^\infty_t(\dot\B_{2,1}^{s+1})}^h
+\Big\|Z\Big\|_{L^1_t(\dot\B_{2,1}^{s+1})}^h
\lesssim\|Z_0\|_{\dot{\B}^{s+1}_{2,1}}^h.
\end{align*}
For the low frequency part, we obtain 
\begin{equation*}
\begin{aligned}
&\sup_{\tau\in I_{j,t}^\ell}\|Z_j(\tau)\|_{L^2}
+\int_{I_{j,t}^\ell}
\left\|\left(\frac{W_j}{b(\tau)},2^{2j}b(\tau)Z_{1,j},2^jZ_{2,j}\right)\right\|_{L^2}
\,d\tau
\lesssim \|Z_j(t_j)\|_{L^2},
\end{aligned}
\end{equation*}
and similarly 
\begin{equation*}
\begin{aligned}
\|Z\|_{\widetilde{L}_t^\infty(\dot{\B}^{s}_{2,1})}
^\ell+
\left\|(b(\tau)\nabla^2 Z_1,\nabla Z_2,\frac{W}{b(\tau)})\right\|_{L_t^1(\dot{\B}^{s}_{2,1})}^\ell
&\lesssim
\sum_{j\in\Z,\;0<t_j\le t}2^{js}\|Z_j(t_j)\|_{L^2}
+\sum_{j\in\Z,\; t_j=0}2^{js}\|Z_j(0)\|_{L^2}\\
&\lesssim
\|Z\|_{\widetilde{L}^\infty_t(\dot\B_{2,1}^{s})}^h+\|Z_0\|_{\dot{\B}^s_{2,1}}^\ell\lesssim\|Z_0\|_{\dot{\B}^s_{2,1}}.
\end{aligned}
\end{equation*}
Consequently, combining the low- and high-frequency estimates, we obtain both \eqref{eqlqlin} and \eqref{estnlhigh}.

\section{Proof of Theorem \ref{thm0}}\label{sec:global}

In this section, we provide a detailed proof of Theorem \ref{thm0} concerning the global existence of the nonlinear system. Throughout this section, we assume that there exists a smooth solution
$Z$ to system \eqref{eq0} on $[0,T)\times\mathbb{R}^d$  satisfying the {\emph{a priori}} assumption
\begin{align}
\label{small}
\sup_{t\in[0,T)}
\|Z\|_{\mathcal{E}_t}\le \varepsilon_0,
\end{align}
Here $\|Z\|_{\mathcal{E}_t}$ is defined in \eqref{X}, and $\varepsilon_0\in(0,1)$ is a small constant to be chosen later.

Owing to the embedding
$\dot{\B}_{2,1}^{d/2}\hookrightarrow L^\infty$,
this assumption implies a uniform $L^\infty$ bound:
\begin{align}\label{Linfty}
\sup_{t\in[0,T)}\|Z(t)\|_{L^\infty}
\lesssim \varepsilon_0 \lesssim 1,
\end{align}
which will play a crucial role in controlling the nonlinear terms.

\medskip

\medskip

We now state the main a priori estimate for the nonlinear system.

\begin{prop}\label{apriori}
Let $Z=(Z_1,Z_2)$ be a solution to the Cauchy problem \eqref{eq0}
defined on $[0,T)\times\mathbb{R}^d$.
There exists a constant $\varepsilon_0>0$ such that if \eqref{small} holds,
then for all $t\in[0,T)$,
\begin{align}\label{aprioriestimate}
\|Z\|_{\mathcal{E}_t}
\le
C_0\bigl(1+\|Z_0\|_{\dot{\B}_{2,1}^{\frac{d}{2}}}\bigr)
\|Z_0\|_{\dot{\B}_{2,1}^{\frac{d}{2}}\cap \dot{\B}_{2,1}^{\frac{d}{2}+1}},
\end{align}
where the norm $\|Z\|_{\mathcal{E}_t}$ is defined in \eqref{X}, and $C_0>0$ is a universal constant independent of $T$.
\end{prop}

If $b(t)\equiv b_0>0$,
then $J_t\equiv J_0$, so that the time-dependent frequency decomposition
reduces to the usual fixed low--high frequency splitting. Moreover, all
terms involving $b'(t)$ vanish. Consequently, the low- and high-frequency
estimates established below reduce directly to the corresponding
fixed-threshold estimates. The same nonlinear estimates therefore yield
Proposition~\ref{apriori} in this case. Therefore, we only provided the details of proof for the strictly monotone cases $b'(t)>0$ and $b'(t)<0$.

\subsection{Low-frequency analysis in the under-damped case}

We begin with the low-frequency analysis and restrict ourselves to the
case where $b(t)$ is increasing.
The decreasing case can be treated similarly. Although the estimates below are inspired by the linear low-frequency
analysis, they cannot be obtained by a straightforward application of the
linear estimate with source terms. The reason is that the nonlinear source
terms entering the equation of the damped mode contain time derivatives of
the unknown. Hence the bounds for \(\partial_t Z_1\) and \(\partial_t Z_2\)
have to be derived simultaneously with those for \(W,Z_1\) and \(Z_2\).

\begin{lem}\label{eslq}
For all $t\in[0,T]$, the following estimate holds:
\begin{align}
\label{eqlq}
&\|Z\|^\ell_{\widetilde{L}^{\infty}_t(\dot\B_{2,1}^{\frac{d}{2}})}
+\int_{0}^t\Big(\|b(\tau) Z_1\|_{\dot{\B}^{\frac d2+2}_{2,1}}^\ell
+\|Z_2\|_{\dot{\B}^{\frac d2+1}_{2,1}}^\ell
+\left\|\frac{W}{b(\tau)}\right\|_{\dot{\B}^{\frac d2}_{2,1}}^\ell
+\|\partial_t Z\|_{\dot{\B}^{\frac d2}_{2,1}}^\ell
\Big)\,d\tau\notag\\
&\quad+\|b(\tau)^{\frac12}Z\|^\ell_{\widetilde{L}^{2}_t(\dot\B_{2,1}^{\frac{d}{2}+1})}
+\|b(\tau)^{-\frac12}Z_2\|^\ell_{\widetilde{L}^{2}_t(\dot\B_{2,1}^{\frac{d}{2}})}
\notag\\
\lesssim\;&
\big(1+\|Z_0\|_{\dot{\B}_{2,1}^{\frac{d}{2}}}\big)
\Big(
\|Z_0\|_{\dot{\B}_{2,1}^{\frac{d}{2}}}^{\ell}
+\|Z_0\|_{\dot{\B}_{2,1}^{\frac{d}{2}+1}}^{h}
\Big)
+\varepsilon_0\,\|Z\|_{\mathcal{E}_t}.
\end{align}
\end{lem}
% We first recall an auxiliary estimate that will be used repeatedly in the
% subsequent analysis.

%    \begin{lem}
    % \label{timeest}
    %     Under hypotheses \eqref{strucadd} and \eqref{small}, we have for all $\sigma \in ]-d/2,d/2],$
     %    \begin{align*}
     %        &\|\partial_t Z_1\|_{\dot{\B}_{2,1}^\sigma}\lesssim\|\nabla Z_2\|_{\dot{\B}_{2,1}^\sigma}+\| Z_2\|_{\dot{\B}_{2,1}^{\frac d2}}\|\nabla Z_1\|_{\dot{\B}_{2,1}^\sigma},\\
      %       &\|\partial_t Z_2\|_{\dot{\B}_{2,1}^\sigma}\lesssim\left\|\frac{W}{b(t)}\right\|_{\dot{\B}_{2,1}^\sigma}.
     %    \end{align*}
  %   \end{lem}
   %  \begin{proof}
    %     The second inequality comes directly from the definition of $W$ is equivalent to 
     %    \begin{align}
    %     \label{z2eq}
     %        \partial_t Z_2=-\frac{L_2W}{b(t)}.
    %     \end{align}
     %    For the first item, we need the explicit expression of $\partial_t Z_1$: since $A_{1,1}^k=0$ and by the condition \eqref{strucadd}, we have,
     %    \begin{align}
     %    \label{z1eq}
     %        \partial_t Z_1+\sum_{k=1}^dA^k_{1,2}\partial_k Z_2=-\sum_{k=1}^d\left(\widetilde{A}^k_{1,1}(Z_2)\partial_k Z_1+\widetilde{A}^k_{1,2}(Z)\partial_k Z_2\right).
    %     \end{align}
       %  All the terms on the right-hand side % are at least quadratic, hence condition \eqref{small}  ensures the desired inequality of $\partial_t Z_1$.
   %  \end{proof}
    % Now we turn to the proof of  Lemma~\ref{eslq}.
    \begin{proof}
    \subsubsection*{Step 1: Decoupling analysis for  \texorpdfstring{$(W,Z_1)$}{TEXT}} 
    %We have the following statement.
We shall prove the low-frequency estimate
\begin{equation}
\begin{aligned}
\label{est:W}
&\quad\|Z\|^\ell_{\widetilde{L}^{\infty}_t(\dot\B_{2,1}^{\frac{d}{2}})}
+\int_{0}^t\Big(\|b(\tau) Z_1\|_{\dot{\B}^{\frac{d}{2}+2}_{2,1}}^\ell +\|Z_2\|_{\dot{\B}^{\frac{d}{2}+1}_{2,1}}^\ell+
\left\|\frac{W}{b(\tau)}\right\|_{\dot{\B}^{\frac{d}{2}}_{2,1}}^\ell+\|\partial_t Z\|_{\dot{\B}^{\frac{d}{2}}_{2,1}}^\ell
\Big)\,d\tau
\\
&\lesssim\|W_0\|_{\dot\B_{2,1}^{\frac{d}{2}}}^\ell+
\|Z_{1,0}\|_{\dot\B_{2,1}^{\frac{d}{2}}}^\ell\\
&\quad+\|(f_1,f_2)\|_{L^1_t(\dot{\B}^{\frac{d}{2}}_{2,1})}^\ell
+\sum_{k=1}^{d}\|b(\tau)N_2^k(Z)\|_{\widetilde{L}^{\infty}_t(\dot\B_{2,1}^{\frac{d}{2}})\cap L^1_t(\dot{\B}^{\frac{d}{2}+2}_{2,1})}^\ell
+\|Q(Z)\|_{\widetilde{L}^{\infty}_t(\dot\B_{2,1}^{\frac{d}{2}})\cap L^1_t(\dot{\B}^{\frac{d}{2}+1}_{2,1})}^\ell.
\end{aligned}
\end{equation}

To prove \eqref{est:W}, we modify the localized energy estimates in the linear analysis
to the nonlinear system. Viewing the nonlinear terms $f_1$ and $f_2$ as source terms in the linear analysis, we obtain, for all
$\tau\in I_{j,t}^\ell$,
\begin{equation}
\begin{aligned}\label{Wj-int-nl}
&\quad\sup_{\tau\in I_{j,t}^\ell}\|W_j(\tau),Z_{1,j}(\tau)\|_{L^2}
+\int_{I_{j,t}^\ell}\Big(\frac{\|W_j(\tau)\|_{L^2}}{b(\tau)}+2^{2j}b(\tau)\|Z_{1,j}(\tau)\|_{L^2}\Big)\,d\tau\\
&\lesssim 
\|W_j(0), Z_{1,j}(0)\|_{L^2}
+\int_{I_{j,t}^\ell}
\big(\|l_{1,j}(\tau)\|_{L^2}+\|f_{1,j}(\tau)\|_{L^2}+\|l_{2,j}(\tau)\|_{L^2}+\|f_{2,j}(\tau)\|_{L^2}\big)\,d\tau.
\end{aligned}
\end{equation}

We observe that the only difference with the linear analysis is the appearance of non-linear terms. Moreover, it will also influence the estimates for $l_1$ and $l_2$. However, the solution is just to use repeatedly the fact that in the low-frequency regime, 
$b(\tau)2^j\le 2^{-k_0}$. By the definition of $W$, we have that
\begin{align}\label{Z2j-nl-abs}
\|Z_{2,j}\|_{L^2}
\lesssim
\|W_j\|_{L^2}
+b(\tau)2^j\|Z_{1,j}\|_{L^2}
+b(\tau)2^j\sum_{k=1}^d\|\dot\Delta_j N_2^k(Z)\|_{L^2}
+\|\dot\Delta_j Q(Z)\|_{L^2}.
\end{align}
Hence, we can recover the bound for $Z_{2,j}$.
Then similarly to the linear case,  we have
\begin{equation}\label{dyadic-coupled-final}
\begin{aligned}
&\quad \sup_{\tau\in I_{j,t}^\ell}\|Z_j(\tau)\|_{L^2}
+\int_{I_{j,t}^\ell}
\left\|\left(\frac{W_j}{b(\tau)},2^{2j}b(\tau)Z_{1,j},2^jZ_{2,j}\right)\right\|_{L^2}
\,d\tau\\
&\qquad+\|(b(\tau)^{\frac{1}{2}}2^jZ_j,b(\tau)^{-\frac{1}{2}}Z_{2,j})\|_{L^2(I_{j,t}^\ell;L^2)}
\\
&\lesssim
\|W_j(0), Z_{1,j}(0)\|_{L^2}+\sum_{k=1}^d\|\dot\Delta_j N_2^k(Z)\|_{L^{\infty}(I_{j,t}^\ell;L^2)}+\|\dot\Delta_j Q(Z)\|_{L^{\infty}(I_{j,t}^\ell;L^2)}
\\
&\qquad
+ \,\int_{I_{j,t}^\ell}
\Big(
b(\tau)\,2^{2j}\sum_{k=1}^{d}
\|\dot\Delta_j N_2^k(Z)(\tau)\|_{L^2}
+ 2^{j}\|\dot\Delta_j Q(Z)(\tau)\|_{L^2}
\Big)\,d\tau \\
&\qquad+ \int_{I_{j,t}^\ell}
\Big(
\|f_{1,j}(\tau)\|_{L^2}
+ \|f_{2,j}(\tau)\|_{L^2}
\Big)\,d\tau .
\end{aligned}
\end{equation}
Then, multiplying \eqref{dyadic-coupled-final} by $2^{j\frac d2}$ and summing over all
low-frequency indices, we recover the low-frequency part of
\eqref{est:W}.

\subsubsection*{Step 2: Estimates of nonlinear terms} 

We are in a position to handle the nonlinear terms on the right-hand side of \eqref{est:W}. We claim 
\begin{align}
\|W_0\|^{\ell}_{\dot{\B}^{\frac{d}{2}}_{2,1}}&\lesssim \big(1+\|Z_0\|_{\dot{\B}^{\frac{d}{2}}_{2,1}}\big)\big(\|Z_0\|_{\dot{\B}^{\frac{d}{2}}_{2,1}}^{\ell}+\|Z_0\|_{\dot{\B}^{\frac{d}{2}+1}_{2,1}}^{h}\big), \label{N1}\\
\|N_2^k(Z)\|_{\widetilde{L}^{\infty}_t(\dot\B_{2,1}^{\frac d2})}^{\ell}&\lesssim \varepsilon_0 \Big(\|Z\|_{\widetilde{L}^{\infty}_t(\dot{\B}^{\frac{d}{2}}_{2,1})}^\ell+  \|b(\tau)Z\|_{\widetilde{L}^{\infty}_t(\dot{\B}^{\frac{d}{2}+1}_{2,1})}^h\Big), \label{N2} \\
\|b(\tau)N_2^k(Z)\|_{L^1_t(\dot{\B}^{\frac{d}{2}+2}_{2,1})}^{\ell}&\lesssim \varepsilon_0\Big( \|b(\tau)Z\|_{L^1_t(\dot{\B}^{\frac{d}{2}+2}_{2,1})}^{\ell}+\|Z\|_{L^1_t(\dot{\B}^{\frac{d}{2}+1}_{2,1})}^{h}\Big),  \label{N3}\\
\|Q(Z)\|_{\widetilde{L}^{\infty}_t(\dot\B_{2,1}^{\frac d2})}^{\ell}&\lesssim  \varepsilon_0\|Z_2\|_{\widetilde{L}^{\infty}_t(\dot\B_{2,1}^{\frac d2})} , \label{N5}\\
\|Q(Z)\|_{L^1_t(\dot{\B}^{\frac{d}{2}+1}_{2,1})}^{\ell}&\lesssim \varepsilon_0 \int_0^t \Big(\|(\frac{W}{b(\tau)},\nabla Z_2)\|_{\dot{\B}^{\frac{d}{2}}_{2,1}}+b(\tau)\|Z_1(\tau)\|_{\dot{\B}^{\frac{d}{2}+2}_{2,1}}^\ell+ \|Z_1(\tau)\|_{\dot{\B}^{\frac{d}{2}+1}_{2,1}}^h\Big)\,d\tau,  \label{N4}\\
\|(f_1,f_2)\|_{L^1_t(\dot{\B}^{\frac{d}{2}}_{2,1})}^{\ell}&\lesssim  \varepsilon_0  \int_0^t \Big(\|(\frac{W}{b(\tau)},\nabla Z_2)\|_{\dot{\B}^{\frac{d}{2}}_{2,1}}+\|\partial_t Z\|_{\dot{\B}^{\frac{d}{2}}_{2,1}}\nonumber\\
&\quad\quad+b(\tau)\|Z_1(\tau)\|_{\dot{\B}^{\frac{d}{2}+2}_{2,1}}^\ell+ \|Z_1(\tau)\|_{\dot{\B}^{\frac{d}{2}+1}_{2,1}}^h\Big)\,d\tau+\varepsilon_0\|Z\|_{\widetilde{L}^{\infty}_t(\dot\B_{2,1}^{\frac d2})}. \label{N6}
\end{align}
 
The estimates \eqref{N1} and \eqref{N2} come directly from the definition, the basic product law and Lemma \ref{lemma51}.

For \eqref{N3}, we have in fact
\begin{equation}
\begin{aligned}
\|b(\tau)N_2^k(Z)\|_{L^1_t(\dot{\B}^{\frac{d}{2}+2}_{2,1})}^{\ell}&\lesssim \int_0^t  \|Z\|_{L^{\infty}} b(\tau)\Big( \|Z\|_{\dot{\B}^{\frac{d}{2}+2}_{2,1}}^{\ell}+\frac{1}{b(\tau)}\|Z\|_{\dot{\B}^{\frac{d}{2}+1}_{2,1}}^h\Big)d\tau\\
&\lesssim \|Z\|_{L^{\infty}_t(L^{\infty})}\Big( \|b(\tau)Z\|_{L^1_t(\dot{\B}^{\frac{d}{2}+2}_{2,1})}^{\ell}+\|Z\|_{L^1_t(\dot{\B}^{\frac{d}{2}+1}_{2,1})}^{h}\Big).
\notag
\end{aligned}
\end{equation}
By \eqref{conditionAVbar}, we have $Q(Z_1,0)=0$. Applying Lemma~\ref{lem:factorization-Z2}, we infer that there exists a smooth
matrix-valued map $q_1$ with $q_1(0)=0$ such that $Q(Z)= q_1(Z) Z_2$. Then \eqref{N5} follows immediately. In view of \eqref{eq:prod1}, 
\begin{equation}\label{Qz:n}
\begin{aligned}
\|Q(Z)\|_{L^1_t(\dot{\B}^{\frac{d}{2}+1}_{2,1})}^{\ell}&\lesssim \int_0^t\Big( \|q_1(Z)\|_{L^{\infty}}\|Z_2\|_{\dot{\B}^{\frac{d}{2}+1}_{2,1}}+\|q_1(Z)\|_{\dot{\B}^{\frac{d}{2}+1}_{2,1}} \|Z_2\|_{L^{\infty}}\Big)d\tau\\
&\lesssim \|Z\|_{L^{\infty}_t(L^{\infty})}\int_0^t \|Z_2\|_{\dot{\B}^{\frac{d}{2}+1}_{2,1}}\,d\tau+\int_0^t \|Z_2\|_{L^{\infty}} \|Z\|_{\dot{\B}^{\frac{d}{2}+1}_{2,1}}\,d\tau.
\end{aligned}
\end{equation}

\iffalse\begin{align*}
\|Z_2\|_{L^{\infty}}&\lesssim \|Z_2\|_{\dot{\B}^{\frac{d}{2}}_{2,1}}^{\ell}+\|Z_2\|_{\dot{\B}^{\frac{d}{2}}_{2,1}}^h
\lesssim  \|Z_2\|_{\dot{\B}^{\frac{d}{2}}_{2,1}}^{\ell}+b(t)\|Z_2\|_{\dot{\B}^{\frac{d}{2}+1}_{2,1}}^h
\end{align*}
and\fi

Plugging \eqref{N2} and \eqref{N5} into \eqref{Z2j-nl-abs} yields
\begin{align*}
\|Z_2(\tau)\|_{\dot{\B}_{2,1}^{\frac{d}{2}}}
\lesssim
\|W(\tau)\|_{\dot{\B}_{2,1}^{\frac{d}{2}}}
+b(\tau)\|Z(\tau)\|_{\dot{\B}_{2,1}^{\frac{d}{2}+1}}
+\varepsilon_0\|Z_2(\tau)\|_{\dot{\B}_{2,1}^{\frac{d}{2}}},
\end{align*}
and $\varepsilon_0$ is suitably small, we infer
\begin{align*}
\|Z_2(\tau)\|_{L^{\infty}}
\lesssim \|W(\tau)\|_{L^{\infty}}+b(\tau)\|\nabla Z(\tau)\|_{L^{\infty}},
\end{align*}
which implies 
\begin{align}\label{Dz:n1}
&\quad\int_0^t \|Z_2(\tau)\|_{L^{\infty}} \|Z(\tau)\|_{\dot{\B}^{\frac{d}{2}+1}_{2,1}}\nonumber\\
&\lesssim \int_0^t \|W\|_{\dot{\B}^{\frac{d}{2}}_{2,1}}\|Z(\tau)\|_{\dot{\B}^{\frac{d}{2}+1}_{2,1}}\,d\tau+\int_0^t b(\tau) \|Z\|_{\dot{\B}^{\frac{d}{2}+1}_{2,1}}^2\,d\tau\nonumber\\
&\lesssim \sup_{\tau\in[0,t]}b(\tau)\|Z(\tau)\|_{\dot{\B}^{\frac{d}{2}+1}_{2,1}}\int_0^t \Big(\|\frac{W}{b(\tau)}\|_{\dot{\B}^{\frac{d}{2}}_{2,1}}+\|Z_2\|_{\dot{\B}^{\frac{d}{2}+1}_{2,1}}\Big)\,d\tau\nonumber\\
&\quad+\sup_{\tau\in [0,t]}\|Z_1(\tau)\|_{\dot{\B}^{\frac{d}{2}}_{2,1}} \int_0^t b(\tau)\|Z_1(\tau)\|_{\dot{\B}^{\frac{d}{2}+2}_{2,1}}^\ell\,d\tau\nonumber\\
&\quad+\sup_{\tau\in [0,t]}b(\tau)\|Z_1(\tau)\|_{\dot{\B}^{\frac{d}{2}+1}_{2,1}}\int_0^t \|Z_1(\tau)\|_{\dot{\B}^{\frac{d}{2}+1}_{2,1}}^h\,d\tau\nonumber\\
&\lesssim \varepsilon_0 \int_0^t \Big(\|\frac{W}{b(\tau)}\|_{\dot{\B}^{\frac{d}{2}}_{2,1}}+\|Z_2\|_{\dot{\B}^{\frac{d}{2}+1}_{2,1}}+b(\tau)\|Z_1(\tau)\|_{\dot{\B}^{\frac{d}{2}+2}_{2,1}}^\ell+ \|Z_1(\tau)\|_{\dot{\B}^{\frac{d}{2}+1}_{2,1}}^h\Big)\,d\tau,
\end{align}
where one has used interpolation. By \eqref{Qz:n} and \eqref{Dz:n1}, we obtain \eqref{N4}.

Furthermore, one observes that
\begin{equation}\label{f2}
\begin{aligned}
\|f_2\|_{L^1_t(\dot{\B}^{\frac{d}{2}}_{2,1})}^{\ell}
\lesssim
\sum_{k=1}^d\|N_1^k(Z)\|_{L^1_t(\dot{\B}^{\frac{d}{2}+1}_{2,1})}^{\ell}
+\sum_{k=1}^d\|b(\tau)N_2^k(Z)\|_{L^1_t(\dot{\B}^{\frac{d}{2}+2}_{2,1})}^{\ell}
+\|Q(Z)\|_{L^1_t(\dot{\B}^{\frac{d}{2}+1}_{2,1})}^{\ell},
\end{aligned}
\end{equation}
where the terms on the right-hand side have been addressed except for $N_1$.

It follows from Lemma~\ref{lem:factorization-Z2} and \eqref{eq:prod1} that
\begin{align*}
\|N_1^k(Z)\|_{L^1_t(\dot{\B}^{\frac{d}{2}+1}_{2,1})}^{\ell}&\lesssim \|Z\|_{L^{\infty}_t(L^{\infty})} \|Z_2\|_{L^1_t(\dot{\B}^{\frac{d}{2}+1}_{2,1})}+\|b(\tau)^{\frac{1}{2}}Z\|_{L^2_t(\dot{\B}^{\frac{d}{2}+1}_{2,1})}\|b(\tau)^{-\frac{1}{2}}Z_2\|_{L^2_t(L^{\infty})}\\
&\lesssim \|Z\|_{L^{\infty}_t(\dot{\B}^{\frac{d}{2}}_{2,1})} \|Z_2\|_{L^1_t(\dot{\B}^{\frac{d}{2}+1}_{2,1})}+ \|b(\tau)^{\frac{1}{2}}Z\|_{L^2_t(\dot{\B}^{\frac{d}{2}+1}_{2,1})}\|b(\tau)^{-\frac{1}{2}}Z_2\|_{L^2_t(\dot{\B}^{\frac{d}{2}}_{2,1})}\lesssim \varepsilon_0 \|Z\|_{\mathcal{E}_t}. 
\end{align*}

Finally, we obtain 
%this term can be bounded similarly to $\|Q(Z)\|_{L^1_t(\dot{\B}^{\frac{d}{2}+1}_{2,1})}^{\ell}
\begin{equation}\label{f1}
\begin{aligned}
\|f_1\|_{L^1_t(\dot{\B}^{\frac{d}{2}}_{2,1})}^{\ell}
&\lesssim
\sum_{k=1}^d \|b'(\tau) N_2^k(Z)\|_{L^1_t(\dot{\B}^{\frac{d}{2}+1}_{2,1})}^{\ell}
+\sum_{k=1}^d\|b(\tau)\partial_Z N_{2}^k(Z) \partial_t Z\|_{L^1_t(\dot{\B}^{\frac{d}{2}+1}_{2,1})}^{\ell}\\
&\quad+\|\partial_t q_1(Z) Z_2\|_{L^1_t(\dot{\B}^{\frac{d}{2}}_{2,1})}^{\ell}+\|q_1(Z)\partial_t Z_2\|_{L^1_t(\dot{\B}^{\frac{d}{2}}_{2,1})}^{\ell}
\end{aligned}
\end{equation}
where
\begin{equation}
\begin{aligned}
\|b'(\tau) N_2^k(Z)\|_{L^1_t(\dot{\B}^{\frac{d}{2}+1}_{2,1})}^{\ell}
&\lesssim \|N_2^k(Z)\|_{\widetilde{L}^{\infty}_t(\dot{\B}^{\frac{d}{2}}_{2,1})}
\lesssim \|Z\|_{\widetilde{L}^{\infty}_t(\dot{\B}^{\frac{d}{2}}_{2,1})}^2,\\
\|b(t)\partial_Z N_2^k(Z) \partial_tZ\|_{L^1_t(\dot{\B}^{\frac{d}{2}+1}_{2,1})}^{\ell}&\lesssim \|\partial_Z N_2^k(Z) \partial_tZ\|_{L^1_t(\dot{\B}^{\frac{d}{2}}_{2,1})}^{\ell}\lesssim \|Z\|_{L^{\infty}_t(\dot{\B}^{\frac{d}{2}}_{2,1})}\|\partial_tZ\|_{L^1_t(\dot{\B}^{\frac{d}{2}}_{2,1})},\\
\|\partial_t q_1(Z) Z_2\|_{L^1_t(\dot{\B}^{\frac{d}{2}}_{2,1})}^{\ell}&\lesssim \|Z_2\|_{\widetilde{L}^{\infty}_t(\dot{\B}^{\frac{d}{2}}_{2,1})} \|\partial_tZ\|_{L^1_t(\dot{\B}^{\frac{d}{2}}_{2,1})},\\
\|q_1(Z)\partial_t Z_2\|_{L^1_t(\dot{\B}^{\frac{d}{2}}_{2,1})}^{\ell}
&\lesssim
\|Z\|_{\widetilde{L}^{\infty}_t(\dot{\B}^{\frac{d}{2}}_{2,1})}
\left\|\frac{W}{b(\tau)}\right\|_{L^1_t(\dot{\B}^{\frac{d}{2}}_{2,1})}^{\ell}.
\end{aligned}
\end{equation}

By \eqref{est:W} and \eqref{N1}--\eqref{N6}, we end up with \eqref{eqlq} and finish the proof of Lemma \ref{eslq}.
\end{proof}

\subsection{High-frequency analysis in the under-damped case}

    We establish the following estimates in the high-frequency regime.  
     
              \begin{lem}
    \label{eslq:high}      
 For all $t\in[0,T]$, we have
\begin{align}
\label{eqlq:h1}
&\|b(\tau)Z\|^h_{\widetilde{L}^{\infty}_t(\dot\B^{\frac{d}{2}+1}_{2,1})}
+\int_{0}^t\Big(
\|Z(\tau)\|^h_{\dot\B_{2,1}^{\frac{d}{2}+1}}
+\|\partial_t Z(\tau)\|_{\dot\B_{2,1}^{\frac{d}{2}}}^h
\Big)\,d\tau\notag\\
\lesssim&
\big(1+\|Z_0\|_{\dot{\B}^{\frac{d}{2}}_{2,1}}\big)\big(\|Z_0\|_{\dot{\B}^{\frac{d}{2}}_{2,1}}^{\ell}+\|Z_0\|_{\dot{\B}^{\frac{d}{2}+1}_{2,1}}^{h}\big)+\varepsilon_0\|Z\|_{\mathcal{E}_t}.
\end{align}
    \end{lem}

The proof of Lemma \ref{eslq:high} consists of the following three steps.

 \subsubsection*{Step 1: Construction of Lyapunov functionals}

 First, we construct spectrally localized Lyapunov functionals as follows.    As we mentioned before, we shall use the method of Lyapunov functionals for the entropy variable $U=U(V)$ defined in \eqref{w}.
 
 \begin{lem}\label{lemma6}
Let $U$ solve \eqref{entropeq}--\eqref{rtildeU} and set $\widetilde{U}=U-\bar{U}$.
For any $j\in \Z$ and $\tau\in I_{j,t}^h$, there exists a Lyapunov functional $\mathcal{L}_{j}(\tau)$ satisfying
\[
\mathcal{L}_{j}(\tau)\sim b(\tau+t^*)\|\widetilde{U}_j(\tau)\|_{L^2}^2
\]
and
 \begin{align}
 &\frac{d}{d\tau}\mathcal{L}_{j}(\tau)+\frac{c}{b(\tau)} \mathcal{L}_{j}(\tau)\lesssim b(\tau+t^*)\Big(\|\nabla U\|_{L^\infty}\|\widetilde{U}_j\|_{L^2}+\|(R_j^1, R_j^2, 2^{-j}b(\tau)^{-1}R_j^3)\|_{L^2}\Big) \sqrt{\frac{\mathcal{L}_{j}(\tau)}{b(\tau+t^*)}},\label{Ly:nonlinear}
 \end{align}
with $c>0$ a universal constant and
\begin{align*}
R_j^1&=\frac{\ddj \mathbf{r}(U)}{b(\tau)},\\
R_j^2&=[\widetilde{A}^0(U),\dot{\Delta}_j]\partial_t \widetilde{U}+\sum_{k=1}^d[\widetilde{A}^k(U),\dot{\Delta}_j]\partial_k \widetilde{U},\\
R_j^3&=-\dot{\Delta}_j\Big(\big( \widetilde{A}^0(U)-\Bar{A}^0\big)\partial_t \widetilde{U}\Big)-\sum_{k=1}^d \dot{\Delta}_j\Big(\big( \widetilde{A}^k(U)-\Bar{A}^k\big)\partial_{x_k} \widetilde{U}\Big),
\end{align*}
where $t^*$ is defined in \eqref{b'condi}
 \end{lem}

 \begin{proof}
 For the energy estimate, we write the system \eqref{entropeq}--\eqref{rtildeU} for $\widetilde{U}_j\triangleq \dot{\Delta}_j \widetilde{U}$ as follows: 
\begin{align}
  \widetilde{A}^0(U) \partial_\tau \widetilde{U}_j
+ \sum_{k=1}^d \widetilde{A}^k(U)\partial_k \widetilde{U}_j
+\frac{\bar B\widetilde{U}_j}{b(\tau)}
=R_j^1+R_j^2.
    \label{eqhigh}
\end{align}
Then, since $\widetilde{A}^k(U)$ is symmetric, taking the $L^2(\R^d,\R^n)$ scalar product with $\widetilde U_j$ and integrating by parts yields
\begin{equation}
\begin{aligned}\label{basicenergy}
    &\frac{d}{d\tau}\langle\widetilde{A}^0(U) \widetilde{U}_j,\widetilde{U}_j\rangle_{L^2}+\frac{1}{b(\tau)}\langle\bar B\widetilde{U}_j,\widetilde{U}_j\rangle_{L^2}\\
   & \lesssim \|\nabla U\|_{L^\infty}\|\widetilde{U}_j\|_{L^2}^2+\mc{(R_j^1,R_j^2)}\mc{\widetilde{U}_j}.
\end{aligned}
\end{equation}
As in the linear estimate, we introduce the corrector $\sI(z)$ defined in \eqref{corrector} and rewrite \eqref{eqhigh} by
\begin{align}
\Bar{A}^0 \partial_\tau \widetilde{U}_j
+ \sum_{k=1}^d \Bar{A}^k\partial_k \widetilde{U}_j
+\frac{\bar B\widetilde{U}_j}{b(\tau)}
=R_j^1+R_j^3.
    \label{eqhigh:1}
\end{align}
 Adapting the computation on linear analysis to the case with a non-zero source term, we get
\begin{equation}\label{low}
\begin{aligned}
&\frac{d}{dt}(\sum_{k=1}^{n-1}\varepsilon_k\Re(BA^{k-1} \widehat{\widetilde{U}}_j\cdot BA^{k}\widehat{\widetilde{U}}_j))+\frac{r}{2}\sum_{k=1}^{n-1}\varepsilon_k|BA^{k}\widehat{\widetilde{U}}_j|^2\\
&\leqslant  \bigg(\frac{1}{4b^2(\tau)r}+\frac{r}{4}+\frac{1}{4b(\tau)}\bigg)|B\widehat{\widetilde{U}}_j|^2\\
&\quad +\sum_{k=1}^{n-1}\varepsilon_k\left(\Re(BA^{k-1} \widehat{\widetilde{U}}_j \cdot BA^{k} (\Bar{A}^0)^{-1}(\widehat{R_j^1+R_j^3}))+\Re(BA^{k-1}(\Bar{A}^0)^{-1}(\widehat{R_j^1+R_j^3})  \cdot BA^{k}\widehat{\widetilde{U}}_j)\right) .
\end{aligned}
\end{equation}
Let us define
\begin{align*}
    \sL_{j}(t)=b(t+t^*)\langle\widetilde{A}^0(U) \widetilde{U}_j,\widetilde{U}_j\rangle_{L^2}+\sum_{k=1}^{n-1}\varepsilon_k\Re\int_{\R^d}(BA^{k-1}\widehat{\widetilde{U}}_j(\xi))\cdot(BA^{k}\widehat{\widetilde{U}}_j(\xi))|\xi|^{-1}d\xi.
\end{align*}
One notes $\underline{a} {\rm Id}\leq \widetilde{A}^0(U)\leq \bar{a}{\rm Id}$ with two positive constants $\underline{a}, \bar{a}$ due to \eqref{small}. As in the linear case, combining \eqref{basicenergy} and \eqref{low}, and using the Cauchy-Schwarz inequality to bound the term involving $R_j^1$ and $R_j^3$, we infer \eqref{Ly:nonlinear}.
\end{proof}

\begin{rem}
Here we observe that $c$ depends only on the coefficients of the system and is independent of $b(t)$.
\end{rem}

%Since there is a time-dependent weight in the definition of the functional, we shall use the following modified energy inequality:
%\begin{align}
 %   \frac{1}{2}\frac{d}{dt}\mc{2^{k_0}b(\tau)Z_j}^2+2^{k_0}\mc{B Z_j}^2\leqslant \frac{1}{2}\mc{2^{k_0}b'(\tau)Z_j}^2
%+C 2^{k_0}b(\tau) \mc{R_j^1}\mc{Z_j}.\label{moeneeqhigh}
%\end{align}

%\subsubsection*{Step 2: Cross estimates}
%We note $Z_j\triangleq\dot{\Delta}_jZ$ and we write the localized equation as 
%\begin{align}
%    \label{eqloc}
 %   \partial_t Z_j+\sum_{k=1}^dA^k\partial_k Z_j+\frac{BZ_j}{b(t)}=F_j\triangleq-\sum_{k=1}^d\widetilde{A}^k(Z)\partial_kZ_j+\sum_{k=1}^d[\widetilde{A}_k(Z),\ddj]\partial_k Z.
%\end{align}

%$b(t+t_*)$

\subsubsection*{Step 2: Estimate of $U$}

\iffalse
We now establish the estimate for $U$ using the Lyapunov inequalities obtained in Lemma \ref{lemma6}. We need to carry out energy estimates with the weight $b(t)$; however, one must overcome the additional $b'(t)$ term. In the case $b(t)=K^{-1}(1+t)^{\alpha}$ with $\alpha\leq 1$, it can be absorbed by the dissipation when $c-\alpha (1+t)^{\alpha-1}\gtrsim 1$, which requires $K>>1$ for $\alpha=1$ and $t>>1$ for $\alpha<1$. Here, we can consider all time $t\geq0$ if the weight $(1+t)^{\alpha}$ is replaced by $(A+t)^{\alpha}$ with a suitably large constant $A$. In the same spirit, we consider general $b(t)$ as we did in the linear case. Similarly, we obtain
\begin{equation}
 \begin{aligned}
 &\frac{d}{dt}\Big(b(t+t^*)\mathcal{L}_{j}(t)\Big)+\frac{c^*}{2}\mathcal{L}_{j}(t)\\
 &\lesssim b(t+t^*)\Big(\|\nabla U\|_{L^\infty}\|\widetilde{U}_j\|_{L^2}+\|(R_j^1, R_j^2, 2^{-j}b(\tau)^{-1}R_j^3)\|_{L^2}\Big) \sqrt{\mathcal{L}_{j}(t)},
 \end{aligned}
\end{equation}
\fi

Taking $t^*>0$ sufficiently large if necessary, we assume that $b'(\tau+t^*)\leq \frac{c^*}{4}$ for $\tau\geq0$. Multiplying the localized Lyapunov inequality of Lemma~\ref{lemma6}
by $b(\tau+t^*)$, we get
\begin{equation}\label{weighted-localized-U}
\begin{aligned}
&\frac{d}{d\tau}\Big(b(\tau+t^*)\sL_j(\tau)\Big)
+\frac{c^*}{2}\sL_j(\tau)  \\
&\lesssim
b(\tau+t^*)\Big(
\|\nabla U\|_{L^\infty}\|\widetilde U_j\|_{L^2}
+\|(R_j^1,R_j^2,2^{-j}b(\tau)^{-1}R_j^3)\|_{L^2}
\Big)\sqrt{\sL_j(\tau)} .
\end{aligned}
\end{equation}
Since Lemma~\ref{lem2} applies to \eqref{weighted-localized-U} with $
f(\tau)=b(\tau+t^*)$, $g=1$ and $X(\tau)=\sqrt{\sL_j(\tau)}$, we infer that
\begin{equation}\label{highlocalized}
\begin{aligned}
&b(t+t^*)\|\widetilde U_j(t)\|_{L^2}
+\int_{t_j}^t\|\widetilde U_j(\tau)\|_{L^2}\,d\tau  \\
&\lesssim
b(t_j+t^*)\|\widetilde U_j(t_j)\|_{L^2}  \\
&\quad+
\int_{t_j}^t b(\tau+t^*)\Big(
\|\nabla U\|_{L^\infty}\|\widetilde U_j\|_{L^2}
+\|(R_j^1,R_j^2,2^{-j}b(\tau)^{-1}R_j^3)\|_{L^2}
\Big)\,d\tau .
\end{aligned}
\end{equation}
Since $b(t+t^*)$ fulfills $C_{t^*}^{-1}b(t)\leq  b(t+t^*)\leq C_{t^*} b(t)$ for some constant $C_{t^*}>0$,
we obtain that
 \begin{equation*}
\begin{aligned}
    &\quad \|b(\tau)\widetilde{U}\|_{\widetilde{L}^\infty_t(\dot\B_{2,1}^{\frac{d}{2}+1})}^h+\|\widetilde{U}\|_{{L}^1_t(\dot\B_{2,1}^{\frac{d}{2}+1})}^h\\
    &\lesssim \sum_{j\in \Z,t_j<t} b(t_j)2^{j(\frac d2+1)}\|\widetilde{U}_j(t_j)\|_{L^2}+\|b(\tau)\nabla U\|_{L^{\infty}_t(L^{\infty})}\|\widetilde{U}(\tau)\|_{{L}^1_t(\dot\B_{2,1}^{\frac{d}{2}+1})}^h\\
    &\quad+\int_{0}^t\sum_{j\geqslant J_\tau-1}
\left(
b(\tau)2^{j(\frac{d}{2}+1)}
\big(\|R_j^1\|_{L^2}+\|R_j^2\|_{L^2}\big)
+2^{\frac{jd}{2}}\|R^3_j\|_{L^2}
\right)\,d\tau.
\end{aligned}
\end{equation*}
By \eqref{conditionAVbar} and the block form $\mathbf Q(Z)=(0,Q(Z))^{\top}$, we have $\mathbf Q(Z_1,0)=0$. Moreover, the construction of the variable $Z$ gives  $S^{-1}\bar T^{\top}G(V(Z))=-BZ+\mathbf Q(Z)$ (see Appendix \ref{AppendixNF}).  Since $B\begin{pmatrix}Z_1\\0\end{pmatrix}=0$, it follows that $G(V(Z_1,0))=0$. Hence $V(Z_1,0)\in\mathcal E$, and the entropy condition implies that $U(Z_1,0)\in\mathcal M$. Since $\bar U\in\mathcal M$ and $\ker D_U\widetilde H(\bar U)=\mathcal M$, Definition \eqref{rtildeU} yields $\mathbf r(U(Z_1,0))=0$.  Therefore, Lemma~\ref{lem:factorization-Z2} implies that $\mathbf r(U(Z))=\widetilde q_1(Z)Z_2$ for some smooth matrix-valued map $\widetilde q_1$. Moreover, $\widetilde q_1(0)=0$ because $\mathbf r$ vanishes to second order at $\bar U$. Then, using the product law \eqref{eq:prod1} and the composition estimates in Lemma~\ref{compositionlp}, we deduce that
%Since $U=U(Z)$ depends on $Z$ smoothly and satisfies  $U(0)=0$, Lemma \ref{lem:factorization-Z2}  implies $\mathbf{r}(U(Z))=\widetilde{q}_1(Z) Z_2$ for some function $\widetilde{q}_1(Z)$ satisfying $\widetilde{q}_1(0)=0$. Then, using the product law \eqref{eq:prod1} and composite estimates in Lemma \ref{compositionlp}, we deduce that
\begin{equation}\label{N1h}
\begin{aligned}
&\quad\int_{0}^t\sum_{j\geqslant J_\tau-1}b(\tau)2^{j(\frac{d}{2}+1)}\mc{R_j^1}\,d\tau\\
&\lesssim \int_{0}^t\sum_{j\geqslant J_\tau-1} 2^{j(\frac d2 +1)}\|\dot{\Delta}_j\mathbf{r}(U)\|_{L^2}\\
&\lesssim \int_0^t  \Big(\|Z\|_{\dot{\B}^{\frac{d}{2}}_{2,1}}\|Z_2\|_{\dot{\B}^{\frac{d}{2}+1}_{2,1}} +\|Z\|_{\dot{\B}^{\frac{d}{2}+1}_{2,1}}\|Z_2\|_{\dot{\B}^{\frac{d}{2}}_{2,1}}\Big)\,d\tau\\
&\lesssim \|Z\|_{L^{\infty}_t(\dot{\B}^{\frac{d}{2}}_{2,1})}\|Z_2\|_{L^{1}_t(\dot{\B}^{\frac{d}{2}+1}_{2,1})}+\|b(\tau)^{\frac{1}{2}}Z\|_{L^{2}_t(\dot{\B}^{\frac{d}{2}+1}_{2,1})}\|b(\tau)^{-\frac{1}{2}}Z_2\|_{L^{2}_t(\dot{\B}^{\frac{d}{2}}_{2,1})}.
\end{aligned}
\end{equation}
Then, by the classical composite and commutator analysis (Lemma \ref{commhigh}), we have
\begin{equation}
\begin{aligned}\label{N2h}
\int_{0}^t\sum_{j\geqslant J_\tau-1}b(\tau)2^{j(\frac{d}{2}+1)}\mc{R_j^2}\,d\tau&\lesssim  \int_{0}^t b(\tau)\|\nabla Z\|_{\dot{\B}^{\frac{d}{2}}_{2,1}} (\|\partial_t Z\|_{\dot{\B}^{\frac{d}{2}}_{2,1}}+\|\nabla Z\|_{\dot{\B}^{\frac{d}{2}}_{2,1}})\,d\tau\\
&\lesssim  \|b(\tau)Z\|_{L^{\infty}_t(\dot{\B}^{\frac{d}{2}+1}_{2,1})} \|\partial_t Z\|_{L^{1}_t(\dot{\B}^{\frac{d}{2}}_{2,1})}+\|b(\tau)^{\frac{1}{2}}Z\|_{L^{2}_t(\dot{\B}^{\frac{d}{2}+1}_{2,1})}^2
\end{aligned}
\end{equation}
Then for $R_j^3$,  
\begin{equation}\label{N3h}
\begin{aligned}
&\quad\int_{0}^t\sum_{j\geqslant J_\tau-1}2^{j\frac{d}{2}}\mc{R_j^3}\,d\tau\\
&\lesssim \int_{0}^t\Big(\| \big( \widetilde{A}^0(U)-\Bar{A}^0\big)\partial_t \widetilde{U}\|_{\dot{\B}^{\frac{d}{2}}_{2,1}}^h+ \sum_{k=1}^d \|\big( \widetilde{A}^k(U)-\Bar{A}^k\big)\partial_{x_k} \widetilde{U}\|_{\dot{\B}^{\frac{d}{2}}_{2,1}}^h\Big)\,d\tau.
%&\lesssim\int_{0}^t \left(\| \widetilde{U}\|_{\dot{\B}^{\frac{d}{2}}_{2,1}} \|\partial_t \widetilde{U}\|_{\dot{\B}^{\frac{d}{2}}_{2,1}}+\| \widetilde{U}_2\|_{\dot{\B}^{\frac{d}{2}}_{2,1}}\|\nabla \widetilde{U}\|_{\dot{\B}^{\frac{d}{2}}_{2,1}}+\| \widetilde{U}\|_{\dot{\B}^{\frac{d}{2}}_{2,1}}\|\nabla \widetilde{U}_2\|_{\dot{\B}^{\frac{d}{2}}_{2,1}}\right)\,d\tau.
\end{aligned}
\end{equation}
As $\partial_t \widetilde{U}=D_{Z}\widetilde{U} \partial_t Z$, we have
\begin{align}
 \int_{0}^t\| \big( \widetilde{A}^0(U)-\Bar{A}^0\big)\partial_t \widetilde{U}\|_{\dot{\B}^{\frac{d}{2}}_{2,1}}^h&\lesssim (1+\|Z\|_{L^{\infty}_t(\dot{\B}^{\frac{d}{2}}_{2,1})})\|Z\|_{L^{\infty}_t(\dot{\B}^{\frac{d}{2}}_{2,1})} \|\partial_t Z\|_{L^1_t(\dot{\B}^{\frac{d}{2}}_{2,1})}.
\end{align}
Also, writing $F_1(Z)=\big( \widetilde{A}^k(U(Z))-\Bar{A}^k\big) D_{Z}\widetilde{U}(Z)$, we have
\begin{equation}
\begin{aligned}
&\quad\int_{0}^t
\|\big( \widetilde{A}^k(U)-\Bar{A}^k\big)\partial_{x_k} \widetilde{U}\|_{\dot{\B}^{\frac{d}{2}}_{2,1}}^h\,d\tau\\
&\lesssim
\int_{0}^t b(\tau)
\|F_1(Z)\partial_{x_k} Z^{\ell}\|_{\dot{\B}^{\frac{d}{2}+1}_{2,1}}^h \,d\tau
+\int_{0}^t
\|F_1(Z)\partial_{x_k} Z^{h}\|_{\dot{\B}^{\frac{d}{2}}_{2,1}}^h \,d\tau\\
&\lesssim
\|b(\tau)^{\frac{1}{2}} Z\|_{L^2_t(\dot{\B}^{\frac{d}{2}+1}_{2,1})}
\|b(\tau)^{\frac{1}{2}} \nabla Z^\ell \|_{L^2_t(\dot{\B}^{\frac{d}{2}}_{2,1})}\\
&\quad
+\|Z\|_{L^{\infty}_t(\dot{\B}^{\frac{d}{2}}_{2,1})}
\Big(
\|b(\tau)\nabla Z^\ell\|_{L^{1}_t(\dot{\B}^{\frac{d}{2}+1}_{2,1})}
+\|\nabla Z^h\|_{L^{1}_t(\dot{\B}^{\frac{d}{2}}_{2,1})}
\Big).
\end{aligned}
\end{equation}
For the first term, by the definition of $t_j$ we have
\begin{align*}
    \sum_{j\in \Z,t_j<t} b(t_j)2^{j(\frac d2+1)}\|\widetilde{U}_j(t_j)\|_{L^2}&\leqslant\sum_{j\in \Z,0<t_j<t} 2^{j\frac d2}\|\widetilde{U}_j(t_j)\|_{L^2}+\sum_{j\in \Z,t_j=0} b(0)2^{j(\frac d2+1)}\|\widetilde{U}_j(0)\|_{L^2}\\
    &\lesssim\sum_{j\in \Z,0<t_j<t} 2^{j\frac d2}\|\widetilde{U}_j(t_j)\|_{L^2}+\|\widetilde{U}_0\|_{\dot\B_{2,1}^{\frac{d}{2}+1}}^h.
\end{align*}
Since the map $Z\mapsto \widetilde U(Z)$ is smooth near $0$ and satisfies
$\widetilde U(0)=0$, standard composition estimates in Proposition~\ref{Composition} yield
\begin{align}
\|\widetilde{U}_0\|_{\dot\B_{2,1}^{\frac{d}{2}+1}}
&\lesssim \|Z_0\|_{\dot{\B}^{\frac{d}{2}+1}_{2,1}}.
\end{align}
To address $\sum_{j\in \Z,0<t_j<t} 2^{j\frac d2}\|\widetilde{U}_j(t_j)\|_{L^2}$, we recall that $U=(D_{V}\eta(V(Z)))^{\top}$, where $\eta(V)$ is convex. Thus,  $\widetilde{U}=\widetilde{U}(Z)$ defines a smooth local diffeomorphism around $Z=0$ and satisfies $\widetilde{U}(0)=0$. By Taylor's expansion for vector-valued functions, there exists a positive definite constant matrix $G_0$ and a function $G(Z)$ ($G(0)=0$) such that
\begin{align}
\widetilde{U}=G_0 Z+G(Z) Z,\label{wideU0Z}
\end{align}
In light of the decomposition \eqref{wideU0Z}, and standard product laws, as well as estimates for the composition function $G(Z)$, it holds 
\begin{equation}\label{104}
\begin{aligned}
\sum_{j\in \Z,0<t_j<t} 2^{j\frac d2}\|\widetilde{U}_j(t_j)\|_{L^2}
&\lesssim \|\widetilde U\|_{\widetilde{L}^{\infty}_t(\dot{\B}^{\frac{d}{2}}_{2,1})}^{\ell}\lesssim \|Z\|_{\widetilde{L}^{\infty}_t(\dot{\B}^{\frac{d}{2}}_{2,1})}^{\ell}+\| Z\|_{\widetilde{L}^{\infty}_t(\dot{\B}^{\frac{d}{2}}_{2,1})}^2\\
%&\lesssim \big(1+\|Z_0\|_{\dot{\B}^{\frac{d}{2}}_{2,1}}\big)\big(\|Z_0\|_{\dot{\B}^{\frac{d}{2}}_{2,1}}^{\ell}+\|Z_0\|_{\dot{\B}^{\frac{d}{2}+1}_{2,1}}^{h}\big)+\varepsilon_0 \|Z\|_{\mathcal{E}_t}.
\end{aligned}
\end{equation}

%the low-frequency estimate \eqref{eqlq} at hand

%Now if $t_j>0$ we integrate \eqref{basicenergy} over $[0,t_j]$, multiplying the resulting inequality with $2^{\frac{d}{2}j}$ and sum it over $j_t<j<j_0$ to arrive at
%\begin{align}\label{99}
%\sum_{j\in \Z,0<t_j<t} 2^{j\frac d2}\|\widetilde{U}_j(t_j)\|_{L^2}&\lesssim \|\widetilde{U}(0)\|_{\dot\B_{2,1}^{\frac{d}{2}}}^\ell+\int_0^{t}\Big( \|\nabla U\|_{L^{\infty}} \|\widetilde{U}(\tau)\|_{\dot\B_{2,1}^{\frac{d}{2}}}^\ell+\sum_{j\leqslant J_\tau} 2^{\frac{d}{2}j}\|(R_j^1,R_j^2)\|_{L^2}\Big)\,d\tau.
%\end{align}

Combining all the above, we obtain:
\begin{equation}\label{highW}
\begin{aligned}
    &\quad\|b(\tau)\widetilde{U}\|_{\widetilde{L}^\infty_t(\dot\B_{2,1}^{\frac{d}{2}+1})}^h+\|\widetilde{U}\|_{{L}^1_t(\dot\B_{2,1}^{\frac{d}{2}+1})}^h+ \|\partial_t \widetilde{U}\|_{L^1_t(\dot{\B}^{\frac{d}{2}}_{2,1})}^h\lesssim \|Z_0\|_{\dot\B_{2,1}^{\frac{d}{2}+1}}+\|Z\|_{\widetilde{L}^{\infty}_t(\dot{\B}^{\frac{d}{2}}_{2,1})}^{\ell}+\varepsilon_0 \|Z\|_{\mathcal{E}_t}.
\end{aligned}
\end{equation}

\subsubsection*{Step 3: Recovering the bounds for $Z$}

We now recover the desired bounds for $Z$ in high frequencies from the estimates for $\widetilde{U}$. 
Recalling the expression \eqref{wideU0Z}, we have
\begin{align*}
\|Z\|_{L^1_t(\dot\B_{2,1}^{\frac{d}{2}+1})}^h&\lesssim \|\widetilde{U}\|_{L^1_t(\dot\B_{2,1}^{\frac{d}{2}+1})}^h+\big\|G(Z) Z\big\|_{L^1_t(\dot\B_{2,1}^{\frac{d}{2}+1})}^h.
%\|b(\tau)Z(\tau)\|_{\widetilde{L}^\infty_t(\dot\B_{2,1}^{\frac{d}{2}+1})}^h\lesssim \|b(\tau)\widetilde{U}(\tau)\|_{\widetilde{L}^\infty_t(\dot\B_{2,1}^{\frac{d}{2}+1})}.
\end{align*}
Note that $g(Z)=G(Z)Z$ satisfies $g(0)=0$ and $D_Z g(0)=0$. Hence, the nonlinear term can be addressed in terms of \eqref{small} and Lemma \ref{compositionlp}:
\begin{align*}
\|G(Z)Z\|_{L^1_t(\dot\B_{2,1}^{\frac{d}{2}+1})}^h&\lesssim  
\|Z\|_{L^{\infty}_t(L^{\infty})}
\Big( \|b(\tau)Z\|_{L^1_t(\dot{\B}^{\frac{d}{2}+2}_{2,1})}^{\ell}+\|Z\|_{L^1_t(\dot{\B}^{\frac{d}{2}+1}_{2,1})}^h\Big)\\
&\lesssim \varepsilon_0 \Big( \|b(\tau)Z\|_{L^1_t(\dot{\B}^{\frac{d}{2}+2}_{2,1})}^{\ell}+\|Z\|_{L^1_t(\dot{\B}^{\frac{d}{2}+1}_{2,1})}^h\Big). 
\end{align*}
Similarly, it holds that
\begin{align*}
\|b(\tau)Z\|_{\widetilde{L}^\infty_t(\dot\B_{2,1}^{\frac{d}{2}+1})}^h&\lesssim \|b(\tau)\widetilde{U}\|_{\widetilde{L}^\infty_t(\dot\B_{2,1}^{\frac{d}{2}+1})}^h+\|b(\tau)G(Z)Z\|_{\widetilde{L}^\infty_t(\dot\B_{2,1}^{\frac{d}{2}+1})}^h\\
%&\lesssim \|b(\tau)\widetilde{U}(\tau)\|_{\widetilde{L}^\infty_t(\dot\B_{2,1}^{\frac{d}{2}+1})}^h+\big\|b(\tau)\big(g_1(Z)-g_1(0) \big)Z\big\|_{\widetilde{L}^\infty_t(\dot\B_{2,1}^{\frac{d}{2}+1})}\|Z\|_{\widetilde{L}^\infty_t(\dot\B_{2,1}^{\frac{d}{2}})}\\
%&\quad+\big\|\big(g_1(Z)-g_1(0) \big)Z\big\|_{\widetilde{L}^\infty_t(\dot\B_{2,1}^{\frac{d}{2}+1})}\|b(\tau) Z\|_{\widetilde{L}^\infty_t(\dot\B_{2,1}^{\frac{d}{2}+1})}\\
%&\lesssim \|b(\tau)\widetilde{U}(\tau)\|_{\widetilde{L}^\infty_t(\dot\B_{2,1}^{\frac{d}{2}+1})}^h+\varepsilon_0\Big( \|Z\|_{\widetilde{L}^\infty_t(\dot\B_{2,1}^{\frac{d}{2}})}^{\ell}+\|b(\tau) Z\|_{\widetilde{L}^\infty_t(\dot\B_{2,1}^{\frac{d}{2}+1})}^h\Big).
\end{align*}
where
\begin{align*}
\|b(\tau)G(Z)Z\|_{\widetilde{L}^\infty_t(\dot\B_{2,1}^{\frac{d}{2}+1})}^h%&\lesssim %\big\|b(\tau) G(Z) \big\|_{\widetilde{L}^\infty_t(\dot\B_{2,1}^{\frac{d}{2}+1})}\|Z\|_{\widetilde{L}^\infty_t(\dot\B_{2,1}^{\frac{d}{2}})}+\big\|G(Z)\big\|_{\widetilde{L}^\infty_t(\dot\B_{2,1}^{\frac{d}{2}+1})}\|b(\tau) Z\|_{\widetilde{L}^\infty_t(\dot\B_{2,1}^{\frac{d}{2}+1})}\\
%&\lesssim \varepsilon_0\Big( \|Z\|_{\widetilde{L}^\infty_t(\dot\B_{2,1}^{\frac{d}{2}})}+\|b(\tau) Z\|_{\widetilde{L}^\infty_t(\dot\B_{2,1}^{\frac{d}{2}+1})}\Big)\\
&\lesssim \varepsilon_0\Big( \|Z\|_{\widetilde{L}^\infty_t(\dot\B_{2,1}^{\frac{d}{2}})}^{\ell}+\|b(\tau) Z\|_{\widetilde{L}^\infty_t(\dot\B_{2,1}^{\frac{d}{2}+1})}^h\Big).
\end{align*}
For bounding the time derivative, using the system \eqref{sys} and the estimates for composite functions, we get 
\begin{equation*}
\begin{aligned}
    \|\partial_t Z\|_{L^1_t(\dot{\B}^{\frac{d}{2}}_{2,1})}^h&\lesssim \|Z\|_{L^1_t(\dot{\B}^{\frac{d}{2}+1}_{2,1})}^h+\left\| \frac{Z_2}{b(\tau)}\right\|_{L^1_t(\dot{\B}^{\frac{d}{2}}_{2,1})}^h+\sum_{k=1}^d \|N_1^k(Z)\|_{L^1_t(\dot{\B}^{\frac{d}{2}+1}_{2,1})}^h+\left\| \frac{Q(Z)}{b(\tau)}\right\|_{L^1_t(\dot{\B}^{\frac{d}{2}}_{2,1})}^h\\
    &\lesssim \|Z\|_{L^1_t(\dot{\B}^{\frac{d}{2}+1}_{2,1})}^h+\|Z\|_{\widetilde{L}^{\infty}_t(\dot{\B}^{\frac{d}{2}}_{2,1})} \|Z\|_{\mathcal{E}_t}.
\end{aligned}
\end{equation*}
Therefore, combining the above estimates with \eqref{highW} yields \eqref{eqlq:h1}.

\subsection{The over-damped case in low and high frequencies}
Similarly to the linear analysis, since $b'(\tau)<0$, we have for all $\tau\in I_{t,j}^h$
\begin{equation}
 \begin{aligned}
 &\frac{d}{d\tau}\Big(b(\tau)\mathcal{L}_{j}(\tau)\Big)+\frac{c}{2}\mathcal{L}_{j}(\tau)\\
 &\lesssim b(\tau)\Big(\|\nabla U\|_{L^\infty}\|\widetilde{U}_j\|_{L^2}+\|(R_j^1, R_j^2, 2^{-j}b(\tau)^{-1}R_j^3)\|_{L^2}\Big) \sqrt{\mathcal{L}_{j}(\tau)},
 \end{aligned}
\end{equation}
and then immediately
\begin{equation}
\begin{aligned}
    &\quad \|b(\tau)\widetilde{U}\|_{\widetilde{L}^\infty_t(\dot\B_{2,1}^{\frac{d}{2}+1})}^h+\|\widetilde{U}\|_{{L}^1_t(\dot\B_{2,1}^{\frac{d}{2}+1})}^h\\
    &\lesssim \sum_{j\in \Z,t_j>0} b(0)2^{j(\frac d2+1)}\|\widetilde{U}_j(0)\|_{L^2}+\|b(\tau)\nabla U\|_{L^{\infty}_t(L^{\infty})}\|\widetilde{U}(\tau)\|_{{L}^1_t(\dot\B_{2,1}^{\frac{d}{2}+1})}^h\\
    &\quad+\int_{0}^t\sum_{j\geqslant J_\tau-1}
\left(
b(\tau)2^{j(\frac{d}{2}+1)}
\big(\|R_j^1\|_{L^2}+\|R_j^2\|_{L^2}\big)
+2^{\frac{jd}{2}}\|R^3_j\|_{L^2}
\right)\,d\tau.
\end{aligned}
\end{equation}
While for the low-frequency part
\begin{equation}
\begin{aligned}
&\quad\|Z\|^\ell_{\widetilde{L}^{\infty}_t(\dot\B_{2,1}^{\frac{d}{2}})}
+\int_{0}^t\Big(\|b(\tau) Z_1\|_{\dot{\B}^{\frac{d}{2}+2}_{2,1}}^\ell +\|Z_2\|_{\dot{\B}^{\frac{d}{2}+1}_{2,1}}^\ell+
\left\|\frac{W}{b(\tau)}\right\|_{\dot{\B}^{\frac{d}{2}}_{2,1}}^\ell+\|\partial_t Z\|_{\dot{\B}^{\frac{d}{2}}_{2,1}}^\ell
\Big)\,d\tau
\\
&\lesssim\sum_{j\in\Z,\;0<t_j\le t}2^{j\frac d2}\|Z_j(t_j)\|_{L^2}
+\sum_{j\in\Z,\; t_j=0}2^{j\frac d2}\|Z_j(0)\|_{L^2}\\
&\quad+\|(f_1,f_2)\|_{L^1_t(\dot{\B}^{\frac{d}{2}}_{2,1})}^\ell
+\sum_{k=1}^{d}\|b(\tau)N_2^k(Z)\|_{\widetilde{L}^{\infty}_t(\dot\B_{2,1}^{\frac{d}{2}})\cap L^1_t(\dot{\B}^{\frac{d}{2}+2}_{2,1})}^\ell
+\|Q(Z)\|_{\widetilde{L}^{\infty}_t(\dot\B_{2,1}^{\frac{d}{2}})\cap L^1_t(\dot{\B}^{\frac{d}{2}+1}_{2,1})}^\ell,
\end{aligned}
\end{equation}
and the terms on the second line can be bounded as
\begin{align*}
&\quad\sum_{j\in\Z,\;0<t_j\le t}2^{j\frac d2}\|Z_j(t_j)\|_{L^2}
+\sum_{j\in\Z,\; t_j=0}2^{j\frac d2}\|Z_j(0)\|_{L^2}\\
&\lesssim\sum_{j\in\Z,\;0<t_j\le t}2^{j(\frac d2+1)}\|b(t_j)Z_j(t_j)\|_{L^2}
+\sum_{j\in\Z,\; t_j=0}2^{j\frac d2}\|Z_j(0)\|_{L^2}\\
&\lesssim
\|b(\tau)Z\|_{\widetilde{L}^\infty_t(\dot\B_{2,1}^{\frac d2+1})}^h
+\|Z_0\|_{\dot{\B}^{\frac d2}_{2,1}}^\ell
\lesssim
\|Z_0\|_{\dot{\B}^{\frac d2}_{2,1}}
+\|Z_0\|_{\dot{\B}^{\frac d2+1}_{2,1}}.
\end{align*}
The nonlinear estimates are identical to those in the under-damped case, which completes the proof of the a priori estimate.

\subsection{Proof of Theorem~\ref{thm0}}

We start with the following local well-posedness result.
\begin{prop}\label{localwell}
For any initial data $Z_0$ in the nonhomogeneous Besov space $\B^{\frac{d}{2}+1}_{2,1}$, there exists a positive time $T_1$, depending only on the coefficients of the matrices $A^j$, on $B$, on $Q$, and on $\|Z_0\|_{\B_{2,1}^{\frac{d}{2}+1}}$, such that System~\eqref{eq0} has a unique classical solution $Z$ with
\[
Z\in \sC^1([0,T_1]\times \R^d)\quad\text{and}\quad
Z\in \sC([0,T_1];\B_{2,1}^{\frac{d}{2}+1})\cap \sC^1([0,T_1];\B_{2,1}^{\frac{d}{2}}).
\]
Furthermore, if the maximal time of existence $T^*$ is finite, then
\[
\int_0^{T^*}\|\nabla Z(t)\|_{L^\infty}\,dt=+\infty.
\]
\end{prop}
The proof of the proposition follows from Chapter 4 of \cite{Bahouri2011}. Then we separate the proof of Theorem \ref{thm0} into several steps. 
\subsubsection*{Step 1: Construction of approximate solutions}
Fix the initial datum $Z_0\in \dot\B_{2,1}^{\frac{d}{2}}\cap \dot\B_{2,1}^{\frac{d}{2}+1}$ satisfying the smallness condition. We approximate it by
\[
Z_0^n\triangleq ({\rm Id}-\dot S_{-n})Z_0,\qquad n\geq1.
\]
By construction, $Z_0^n\in \B_{2,1}^{\frac{d}{2}+1}$. Consequently, Proposition~\ref{localwell} provides a unique maximal solution
\[
Z^n\in \sC([0,T_n^*);\B_{2,1}^{\frac{d}{2}+1})\cap \sC^1([0,T_n^*);\B_{2,1}^{\frac{d}{2}})
\]
associated with the data $Z_0^n$.

\subsubsection*{Step 2: Uniform estimates}

Note that, by the definition $Z_0^n=({\rm Id}-\dot S_{-n})Z_0$ and the boundedness of Littlewood--Paley operators on Besov spaces,
\[
\|Z^n_0\|_{\dot\B_{2,1}^{\frac{d}{2}}}+\|Z^n_0\|_{\dot\B_{2,1}^{\frac{d}{2}+1}}
\leq C_1 \|Z_0\|_{\dot\B_{2,1}^{\frac d2}\cap \dot\B_{2,1}^{\frac d2+1}}
\]
for some constant $C_1>0$ independent of $n$. Let $T_n^*\in(0,\infty]$ be the maximal existence time of the classical solution $Z^n$ given by Proposition~\ref{localwell}, and define
\[
T_n^{**}\triangleq \sup\Big\{\,t\in[0,T_n^*)\;:\; \|Z^n\|_{\mathcal E_t}\le 2C_0C_1 \|Z_0\|_{\dot\B_{2,1}^{\frac d2}\cap \dot\B_{2,1}^{\frac d2+1}}\,\Big\},
\]
where $C_0>0$ is the constant in Proposition~\ref{apriori}. On $[0,T_n^{**})$, the bootstrap assumption $X(Z^n)\le 2C_0C_1\|Z_0\|_{\dot\B_{2,1}^{\frac d2}\cap \dot\B_{2,1}^{\frac d2+1}}$ together with the a priori estimate in Proposition~\ref{apriori} yields, for $\|Z_0\|_{\dot\B_{2,1}^{\frac d2}\cap \dot\B_{2,1}^{\frac d2+1}}$ sufficiently small,
\[
\|Z^n\|_{\mathcal E_t}\le \frac32\,C_0C_1 \|Z_0\|_{\dot\B_{2,1}^{\frac d2}\cap \dot\B_{2,1}^{\frac d2+1}},\qquad t\in[0,T_n^{**}).
\]
In particular, this improves the bootstrap bound. Hence the standard  argument implies that $T_n^{**}=T_n^*$, and therefore
\[
\sup_{t\in[0,T]}\Big(\|Z^n(t)\|_{\dot\B_{2,1}^{\frac{d}{2}}}+\|b(t)Z^n(t)\|_{\dot\B_{2,1}^{\frac{d}{2}+1}}\Big)\le C\,\|Z_0\|_{\dot\B_{2,1}^{\frac d2}\cap \dot\B_{2,1}^{\frac d2+1}}
\quad\text{for all }T<T_n^*.
\]

Assume by contradiction that $T_n^*<\infty$. Since $b$ is continuous and strictly positive, due to the uniform estimate, there exist constants $0<m\le M<\infty$ independent of $n$ such that
\[
m\le b(t)\le M\qquad\text{for all }t\in[0,T_n^*).
\]
Hence the above uniform bound implies $Z^n\in L^\infty(0,T_n^*;\dot\B_{2,1}^{\frac{d}{2}+1})$. Using the critical embedding $\dot\B_{2,1}^{\frac{d}{2}+1}\hookrightarrow \dot{W}^{1,\infty}$, we obtain
\[
\|\nabla Z^n(t)\|_{L^\infty}\lesssim \|Z^n(t)\|_{\dot\B_{2,1}^{\frac{d}{2}+1}}
\quad\text{for a.e. }t\in(0,T_n^*),
\]
so that $\nabla Z^n\in L^1(0,T_n^*;L^\infty)$. This contradicts the blow-up criterion in Proposition~\ref{localwell}. Therefore, $T_n^*=\infty$ for all $n\ge1$, and we prove the global existence of approximate solutions $Z^n$ for $n\geq1$.

\subsubsection*{Step 3: Convergence}
\begin{prop}[Stability estimate]\label{stability}
Let $\widetilde{Z}=Z^1-Z^2$, where $Z^1$ and $Z^2$ are two solutions of \eqref{eq0} with respective initial data $Z_0^1$ and $Z_0^2$, satisfying the bounds encoded in the norm $X$ on $[0,T]$. If both $\|Z^1\|_{L^\infty(\dot\B_{2,1}^{\frac{d}{2}})}$ and $\|Z^2\|_{L^\infty(\dot\B_{2,1}^{\frac{d}{2}})}$ are smaller than a constant $c$, then for all $t\in[0,T]$,
\[
\|\widetilde{Z}\|_{L_t^\infty(\dot\B_{2,1}^{\frac{d}{2}})}
\lesssim_c
\|\widetilde{Z}_0\|_{\dot\B_{2,1}^{\frac{d}{2}}}
+\int_0^t\Big(\|(Z^1,Z^2)\|_{\dot\B_{2,1}^{\frac{d}{2}}}+\|(Z^1,Z^2)\|_{\dot\B_{2,1}^{\frac{d}{2}+1}}\Big)\,\|\widetilde{Z}\|_{\dot\B_{2,1}^{\frac{d}{2}}}\,d\tau.
\]
\end{prop}

\begin{proof}
Let $U^1$ and $U^2$ be the corresponding entropy variables, and set
\[
\delta U=U^1-U^2,\qquad V_i=V(U^i)\quad (i=1,2).
\]
Then $\delta U$ solves
\[
\partial_t \delta U+\sum_{k=1}^d A^k(V_1)\partial_k \delta U
=-\sum_{k=1}^d (A^k(V_1)-A^k(V_2))\partial_k U^2
-\frac{\bar B\delta U-(\mathbf r(U^1)-\mathbf r(U^2))}{b(t)}.
\]
Applying $\ddj$, taking the $L^2$ inner product with $\delta U_j$, integrating over $\R^d$, and using Lemma~\ref{lem2}, we get for all $j\in \Z$ and $t\in[0,T]$,
\begin{align*}
\|\delta U_j(t)\|_{L^2}+\kappa\int_0^t\frac{\|\delta U_{2,j}\|_{L^2}}{b(\tau)}\,d\tau
&\leqslant \|\delta U_{j,0}\|_{L^2}+\int_0^t\frac{\|\ddj(\mathbf r(U^1)-\mathbf r(U^2))\|_{L^2}}{b(\tau)}\,d\tau\\
&\quad+\int_{0}^t\sum_{k=1}^d\big\|[\ddj,A^k(V_1)]\partial_k \delta U\big\|_{L^2}\,d\tau\\
&\quad+\int_0^t\sum_{k=1}^d\big\|\ddj\big((A^k(V_1)-A^k(V_2))\partial_k U^2\big)\big\|_{L^2}\,d\tau.
\end{align*}
Multiplying this inequality by $2^{j\frac{d}{2}}$ and using the commutator estimate \eqref{eq:com1} yields
\begin{align*}
2^{j\frac{d}{2}}\|\delta U_j(t)\|_{L^2}
&\leqslant 2^{j\frac{d}{2}}\|\delta U_{j,0}\|_{L^2}
+\int_0^t\frac{2^{j\frac{d}{2}}\|\ddj(\mathbf r(U^1)-\mathbf r(U^2))\|_{L^2}}{b(\tau)}\,d\tau\\
&\quad+\int_{0}^t c_j\|\nabla U^1\|_{\dot\B_{2,1}^{\frac{d}{2}}}\,2^{j\frac{d}{2}}\|\delta U_j\|_{L^2}\,d\tau
+\int_0^t\sum_{k=1}^d 2^{j\frac{d}{2}}
\big\|\ddj\big((A^k(V_1)-A^k(V_2))\partial_k U^2\big)\big\|_{L^2}\,d\tau.
\end{align*}
By the product law and the composition estimate \eqref{eq:compo}, we have
\[
\|(A^k(V_1)-A^k(V_2))\partial_k U^2\|_{\dot\B_{2,1}^{\frac{d}{2}}}
\lesssim
\|\delta U\|_{\dot\B_{2,1}^{\frac{d}{2}}}\,\|\nabla U^2\|_{\dot\B_{2,1}^{\frac{d}{2}}},
\]
and similarly, since $b$ is positive and continuous, both $b$ and $1/b$ are bounded on compact time intervals,
\[
\Big\|\frac{\mathbf r(U^1)-\mathbf r(U^2)}{b(\tau)}\Big\|_{\dot\B_{2,1}^{\frac{d}{2}}}
\lesssim
\|\delta U\|_{\dot\B_{2,1}^{\frac{d}{2}}}\,\|(U^1,U^2)\|_{\dot\B_{2,1}^{\frac{d}{2}}}.
\]
Summing over $j\in\Z$ and using the fact that $Z$ is a smooth function of $U$ gives the desired estimate.
\end{proof}

The stability estimate above, together with $Z_0^n\to Z_0$ in $\dot\B_{2,1}^{\frac{d}{2}}$, implies that $(Z^n)_{n\in \N}$ is a Cauchy sequence in $L_T^\infty(\dot\B_{2,1}^{\frac{d}{2}})$ for all finite $T$, and thus converges to some limit $Z$ in that space. Using a diagonal argument, we further obtain $Z\in L^\infty(\R_+;\dot\B_{2,1}^{\frac{d}{2}})$. Since the regularity is high, passing to the limit in \eqref{eq0} is straightforward.

\subsubsection*{Step 4: Uniqueness}

 Let $Z^1$ and $Z^2$ be two solutions of \eqref{eq0} satisfying the bounds encoded in the norm $X$ on $[0,T]$. For all $T>0$, owing to the embedding $L_T^\infty\hookrightarrow L_T^1$ and the continuity of $b$, we have
\[
\int_0^T\Big(\|(Z^1,Z^2)\|_{\dot\B_{2,1}^{\frac{d}{2}}}+\|(Z^1,Z^2)\|_{\dot\B_{2,1}^{\frac{d}{2}+1}}\Big)\,dt<+\infty.
\]
Furthermore, $\|(Z^1,Z^2)\|_{L^\infty_T(\dot\B_{2,1}^{\frac{d}{2}})}$ is bounded. Combining Proposition~\ref{stability} with Gronwall's inequality yields that $Z^1\equiv Z^2$ on $[0,T]$. As $T$ is arbitrary, uniqueness follows globally in time.

\section{Proof of Theorem \ref{thm2} and Theorem \ref{thm3}}\label{sec:asymptotics}

\subsection{Time-weighted estimates}

We are now in a position to establish optimal time-decay estimates. Since $b$ is positive and continuous on $[0,t_b]$, up to modifying the constants $c_b$ and $C_b$, we may assume that \eqref{mu:infer} holds for all $t\geq0$. In the critical case $\alpha=1$, the constant $c_b$ is further assumed to be sufficiently large, in accordance with the smallness condition on $b'$ in Theorem~\ref{thm0}.
 %Without loss of generality, in what follows we just consider $b(t)=K(1+t)^{-\alpha}$.

\begin{prop}
Let $S\in C^1(\mathbb R_+)$ be an increasing function with $S(0)=0$. It holds that
\begin{align}
\|S(\tau)Z\|_{\mathcal{E}_t}
\lesssim
\int_0^t S'(\tau)\Big(
\|Z(\tau)\|_{\dot{\B}^{\frac{d}{2}}_{2,1}}^{\ell}
+b(\tau)\|Z(\tau)\|_{\dot{\B}^{\frac{d}{2}+1}_{2,1}}^h
\Big)\,d\tau.
\label{110}
\end{align}
Here, $\|\cdot\|_{\mathcal{E}_t}$ is defined by \eqref{X}.
\end{prop}

\begin{proof}
Multiplying \eqref{eqz1} and \eqref{weq} by $S(t)$, we obtain
\begin{equation}
\begin{aligned}
    \label{eqz1:t}
    &\partial_t(S(t) Z_1)+\sum_{k=1}^dA^k_{1,1}\partial_k(S(t) Z_1)
    -b(t)\sum_{k=1}^d\sum_{l=1}^dA^k_{1,2}L_2^{-1}A^l_{2,1}
    \partial_k\partial_l (S(t) Z_1)\\
    &=S'(t)Z_1+S(t)l_2+S(t)f_2,
\end{aligned}
\end{equation}
and
\begin{align}
    \partial_t (S(t)W)+\frac{L_2 (S(t)W)}{b(t)}
    =S'(t) W+S(t)l_1+S(t)f_1 .
    \label{weq:t}
\end{align}
Then, carrying out the same energy estimates on \eqref{eqz1:t}--\eqref{weq:t}
as in the proof of the low-frequency estimate and using the decoupling
argument as in \eqref{est:W}, we have
\begin{equation*}
\begin{aligned}
&\quad\|S(\tau) Z\|^\ell_{\widetilde{L}^{\infty}_t(\dot\B_{2,1}^{\frac d2})}
+\int_{0}^t S(\tau)\Big(
\|b(\tau) Z_1\|_{\dot{\B}^{\frac d2+2}_{2,1}}^\ell
+\|Z_2\|_{\dot{\B}^{\frac d2+1}_{2,1}}^\ell
+\left\|\frac{W}{b(\tau)}\right\|_{\dot{\B}^{\frac d2}_{2,1}}^\ell
+\|\partial_t Z\|_{\dot{\B}^{\frac d2}_{2,1}}^\ell
\Big)\,d\tau\\
&\quad
+\|S(\tau)b(\tau)^{\frac{1}{2}} Z\|_{\widetilde{L}^2_t(\dot{\B}^{\frac{d}{2}+1}_{2,1})}^{\ell}
+\|S(\tau)b(\tau)^{-\frac{1}{2}} Z_2\|_{\widetilde{L}^2_t(\dot{\B}^{\frac{d}{2}}_{2,1})}^{\ell}\\
&\lesssim
\|S'(\tau) (Z_1,W)\|^\ell_{L^{1}_t(\dot\B_{2,1}^{\frac d2})}
+\|S(\tau)(f_1,f_2)\|_{L^1_t(\dot{\B}^{\frac{d}{2}}_{2,1})}^\ell\\
&\quad
+\sum_{k=1}^{d}\|S(\tau)b(\tau)N_2^k(Z)\|_{\widetilde{L}^{\infty}_t(\dot\B_{2,1}^{\frac{d}{2}})
\cap L^1_t(\dot{\B}^{\frac{d}{2}+2}_{2,1})}^\ell+\|S(\tau)Q(Z)\|_{\widetilde{L}^{\infty}_t(\dot\B_{2,1}^{\frac{d}{2}})
\cap L^1_t(\dot{\B}^{\frac{d}{2}+1}_{2,1})}^\ell.
\end{aligned}
\end{equation*}
The analysis of the nonlinear terms follows the same lines as in the proof of
\eqref{N2}--\eqref{N6}. More precisely,
\begin{align*}
\|S(\tau) N_2^k(Z)\|_{\widetilde{L}^{\infty}_t(\dot\B_{2,1}^{\frac d2})}^{\ell}
&\lesssim
\|Z\|_{\mathcal{E}_t} \Big(
\|S(\tau) Z\|_{\widetilde{L}^{\infty}_t(\dot{\B}^{\frac{d}{2}}_{2,1})}^\ell
+\|S(\tau)b(\tau)Z\|_{\widetilde{L}^{\infty}_t(\dot{\B}^{\frac{d}{2}+1}_{2,1})}^h
\Big), \\
\|S(\tau)b(\tau)N_2^k(Z)\|_{L^1_t(\dot{\B}^{\frac{d}{2}+2}_{2,1})}^{\ell}
&\lesssim
\|Z\|_{\mathcal{E}_t}\Big(
\|S(\tau)b(\tau)Z\|_{L^1_t(\dot{\B}^{\frac{d}{2}+2}_{2,1})}^{\ell}
+\|S(\tau) Z\|_{L^1_t(\dot{\B}^{\frac{d}{2}+1}_{2,1})}^{h}
\Big),  \\
\|S(\tau) Q(Z)\|_{\widetilde{L}^{\infty}_t(\dot\B_{2,1}^{\frac d2})}^{\ell}
&\lesssim
\|Z\|_{\mathcal{E}_t}\|S(\tau) Z\|_{\widetilde{L}^{\infty}_t(\dot\B_{2,1}^{\frac d2})}, \\
\|S(\tau) Q(Z)\|_{L^1_t(\dot{\B}^{\frac{d}{2}+1}_{2,1})}^{\ell}
&\lesssim
\|Z\|_{\mathcal{E}_t}\int_0^t S(\tau) \Big(
\left\|\frac{W}{b(\tau)}\right\|_{\dot{\B}^{\frac{d}{2}}_{2,1}}
+\|\nabla Z_2\|_{\dot{\B}^{\frac{d}{2}}_{2,1}}\\
&\qquad\qquad\qquad\qquad
+b(\tau)\|Z_1(\tau)\|_{\dot{\B}^{\frac{d}{2}+2}_{2,1}}^\ell
+\|Z_1(\tau)\|_{\dot{\B}^{\frac{d}{2}+1}_{2,1}}^h
\Big)\,d\tau,  \\
\|S(\tau)(f_1,f_2)\|_{L^1_t(\dot{\B}^{\frac{d}{2}}_{2,1})}^{\ell}
&\lesssim
\|Z\|_{\mathcal{E}_t}\int_0^t S(\tau) \Big(
\left\|\frac{W}{b(\tau)}\right\|_{\dot{\B}^{\frac{d}{2}}_{2,1}}
+\|\nabla Z_2\|_{\dot{\B}^{\frac{d}{2}}_{2,1}}
+\|\partial_t Z\|_{\dot{\B}^{\frac{d}{2}}_{2,1}}\\
&\qquad\qquad\qquad\qquad
+b(\tau)\|Z_1(\tau)\|_{\dot{\B}^{\frac{d}{2}+2}_{2,1}}^\ell
+\|Z_1(\tau)\|_{\dot{\B}^{\frac{d}{2}+1}_{2,1}}^h
\Big)\,d\tau\\
&\quad
+\|Z\|_{\mathcal{E}_t}\|S(\tau) Z\|_{\widetilde{L}^{\infty}_t(\dot\B_{2,1}^{\frac d2})}.
\end{align*}
Note that the $\widetilde{L}^{\infty}_t(\dot\B_{2,1}^{\frac d2})$ estimates
of $N_2^k(Z)$ and $Q(Z)$ motivate the present time-weighted Lyapunov method,
since the standard differential energy method is not sufficient to control
these terms directly.

Moreover, we have
\[
\|S'(\tau) (Z_1,W)\|^\ell_{L^{1}_t(\dot\B_{2,1}^{\frac d2})}
\lesssim
(1+\|Z\|_{\mathcal{E}_t})\int_0^t S'(\tau)\Big(
\|Z(\tau)\|_{\dot{\B}^{\frac d2}_{2,1}}^{\ell}
+b(\tau)\|Z(\tau)\|_{\dot{\B}^{\frac{d}{2}+1}_{2,1}}^h
\Big)\,d\tau.
\]
Consequently, we establish the desired low-frequency estimate:
\begin{equation}\label{low:timeweight}
\begin{aligned}
&\|S(\tau) Z\|^\ell_{\widetilde{L}^{\infty}_t(\dot\B_{2,1}^{\frac d2})}
+\int_{0}^t S(\tau)\Big(
\|b(\tau) Z_1\|_{\dot{\B}^{\frac d2+2}_{2,1}}^\ell
+\|Z_2\|_{\dot{\B}^{\frac d2+1}_{2,1}}^\ell
+\left\|\frac{W}{b(\tau)}\right\|_{\dot{\B}^{\frac d2}_{2,1}}^\ell
+\|\partial_t Z\|_{\dot{\B}^{\frac d2}_{2,1}}^\ell
\Big)\,d\tau\\
&\quad
+\|S(\tau)b(\tau)^{\frac{1}{2}} Z\|_{\widetilde{L}^2_t(\dot{\B}^{\frac{d}{2}+1}_{2,1})}^{\ell}
+\|S(\tau)b(\tau)^{-\frac{1}{2}} Z_2\|_{\widetilde{L}^2_t(\dot{\B}^{\frac{d}{2}}_{2,1})}^{\ell}\\
&\lesssim
(1+\|Z\|_{\mathcal{E}_t})\int_0^t S'(\tau)\Big(
\|Z(\tau)\|_{\dot{\B}^{\frac d2}_{2,1}}^{\ell}
+b(\tau)\|Z(\tau)\|_{\dot{\B}^{\frac{d}{2}+1}_{2,1}}^h
\Big)\,d\tau
+\|Z\|_{\mathcal{E}_t}X(S(\tau) Z).
\end{aligned}
\end{equation}

It remains to address the time-weighted regularity estimates in high frequencies.
For the under-damped case, starting from the $b$-weighted high-frequency Lyapunov inequality in
Step~2, we have
\[
\frac{d}{dt}\Big(b(t+t^*)\mathcal L_j(t)\Big)
+c\mathcal L_j(t)
\lesssim
b(t+t^*)\Big(
\|\nabla U\|_{L^\infty}\|\widetilde U_j\|_{L^2}
+\|(R_j^1,R_j^2,2^{-j}b(t)^{-1}R_j^3)\|_{L^2}
\Big)\sqrt{\mathcal L_j(t)} .
\]
Multiplying the above inequality by $S(t)^2$, we get
\begin{equation}\label{Ly:nonlinear-timeweighted}
\begin{aligned}
&\frac{d}{dt}\Big(S(t)^2b(t+t^*)\mathcal L_j(t)\Big)
+cS(t)^2\mathcal L_j(t)\\
&\quad\lesssim
S(t)^2b(t+t^*)\Big(
\|\nabla U\|_{L^\infty}\|\widetilde U_j\|_{L^2}
+\|(R_j^1,R_j^2,2^{-j}b(t)^{-1}R_j^3)\|_{L^2}
\Big)\sqrt{\mathcal L_j(t)}\\
&\qquad
+S(t)S'(t)b(t+t^*)\mathcal L_j(t).
\end{aligned}
\end{equation}
The term involving $S'(t)$ is kept as a source term. Then, arguing as in
Steps~2--3 of the proof of Lemma~\ref{eslq:high} and applying Lemma~\ref{lem2} to \eqref{Ly:nonlinear-timeweighted} with $f(t)=b(t+t^*)$,
        $g(t)=1$ and $X(t)=S(t)\sqrt{\mathcal L_j(t)}$, one obtains
\begin{equation}\label{high:timeweight}
\begin{aligned}
&\|S(\tau)b(\tau)Z\|^h_{\widetilde{L}^{\infty}_t(\dot\B^{\frac{d}{2}+1}_{2,1})}
+\int_{0}^tS(\tau)\Big(
\|Z(\tau)\|^h_{\dot\B_{2,1}^{\frac{d}{2}+1}}
+\left\|\frac{W(\tau)}{b(\tau)}\right\|_{\dot{\B}^{\frac{d}{2}}_{2,1}}^h
+\|\partial_t Z(\tau)\|_{\dot\B_{2,1}^{\frac{d}{2}}}^h
\Big)\,d\tau\\
&\lesssim
(1+\|Z\|_{\mathcal{E}_t})\|S'(\tau)b(\tau)Z\|^h_{L^{1}_t(\dot\B^{\frac{d}{2}+1}_{2,1})}
+(1+\|Z\|_{\mathcal{E}_t})\|Z\|_{\mathcal{E}_t}\|S(t)Z\|_{\mathcal{E}_t}.
\end{aligned}
\end{equation}
The over-damped case can be addressed similarly. We omit the details for brevity.

Eventually, combining \eqref{low:timeweight}, \eqref{high:timeweight}, and
$\|Z\|_{\mathcal{E}_t}\lesssim \|Z_0\|_{\dot{\B}^{\frac d2}_{2,1}\cap \dot{\B}^{\frac d2+1}_{2,1}}\ll1$, we obtain \eqref{110}.
\end{proof}

\subsection{Gain of decay rates}\label{subsection:inter}

Now we are ready to prove the decay estimates \eqref{decay1} and \eqref{decay2}. The proof is divided into two cases.
\begin{itemize}
\item $-1<\alpha\leqslant1$.
\end{itemize}
We shall prove
\begin{align}\label{inter}
\|Z(t)\|_{\dot{\B}_{2,1}^{\frac d2}}^\ell
+\|b(t)Z(t)\|_{\dot{\B}_{2,1}^{\frac d2+1}}^h
\lesssim
(1+t)^{-\frac{1+\alpha}{2}\big(\frac d2-\sigma_1\big)}
\Big(
\|Z_0\|_{\dot{\B}_{2,1}^{\frac d2}\cap \dot{\B}_{2,1}^{\frac d2+1}}
+\|Z\|_{\widetilde L^{\infty}_t(\dot{\B}^{\sigma_1}_{2,\infty})}
\Big),
\end{align}
for $-1<\alpha\leqslant 1$.
Using the real interpolation inequality, we have
\begin{align}
\|Z\|_{\dot{\B}^{\frac d2}_{2,1}}^\ell
\lesssim
\big(\|Z\|_{\dot{\B}^{\sigma_1}_{2,\infty}}^\ell\big)^{\theta}
\big(\|Z\|_{\dot{\B}^{\frac d2+2}_{2,1}}^\ell\big)^{1-\theta},
\end{align}
where $\theta\in(0,1)$ is determined by
\[
\theta\sigma_1+(1-\theta)\Big(\frac d2+2\Big)=\frac d2 .
\]
Note that $b(\tau)\sim 1$ for $0\leq \tau\leq t_b$ and that \eqref{mu:infer} holds for $\tau\geq t_b$, with $t_b\gg1$ fixed. Set
\begin{align*}
a_1
&:=
\sup_{t\geq0}\Big(
\|Z(t)\|_{\dot{\B}^{\frac d2}_{2,1}}^\ell
+\|b(t)Z(t)\|_{\dot{\B}^{\frac d2+1}_{2,1}}^h
\Big),\quad a_2:=
\|Z\|_{\widetilde L^{\infty}_t(\dot{\B}^{\sigma_1}_{2,\infty})}.
\end{align*}
We also denote
\[
E^\ell(t):=\|Z(t)\|_{\dot{\B}^{\frac d2}_{2,1}}^\ell,
\qquad
E^h(t):=\|b(t)Z(t)\|_{\dot{\B}^{\frac d2+1}_{2,1}}^h,
\qquad
E(t):=E^\ell(t)+E^h(t).
\]

We  consider the cases $-1<\alpha<1$ and $\alpha=1$ separately.
\begin{itemize}
\item Case 1: $-1<\alpha<1$.
\end{itemize}
Taking $S(t)=t^M$ in \eqref{110}, we obtain
\begin{align}
&t^M \Big( E^\ell(t)+E^h(t)\Big)
+a_2^{-\frac{\theta}{1-\theta}}
\int_0^t \tau^{M+\alpha} \big(E^{\ell}(\tau) \big)^{\frac{1}{1-\theta}}\,d\tau
+a_3\int_0^t \tau^{M-\alpha} E^h(\tau)\,d\tau \nonumber \\
&\leq
M\int_0^t\tau^{M-1}\Big( E^\ell(\tau)+E^h(\tau)\Big)\,d\tau,
\label{DD3}
\end{align}
where $a_3>0$ is the constant in the dissipation rate.

To analyze the right-hand side of \eqref{DD3}, we make full use of the dissipation. First, by H\"older's and Young's inequalities, we have
\begin{align}
&\quad M\int_0^t \tau^{M-1} E^\ell(\tau)\,d\tau \nonumber\\
&\leq C\Bigg(\int_0^t
\tau^{\frac{M-1-(1-\theta)(M+\alpha)}{\theta}}\,d\tau
\Bigg)^{\theta}
\Bigg(\int_0^t
\tau^{M+\alpha}\big(E^\ell(\tau)\big)^{\frac{1}{1-\theta}}\,d\tau
\Bigg)^{1-\theta}\nonumber \\
&\leq C
\bigg(a_2 t^{M-\frac{(1+\alpha)(1-\theta)}{\theta}}\bigg)^{\theta}
\Bigg(
a_2^{-\frac{\theta}{1-\theta}}
\int_0^t \tau^{M+\alpha}\big(E^\ell(\tau)\big)^{\frac{1}{1-\theta}}\,d\tau
\Bigg)^{1-\theta} \nonumber \\
&\leq
\frac{1}{2}a_2^{-\frac{\theta}{1-\theta}}
\int_0^t \tau^{M+\alpha}\big(E^\ell(\tau)\big)^{\frac{1}{1-\theta}}\,d\tau
+ C\,a_2 t^{M-\frac{(1+\alpha)(1-\theta)}{\theta}}.
\label{R1111}
\end{align}
To handle the term involving $E^h(t)$ on the right-hand side of \eqref{DD3}, we have
\begin{align}
&\quad M\int_0^t\tau^{M-1}E^h(\tau)\,d\tau \nonumber\\
&\leq
C\Big(\int_0^t \tau^{M-\alpha} E^h(\tau)\,d\tau\Big)^{\frac{M-1}{M-\alpha}}
\Big(\int_0^t \sup_{\tau\in[0,t]}E(\tau)\,d\tau\Big)^{\frac{1-\alpha}{M-\alpha}} \nonumber\\
&\leq
\frac{1}{2} a_3\int_0^t \tau^{M-\alpha} E^h(\tau)\,d\tau
+C a_3^{-1} a_1 t.
\label{R11112}
\end{align}
This essentially corresponds to the fact that the rate
$e^{-a_3\int_0^t \tau^{-\alpha}\,d\tau}$, formally obtained by setting
$E^\ell(t)=0$, can lead to any algebraic decay rate $t^{1-M}$.
Substituting \eqref{R1111} and \eqref{R11112} into \eqref{DD3}, we get
\begin{equation*}
t^M \Big( E^\ell(t)+E^h(t)\Big)
\lesssim
\max\{a_1,a_2\}
\Big(t^{M-\frac{(1+\alpha)(1-\theta)}{\theta}}+t \Big).
\end{equation*}
Since $M\gg1$ can be chosen arbitrarily large, we arrive at \eqref{inter}
for $-1<\alpha<1$.

In the limiting case $\alpha=1$, the term involving $E^h(t)$ on the right-hand side of \eqref{DD3} can be directly absorbed by the left-hand side. Consequently, by \eqref{R1111} and \eqref{DD3}, we have
\begin{equation*}
t^M \Big( E^\ell(t)+E^h(t)\Big)
+(a_3-M)\int_0^t\tau^{M-1} E^h(\tau)\,d\tau
\lesssim
a_2 t^{M-\frac{2(1-\theta)}{\theta}}.
\end{equation*}
Once the large constant $M$ is fixed, the constant in the critical case $\alpha=1$ is chosen so that the corresponding dissipation constant $a_3$ satisfies $a_3>M$. Hence \eqref{inter} also holds in the critical case.

For later use in the proof of Theorem~\ref{thm3}, in the case $-1<\alpha\leq1$,
\eqref{110} with $S(t)=t^M$ implies the following stronger result:
\begin{equation}
\begin{aligned}\label{weight:ee}
&\quad
\|\tau^M Z\|_{\widetilde{L}^{\infty}_t(\dot{\B}_{2,1}^{\frac d2})}^\ell
+\|\tau^M b(\tau)Z\|_{\widetilde{L}^{\infty}_t(\dot{\B}_{2,1}^{\frac d2+1})}^h\\
&\quad
+\|\tau^M b(\tau)^{\frac{1}{2}}Z\|_{\widetilde{L}^{2}_t(\dot{\B}_{2,1}^{\frac d2+1})}
+\|\tau^M b(\tau)^{-\frac{1}{2}}Z_2\|_{\widetilde{L}^{2}_t(\dot{\B}_{2,1}^{\frac d2})}\\
&\quad
+\int_0^t\tau^M\Big(
b(\tau)\|Z_1\|_{\dot{\B}^{\frac d2+2}_{2,1}}^{\ell}
+\|Z_1\|_{\dot{\B}^{\frac d2+1}_{2,1}}^h
+\|Z_2\|_{\dot{\B}^{\frac d2+1}_{2,1}}
+\|\partial_t Z\|_{\dot{\B}^{\frac d2}_{2,1}}
+\frac{1}{b(\tau)}\|W\|_{\dot{\B}^{\frac d2}_{2,1}}
\Big)\,d\tau\\
&\lesssim
t^{M-\frac{1+\alpha}{2}(\frac{d}{2}-\sigma_1)}
\Big(
\|Z_0\|_{\dot{\B}_{2,1}^{\frac d2}\cap \dot{\B}_{2,1}^{\frac d2+1}}
+\|Z\|_{\widetilde L^{\infty}_t(\dot{\B}^{\sigma_1}_{2,\infty})}
\Big).
\end{aligned}
\end{equation}

\begin{itemize}
\item Case 2: $\alpha=-1$.
\end{itemize}
In this case, we will prove
\begin{align}\label{118}
\| Z(t)\|_{\dot{\B}_{2,1}^{\frac d2}}^\ell
+\| b(t)Z(t)\|_{\dot{\B}_{2,1}^{\frac d2+1}}^h
\lesssim
(\log(e+t))^{-\frac {1}{2}(\frac{d}{2}-\sigma_1)}
\Big(
\|Z_0\|_{\dot{\B}_{2,1}^{\frac d2}\cap \dot{\B}_{2,1}^{\frac d2+1}}
+\|Z\|_{\widetilde L^{\infty}_t(\dot{\B}^{\sigma_1}_{2,\infty})}
\Big).
\end{align}
For this purpose, we take
$S(t)=(\log(1+t))^M$, where $M$ is chosen sufficiently large so that
all the powers of $\log(1+\tau)$ appearing below are integrable near
$\tau=0$. Then \eqref{110} yields
\begin{align*}
&(\log(1+t))^M\Big(E^\ell(t)+E^h(t)\Big)\\
&\quad
+\int_0^t(\log(1+\tau))^M
\left(
a_2^{-\frac{\theta}{1-\theta}}(1+\tau)^{-1}
\big(E^\ell(\tau)\big)^{\frac{1}{1-\theta}}
+a_3(1+\tau)E^h(\tau)
\right)\,d\tau \\
&\leq
M\int_0^t(\log(1+\tau))^{M-1}(1+\tau)^{-1}
\Big(E^\ell(\tau)+E^h(\tau)\Big)\,d\tau.
\end{align*}
Similarly, we have
\begin{align*}
&\quad
\int_0^t
\frac{(\log(1+\tau))^{M-1}}{1+\tau}E^\ell(\tau)\,d\tau \\
&\leq
C\Bigg(\int_0^t
\frac{(\log(1+\tau))^M}{1+\tau}
\big(E^\ell(\tau)\big)^{\frac{1}{1-\theta}}\,d\tau
\Bigg)^{1-\theta}
\Bigg(\int_0^t
\frac{(\log(1+\tau))^{\frac{M\theta-1}{\theta}}}{1+\tau}\,d\tau
\Bigg)^{\theta}\\
&\leq
\frac{1}{2}a_2^{-\frac{\theta}{1-\theta}}
\int_0^t
\frac{(\log(1+\tau))^M}{1+\tau}
\big(E^\ell(\tau)\big)^{\frac{1}{1-\theta}}\,d\tau
+C\,a_2(\log(1+t))^{M-\frac{1-\theta}{\theta}},
\end{align*}
and
\begin{align*}
&\quad
\int_0^t
\frac{(\log(1+\tau))^{M-1}}{1+\tau}E^h(\tau)\,d\tau \\
&\leq
C\Big(\int_0^t(\log(1+\tau))^M(1+\tau)E^h(\tau)\,d\tau\Big)^{\frac{1}{2}}
\Big(\int_0^t
\frac{(\log(1+\tau))^{M-2}}{(1+\tau)^3}\,d\tau\,
\sup_{\tau\in[0,t]}E(\tau)
\Big)^{\frac{1}{2}}\\
&\leq
\frac{1}{2} a_3\int_0^t(\log(1+\tau))^M(1+\tau)E^h(\tau)\,d\tau
+C a_3^{-1}a_1 .
\end{align*}
Choosing $M$ sufficiently large, we obtain \eqref{118}.
Keeping all the dissipative terms in
$\|S(\tau)Z\|_{\mathcal E_t}$ in \eqref{110} also yields the
logarithmic counterpart of \eqref{weight:ee}. This estimate will be
used below to control the source terms in the error equation.

We now claim that
\begin{align}\label{Zsigma1}
\|Z\|_{\widetilde L^{\infty}_t(\dot{\B}^{\sigma_1}_{2,\infty})}\lesssim 1 .
\end{align}
Assuming \eqref{Zsigma1}, the estimate \eqref{inter} yields \eqref{decay1}
with $\sigma=\frac d2$ for $-1<\alpha\leq1$, while the logarithmic estimate above yields \eqref{decay2}
with $\sigma=\frac d2$ in the case $\alpha=-1$.
The decay rates in \eqref{decay1} and \eqref{decay2} corresponding to any $\sigma\in(\sigma_1,\frac d2)$ then follow by interpolation between \eqref{Zsigma1} and the corresponding estimate at the level $\dot{\B}^{\frac d2}_{2,1}$.

\subsection{Uniform evolution at lower-order regularity}

As discussed above, the decay estimates rely on the key lower-order regularity estimate \eqref{Zsigma1}. This is justified in the following proposition.

\begin{prop}\label{prop:lowre}
Additionally, assume that $Z_0\in \dot{\B}^{\sigma_1}_{2,\infty}\cap \dot{\B}^{\frac d2+1}_{2,1}$ with $
-\frac{d}{2}\leq \sigma_1<\frac{d}{2}$. 
Define
\begin{equation*}
\begin{aligned}
\|Z\|_{\mathcal{Y}_t}:={}&
\|Z\|_{\widetilde{L}_t^\infty(\dot{\B}^{\sigma_1}_{2,\infty})}
+\|b(\tau)Z\|_{\widetilde{L}_t^\infty(\dot{\B}^{\sigma_1+1}_{2,\infty})}
+\|b(\tau)Z_1\|^\ell_{L_t^1(\dot{\B}^{\sigma_1+2}_{2,\infty})}
+\|Z\|^h_{L_t^1(\dot{\B}^{\sigma_1+1}_{2,\infty})}\\
&+\|Z_2\|_{L_t^1(\dot{\B}^{\sigma_1+1}_{2,\infty})}
+\|b(\tau)^{\frac12}Z\|_{\widetilde{L}_t^2(\dot{\B}^{\sigma_1+1}_{2,\infty})}
+\|b(\tau)^{-\frac12}Z_2\|_{\widetilde{L}_t^2(\dot{\B}^{\sigma_1}_{2,\infty})}\\
&+\|\partial_t Z\|_{L_t^1(\dot{\B}^{\sigma_1}_{2,\infty})}
+\|b(\tau)^{-1}W\|_{L_t^1(\dot{\B}^{\sigma_1}_{2,\infty})}.
\end{aligned}
\end{equation*}
Then we have
\begin{equation}
\begin{aligned}\label{low:a}
\|Z\|_{\mathcal{Y}_t}
\lesssim
\|Z_0\|_{\dot{\B}^{\sigma_1}_{2,\infty}\cap \dot{\B}^{\frac{d}{2}+1}_{2,1}}.
\end{aligned}
\end{equation}
In particular, \eqref{Zsigma1} holds.
\end{prop}

\begin{proof}
We address the low and high frequencies separately.
\begin{itemize}
\item {\emph{Step 1: Low-frequency analysis.}}
\end{itemize}

Multiplying \eqref{dyadic-coupled-final} by $2^{\sigma_1j}$ and taking the supremum over the low-frequency dyadic blocks, we obtain
\begin{equation}
\begin{aligned}
\label{est:W999}
&\quad
\|Z\|^\ell_{\widetilde{L}^{\infty}_t(\dot\B_{2,\infty}^{\sigma_1})}
+\int_{0}^t\Big(
\|b(\tau) Z_1(\tau) \|_{\dot\B_{2,\infty}^{\sigma_1+2}}^\ell
+\|Z_2(\tau)\|_{\dot\B_{2,\infty}^{\sigma_1+1}}^\ell\\
&\qquad\qquad\qquad
+\left\|\frac{W(\tau)}{b(\tau)}\right\|_{\dot\B_{2,\infty}^{\sigma_1}}^\ell
+\|\partial_t Z(\tau) \|_{\dot\B_{2,\infty}^{\sigma_1}}^\ell
\Big)\,d\tau\\
&\quad
+\|b(\tau)^{\frac12}Z\|_{\widetilde{L}_t^2(\dot{\B}^{\sigma_1+1}_{2,\infty})}^{\ell}
+\|b(\tau)^{-\frac12}Z_2\|_{\widetilde{L}_t^2(\dot{\B}^{\sigma_1}_{2,\infty})}^{\ell}\\
&\lesssim
\|W_0\|_{\dot\B_{2,\infty}^{\sigma_1}}^\ell
+\|Z_{1,0}\|_{\dot\B_{2,\infty}^{\sigma_1}}^\ell
+\|(f_1,f_2)\|_{L^1_t(\dot\B_{2,\infty}^{\sigma_1})}^\ell\\
&\quad
+\sum_{k=1}^{d}\|b(\tau)N_2^k(Z)\|_{\widetilde{L}^{\infty}_t(\dot\B_{2,\infty}^{\sigma_1})
\cap L^1_t(\dot\B_{2,\infty}^{\sigma_1+2})}^\ell
+\|Q(Z)\|_{\widetilde{L}^{\infty}_t(\dot\B_{2,\infty}^{\sigma_1})
\cap L^1_t(\dot\B_{2,\infty}^{\sigma_1+1})}^\ell .
\end{aligned}
\end{equation}
According to the definition of $W$ and the standard estimates for composition functions, we have
\begin{equation*}
\begin{aligned}
\|W_0\|_{\dot\B_{2,\infty}^{\sigma_1}}^\ell
&\lesssim
\|Z_0\|_{\dot\B_{2,\infty}^{\sigma_1}}^\ell
+\|b(0)\nabla Z_0\|_{\dot\B_{2,\infty}^{\sigma_1}}^\ell
+\sum_{k=1}^d\|b(0)\nabla N_2^k(Z_0)\|_{\dot\B_{2,\infty}^{\sigma_1}}^\ell
+\|Q(Z_0)\|_{\dot\B_{2,\infty}^{\sigma_1}}\\
&\lesssim
\|Z_0\|_{\dot\B_{2,\infty}^{\sigma_1}}
+\|Z_0\|_{\dot{\B}^{\frac{d}{2}}_{2,1}}\|Z_0\|_{\dot\B_{2,\infty}^{\sigma_1}}.
\end{aligned}
\end{equation*}
Moreover,
\begin{equation*}
\begin{aligned}
\sum_{k=1}^d
\|N_2^k(Z)\|^\ell_{\widetilde{L}^{\infty}_t(\dot\B_{2,\infty}^{\sigma_1})}
+\|Q(Z)\|^\ell_{\widetilde{L}^{\infty}_t(\dot\B_{2,\infty}^{\sigma_1})}
\lesssim
\|Z\|_{\widetilde{L}^{\infty}_t(\dot\B_{2,1}^{\frac{d}{2}})}
\|Z\|_{\widetilde{L}^{\infty}_t(\dot\B_{2,\infty}^{\sigma_1})}.
\end{aligned}
\end{equation*}
Here and in what follows, we frequently use the product law \eqref{eq:prod3} under the restriction
$-\frac{d}{2}\leq \sigma_1<\frac{d}{2}$. Using  the standard product and composition estimates (see Proposition \ref{LP}, Lemmas \ref{Composition} and \ref{compositionlp}), straightforward computations yield
\begin{equation*}
\begin{aligned}
\sum_{k=1}^{d}
\|b(\tau)N_2^k(Z)\|_{L^1_t(\dot\B_{2,\infty}^{\sigma_1+2})}^\ell
&\lesssim
\|Z\|_{L^{\infty}_t(\dot{\B}^{\frac{d}{2}}_{2,1})}
\Big(
\|b(\tau)Z\|_{L^1_t(\dot\B_{2,\infty}^{\sigma_1+2})}^\ell
+\|Z\|_{L^1_t(\dot{\B}^{\sigma_1+1}_{2,\infty})}^h
\Big),\\
\|Q(Z)\|_{L^1_t(\dot\B_{2,\infty}^{\sigma_1+1})}^\ell
&\lesssim
\|Z\|_{L^{\infty}_t(\dot{\B}^{\frac{d}{2}}_{2,1})}
\Big(
\|Z_2\|_{L^1_t(\dot\B_{2,\infty}^{\sigma_1+1})}^\ell
+\|Z\|_{L^1_t(\dot{\B}^{\sigma_1+1}_{2,\infty})}^h
\Big).
\end{aligned}
\end{equation*}
The terms concerning $f_1$ and $f_2$ can be analyzed as in \eqref{f2}--\eqref{f1}. More precisely,
\begin{equation*}
\begin{aligned}
\|(f_1,f_2)\|_{L^1_t(\dot\B^{\sigma_1}_{2,\infty})}^{\ell}
\lesssim
\|Z\|_{\mathcal{E}_t}\,\|Z\|_{\mathcal{Y}_t}.
\end{aligned}
\end{equation*}
Combining the above estimates with \eqref{est:W999}, we get
\begin{equation}
\begin{aligned}\label{low:sigma1}
&\|Z\|^\ell_{\widetilde{L}^{\infty}_t(\dot\B_{2,\infty}^{\sigma_1})}
+\|b(\tau)^{\frac12}Z\|_{\widetilde{L}_t^2(\dot{\B}^{\sigma_1+1}_{2,\infty})}^{\ell}
+\|b(\tau)^{-\frac12}Z_2\|_{\widetilde{L}_t^2(\dot{\B}^{\sigma_1}_{2,\infty})}^{\ell}\\
&\quad
+\int_{0}^t\Big(
\|b(\tau) Z_1(\tau) \|_{\dot\B_{2,\infty}^{\sigma_1+2}}^\ell
+\|Z_2(\tau)\|_{\dot\B_{2,\infty}^{\sigma_1+1}}^\ell
+\left\|\frac{W(\tau)}{b(\tau)}\right\|_{\dot\B_{2,\infty}^{\sigma_1}}^\ell
+\|\partial_t Z(\tau) \|_{\dot\B_{2,\infty}^{\sigma_1}}^\ell
\Big)\,d\tau\\
&\lesssim
\|Z_0\|_{\dot{\B}^{\sigma_1}_{2,\infty}}
+\|Z\|_{\mathcal{E}_t}\,\|Z\|_{\mathcal{Y}_t}.
\end{aligned}
\end{equation}

\begin{itemize}
\item {\emph{Step 2: High-frequency analysis.}}
\end{itemize}

When $b'(t)\leq0$ (in particular, in the polynomial case with $\alpha\leq0$), we have $b(t)\lesssim1$. In this case, the desired high-frequency bounds follow directly from the uniform estimate \eqref{Xbound} and Lemma~\ref{lemma51}.

It remains to consider the case $b'(t)>0$. Recalling the localized estimate \eqref{highlocalized}, we arrive at
\begin{equation}
\begin{aligned}
&\quad
\|b(\tau)\widetilde{U}\|_{\widetilde{L}^\infty_t(\dot\B_{2,\infty}^{\sigma_1+1})}^h
+\|\widetilde{U}\|_{L^1_t(\dot\B_{2,\infty}^{\sigma_1+1})}^h\\
&\lesssim
\sup_{\substack{j\in\Z\\ t_j<t}}
b(t_j)2^{j(\sigma_1+1)}\|\widetilde{U}_j(t_j)\|_{L^2}
+\|b(\tau)\nabla U\|_{L^{\infty}_t(L^{\infty})}
\|\widetilde{U}\|_{L^1_t(\dot\B_{2,\infty}^{\sigma_1+1})}^h\\
&\quad
+\int_{0}^t
\sup_{j\geqslant J_\tau-1}
\left(
b(\tau)2^{j(\sigma_1+1)}
\big(\|R_j^1\|_{L^2}+\|R_j^2\|_{L^2}\big)
+2^{j\sigma_1}\|R_j^3\|_{L^2}
\right)\,d\tau .
\end{aligned}
\end{equation}
In the spirit of \eqref{N1h}--\eqref{N3h}, using Lemmas~\ref{commhigh} and \ref{compositionlp}, the commutator and composite function terms are bounded as follows:
\begin{equation}
\begin{aligned}
&\quad
\int_{0}^t
\sup_{j\geqslant J_\tau-1}
b(\tau)2^{j(\sigma_1+1)}\|R_j^1\|_{L^2}\,d\tau\\
&\lesssim
\|Z\|_{L^{\infty}_t(\dot{\B}^{\frac{d}{2}}_{2,1})}
\|Z_2\|_{L^{1}_t(\dot{\B}^{\sigma_1+1}_{2,\infty})}
+\|b(\tau)^{\frac{1}{2}}Z\|_{L^{2}_t(\dot{\B}^{\frac{d}{2}+1}_{2,1})}
\|b(\tau)^{-\frac{1}{2}}Z_2\|_{L^{2}_t(\dot{\B}^{\sigma_1}_{2,\infty})},\\
&\quad
\int_{0}^t
\sup_{j\geqslant J_\tau-1}
b(\tau)2^{j(\sigma_1+1)}\|R_j^2\|_{L^2}\,d\tau\\
&\lesssim
\|b(\tau)Z\|_{L^{\infty}_t(\dot{\B}^{\frac{d}{2}+1}_{2,1})}
\|\partial_t Z\|_{L^{1}_t(\dot{\B}^{\sigma_1}_{2,\infty})}
+\|b(\tau)^{\frac{1}{2}}Z\|_{L^{2}_t(\dot{\B}^{\frac{d}{2}+1}_{2,1})}
\|b(\tau)^{\frac{1}{2}}Z\|_{L^{2}_t(\dot{\B}^{\sigma_1+1}_{2,\infty})}.
\end{aligned}
\end{equation}
For $R_j^3$, we have
\begin{equation}
\begin{aligned}
&\quad
\int_{0}^t
\sup_{j\geqslant J_\tau-1}
2^{j\sigma_1}\|R_j^3\|_{L^2}\,d\tau\\
&\lesssim
\int_{0}^t
\Big(
\|(\widetilde{A}^0(U)-\bar{A}^0)\partial_t\widetilde{U}\|_{\dot{\B}^{\sigma_1}_{2,\infty}}^h
+\sum_{k=1}^d
\|(\widetilde{A}^k(U)-\bar{A}^k)\partial_{x_k}\widetilde{U}\|_{\dot{\B}^{\sigma_1}_{2,\infty}}^h
\Big)\,d\tau .
\end{aligned}
\end{equation}
Moreover,
\begin{align}
&\quad
\int_{0}^t
\|(\widetilde{A}^0(U)-\bar{A}^0)\partial_t\widetilde{U}\|_{\dot{\B}^{\sigma_1}_{2,\infty}}^h\,d\tau
\lesssim
(1+\|Z\|_{L^{\infty}_t(\dot{\B}^{\frac{d}{2}}_{2,1})})
\|Z\|_{L^{\infty}_t(\dot{\B}^{\frac{d}{2}}_{2,1})}
\|\partial_t Z\|_{L^1_t(\dot{\B}^{\sigma_1}_{2,\infty})},
\label{R3-lowreg-1}
\end{align}
and
\begin{align}
&\quad
\int_{0}^t
\|(\widetilde{A}^k(U)-\bar{A}^k)\partial_{x_k}\widetilde{U}\|_{\dot{\B}^{\sigma_1}_{2,\infty}}^h\,d\tau
\notag\\
&\lesssim
\|b(\tau)^{\frac{1}{2}}Z\|_{L^2_t(\dot{\B}^{\frac{d}{2}+1}_{2,1})}
\|b(\tau)^{\frac{1}{2}}\nabla Z^\ell\|_{L^2_t(\dot{\B}^{\sigma_1}_{2,\infty})}
\notag\\
&\quad
+\|Z\|_{L^{\infty}_t(\dot{\B}^{\frac d2}_{2,1})}
\Big(
\|b(\tau)\nabla Z^\ell\|_{L^{1}_t(\dot{\B}^{\sigma_1+1}_{2,\infty})}
+\|\nabla Z^h\|_{L^{1}_t(\dot{\B}^{\sigma_1}_{2,\infty})}
\Big).
\label{R3-lowreg-2}
\end{align}
According to the definition of $t_j$, we have
\begin{align}
\sup_{\substack{j\in\Z\\ t_j<t}}
b(t_j)2^{j(\sigma_1+1)}\|\widetilde{U}_j(t_j)\|_{L^2}&\lesssim
\sup_{\substack{j\in\Z\\ 0<t_j<t}}
2^{j\sigma_1}\|\widetilde{U}_j(t_j)\|_{L^2}
+\sup_{\substack{j\in\Z\\ t_j=0}}
b(0)2^{j(\sigma_1+1)}\|\widetilde{U}_j(0)\|_{L^2}
\notag\\
&\lesssim
\|\widetilde U\|_{\widetilde L^\infty_t(\dot{\B}^{\sigma_1}_{2,\infty})}^{\ell}
+\|\widetilde U_0\|_{\dot{\B}^{\sigma_1+1}_{2,\infty}}^h
\notag\\
&\lesssim
\|Z\|_{\widetilde L^\infty_t(\dot{\B}^{\sigma_1}_{2,\infty})}^{\ell}
+\|Z_0\|_{\dot{\B}^{\frac d2+1}_{2,1}}.
\label{transition-lowreg}
\end{align}
Combining the above estimates, the low-frequency estimate \eqref{low:sigma1}, and the smallness of $\|Z\|_{\mathcal{E}_t}$, we derive the desired bounds for $\widetilde U$ in high frequencies. Then, using composition and product estimates, one recovers the desired bounds for $Z$ in high frequencies, as in Step~3 of the proof of Lemma~\ref{eslq:high}. Consequently,
\begin{equation}
\begin{aligned}\label{high:sigma1}
&\|b(\tau)Z\|_{\widetilde L^\infty_t(\dot{\B}^{\sigma_1+1}_{2,\infty})}^h
+\|Z\|_{L^1_t(\dot{\B}^{\sigma_1+1}_{2,\infty})}^h\lesssim
\|Z_0\|_{\dot{\B}^{\sigma_1}_{2,\infty}\cap\dot{\B}^{\frac d2+1}_{2,1}}
+\|Z\|_{\mathcal{E}_t}\|Z\|_{\mathcal{Y}_t}.
\end{aligned}
\end{equation}

\begin{itemize}
\item {\emph{Step 3: Closing the lower-order estimate.}}
\end{itemize}

Collecting all the above estimates in low and high frequencies, we end up with
\begin{align*}
\|Z\|_{\mathcal{Y}_t}
\lesssim
\|Z_0\|_{\dot{\B}^{\sigma_1}_{2,\infty}\cap \dot{\B}^{\frac d2+1}_{2,1}}
+\|Z\|_{\mathcal{E}_t}\|Z\|_{\mathcal{Y}_t}.
\end{align*}
Since $\|Z\|_{\mathcal{E}_t}\lesssim \|Z_0\|_{\dot{\B}^{\frac d2}_{2,1}\cap \dot{\B}^{\frac d2+1}_{2,1}}\ll1$, the last term can be absorbed into the left-hand side. This yields \eqref{low:a} and finishes the proof of Proposition~\ref{prop:lowre}.
\end{proof}

\subsection{Enhanced stability of \texorpdfstring{$Z_2$}{Z2}}

We are now in a position to prove the improved decay estimate for $Z_2$.
From the second equation of \eqref{sys} and the dissipativity of $L_2$, we
have, for each dyadic block,
\begin{align}
\frac{d}{dt}\|Z_{2,j}(t)\|_{L^2}
+\frac{\kappa}{b(t)}\|Z_{2,j}(t)\|_{L^2}
\lesssim
2^j\|Z_j(t)\|_{L^2}
+2^j\sum_{k=1}^d\|\dot{\Delta}_j N^k_2(Z)(t)\|_{L^2}
+\frac{1}{b(t)}\|\dot{\Delta}_j Q(Z)(t)\|_{L^2}.
\label{Z2-dyadic-enhanced}
\end{align}
By Gronwall's inequality, for any $t\geq t_b$ with $t_b\gg1$ given in
\eqref{mu:infer}, multiplying the resulting inequality by $2^{j\sigma}$ and
summing over $j\in\mathbb Z$, we obtain
\begin{align}
&\quad\|Z_2(t)\|_{\dot{\B}^{\sigma}_{2,1}}\notag\\
&\lesssim
\exp\left(-\kappa\int_{t_b}^t\frac{ds}{b(s)}\right)
\|Z_2(t_b)\|_{\dot{\B}^{\sigma}_{2,1}}
\notag\\
&\quad
+\int_{t_b}^t
\exp\left(-\kappa\int_{\tau}^t\frac{ds}{b(s)}\right)
\Big(
\|Z(\tau)\|_{\dot{\B}^{\sigma+1}_{2,1}}
+\sum_{k=1}^d\|N_2^k(Z)(\tau)\|_{\dot{\B}^{\sigma+1}_{2,1}}
+\frac{1}{b(\tau)}\|Q(Z)(\tau)\|_{\dot{\B}^{\sigma}_{2,1}}
\Big)\,d\tau .
\label{Z2-duhamel-besov}
\end{align}
Here, due to interpolation, \eqref{Xbound} and \eqref{low:a}, we have
\begin{align}
\|Z_2(t)\|_{\dot{\B}^{\sigma}_{2,1}}
\lesssim
\|Z_2(t)\|_{\dot{\B}^{\frac d2}_{2,1}}
+\|Z_2(t)\|_{\dot{\B}^{\sigma_1}_{2,\infty}}
\lesssim
\|Z_0\|_{\dot{\B}_{2,1}^{\frac d2}\cap \dot{\B}_{2,1}^{\frac d2+1}}
+\|Z_0\|_{\dot{\B}^{\sigma_1}_{2,\infty}},
\quad t\geq0.
\label{init}
\end{align}
Recalling \eqref{mu:infer}, we have
\[
\int_{t_b}^t\frac{ds}{b(s)}
\geq
c_b\int_{t_b}^t(1+s)^{-\alpha}\,ds.
\]
Therefore,
\begin{equation}\label{exp1}
\begin{aligned}
\exp\left(-\kappa\int_{t_b}^t\frac{ds}{b(s)}\right)
\lesssim
\begin{cases}
\exp\big(-c(1+t)^2\big), & \alpha=-1,\\[0.4em]
\exp\big(-c(1+t)^{1-\alpha}\big), & -1<\alpha<1,\\[0.4em]
(1+t)^{-\kappa c_b}, & \alpha=1.
\end{cases}
\end{aligned}
\end{equation}
In particular, when $-1\leq\alpha<1$, this decay implies any algebraic
decay, while in the case $\alpha=1$ the algebraic decay rate is ensured by
choosing $c_b$ sufficiently large.

\begin{itemize}
\item Case 1: $-1<\alpha\leq1$.
\end{itemize}
We first consider the case $-1<\alpha\leq1$. For any fixed
$\gamma\in\mathbb R$, we claim that
\begin{align}
\int_{t_b}^t
\exp\left(-\kappa\int_\tau^t\frac{ds}{b(s)}\right)
(1+\tau)^{-\gamma}\,d\tau
\lesssim
(1+t)^{\alpha-\gamma}.
\label{kernel-algebraic-refined}
\end{align}
Indeed, for $\tau\in[t_b,t/2]$, \eqref{exp1} gives a sufficiently fast decay
with respect to $t$, and hence
\[
\int_{t_b}^{t/2}
\exp\left(-\kappa\int_\tau^t\frac{ds}{b(s)}\right)
(1+\tau)^{-\gamma}\,d\tau
\lesssim
(1+t)^{\alpha-\gamma}.
\]
For $\tau\in[t/2,t]$, we have
\[
\int_\tau^t\frac{ds}{b(s)}
\gtrsim
(1+t)^{-\alpha}(t-\tau),
\]
and $\tau\approx t$. Thus,
\[
\begin{aligned}
&\quad
\int_{t/2}^{t}
\exp\left(-\kappa\int_\tau^t\frac{ds}{b(s)}\right)
(1+\tau)^{-\gamma}\,d\tau\\
&\lesssim
(1+t)^{-\gamma}
\int_{t/2}^{t}
\exp\big(-c(1+t)^{-\alpha}(t-\tau)\big)\,d\tau
\lesssim
(1+t)^{\alpha-\gamma}.
\end{aligned}
\]
Combining the above two estimates proves \eqref{kernel-algebraic-refined}.

Let $\sigma_1<\sigma\leq \frac d2-1$. By the decay estimate \eqref{decay1}
and the composite estimates, we have
\begin{align}
\|Z(\tau)\|_{\dot{\B}^{\sigma+1}_{2,1}}
+\sum_{k=1}^d\|N_2^k(Z)(\tau)\|_{\dot{\B}^{\sigma+1}_{2,1}}
\lesssim
(1+\tau)^{-\frac{1+\alpha}{2}(\sigma+1-\sigma_1)}.
\label{source-Z2-decay}
\end{align}
Define
\[
M_\sigma(t):=
\sup_{t_b\leq s\leq t}
(1+s)^{\frac{1+\alpha}{2}(\sigma-\sigma_1)+\frac{1-\alpha}{2}}
\|Z_2(s)\|_{\dot{\B}^{\sigma}_{2,1}}.
\]
Using the factorization $Q(Z)=q_1(Z)Z_2$, \eqref{Xbound}, and the definition of
$M_\sigma(t)$, we get
\begin{align}
\frac{1}{b(\tau)}\|Q(Z)(\tau)\|_{\dot{\B}^{\sigma}_{2,1}}
&\lesssim
(1+\tau)^{-\alpha}
\|Z(\tau)\|_{\dot{\B}^{\frac d2}_{2,1}}
\|Z_2(\tau)\|_{\dot{\B}^{\sigma}_{2,1}}
\notag\\
&\lesssim
\|Z_0\|_{\dot{\B}_{2,1}^{\frac d2}\cap \dot{\B}_{2,1}^{\frac d2+1}}
(1+\tau)^{-\frac{1+\alpha}{2}(\sigma+1-\sigma_1)}
M_\sigma(t).
\label{Q-Z2-besov}
\end{align}
Substituting \eqref{init}, \eqref{source-Z2-decay} and \eqref{Q-Z2-besov}
into \eqref{Z2-duhamel-besov}, and using \eqref{exp1} and
\eqref{kernel-algebraic-refined}, we arrive at
\begin{align*}
M_\sigma(t)
\lesssim
\|Z_0\|_{\dot{\B}_{2,1}^{\frac d2}\cap \dot{\B}_{2,1}^{\frac d2+1}}
+\|Z_0\|_{\dot{\B}^{\sigma_1}_{2,\infty}}
+\|Z_0\|_{\dot{\B}_{2,1}^{\frac d2}\cap \dot{\B}_{2,1}^{\frac d2+1}}
M_\sigma(t).
\end{align*}
Taking the initial data sufficiently small, we infer that
\[
M_\sigma(t)
\lesssim
\|Z_0\|_{\dot{\B}_{2,1}^{\frac d2}\cap \dot{\B}_{2,1}^{\frac d2+1}}
+\|Z_0\|_{\dot{\B}^{\sigma_1}_{2,\infty}}.
\]
Consequently, \eqref{decayimprove1} holds 
for all $\sigma_1<\sigma\leq \frac d2-1$ and $-1<\alpha\leq1$.

\begin{itemize}
\item Case 2: $\alpha=-1$.
\end{itemize}

It remains to consider the limiting case $\alpha=-1$. In this case, for fixed
$\gamma\in\mathbb R$, we claim that
\begin{align}
\int_{t_b}^t
\exp\left(-\kappa\int_{\tau}^t\frac{ds}{b(s)}\right)
(\log(e+\tau))^{-\gamma}\,d\tau
\lesssim
(1+t)^{-1}(\log(e+t))^{-\gamma}.
\label{kernel-log}
\end{align}
Let us justify the corresponding logarithmic convolution estimate. Since
\[
\exp\left(-\kappa\int_{\tau}^t\frac{ds}{b(s)}\right)
\lesssim
\exp\left(-c\big((1+t)^2-(1+\tau)^2\big)\right),
\]
we split the time integral into $[t_b,t/2]$ and $[t/2,t]$. For
$\tau\in[t_b,t/2]$,
\[
\int_{t_b}^{t/2}
\exp\left(-\kappa\int_{\tau}^t\frac{ds}{b(s)}\right)
(\log(e+\tau))^{-\gamma}\,d\tau
\lesssim
e^{-c(1+t)^2}(1+t)
\lesssim
(1+t)^{-1}(\log(e+t))^{-\gamma}.
\]
For $\tau\in[t/2,t]$, we use $(1+t)^2-(1+\tau)^2
=(t-\tau)(2+t+\tau)
\gtrsim
(1+t)(t-\tau)$ and $\log(e+\tau)\approx\log(e+t)$ to obtain
\[
\begin{aligned}
&\quad
\int_{t/2}^{t}
\exp\left(-\kappa\int_{\tau}^t\frac{ds}{b(s)}\right)
(\log(e+\tau))^{-\gamma}\,d\tau\\
&\lesssim
(\log(e+t))^{-\gamma}
\int_{t/2}^{t}
e^{-c(1+t)(t-\tau)}\,d\tau
\lesssim
(1+t)^{-1}(\log(e+t))^{-\gamma}.
\end{aligned}
\]
Combining the two estimates gives \eqref{kernel-log}.

Using the logarithmic decay estimate \eqref{decay2}, we have
\begin{align}
\|Z(\tau)\|_{\dot{\B}^{\sigma+1}_{2,1}}
+\sum_{k=1}^d\|N_2^k(Z)(\tau)\|_{\dot{\B}^{\sigma+1}_{2,1}}
\lesssim
(\log(e+\tau))^{-\frac12(\sigma+1-\sigma_1)},
\quad
\sigma_1<\sigma\leq \frac d2-1.
\label{source-Z2-log}
\end{align}
Define
\[
M_\sigma^{\log}(t):=
\sup_{t_b\leq s\leq t}
(1+s)(\log(e+s))^{\frac12(\sigma+1-\sigma_1)}
\|Z_2(s)\|_{\dot{\B}^{\sigma}_{2,1}}.
\]
Then, using the factorization $Q(Z)=q_1(Z)Z_2$ and \eqref{Xbound}, we get
\begin{align}
\frac{1}{b(\tau)}\|Q(Z)(\tau)\|_{\dot{\B}^{\sigma}_{2,1}}
&\lesssim
(1+\tau)
\|Z(\tau)\|_{\dot{\B}^{\frac d2}_{2,1}}
\|Z_2(\tau)\|_{\dot{\B}^{\sigma}_{2,1}}
\notag\\
&\lesssim
\|Z_0\|_{\dot{\B}_{2,1}^{\frac d2}\cap \dot{\B}_{2,1}^{\frac d2+1}}
(\log(e+\tau))^{-\frac12(\sigma+1-\sigma_1)}
M_\sigma^{\log}(t).
\label{Q-Z2-log}
\end{align}
Combining \eqref{Z2-duhamel-besov}, \eqref{exp1}, \eqref{kernel-log},
\eqref{source-Z2-log} and \eqref{Q-Z2-log}, we obtain
\begin{align*}
M_\sigma^{\log}(t)
\lesssim
\|Z_0\|_{\dot{\B}_{2,1}^{\frac d2}\cap \dot{\B}_{2,1}^{\frac d2+1}}
+\|Z_0\|_{\dot{\B}^{\sigma_1}_{2,\infty}}
+\|Z_0\|_{\dot{\B}_{2,1}^{\frac d2}\cap \dot{\B}_{2,1}^{\frac d2+1}}
M_\sigma^{\log}(t).
\end{align*}
Taking the initial data sufficiently small, we conclude that
\[
M_\sigma^{\log}(t)
\lesssim
\|Z_0\|_{\dot{\B}_{2,1}^{\frac d2}\cap \dot{\B}_{2,1}^{\frac d2+1}}
+\|Z_0\|_{\dot{\B}^{\sigma_1}_{2,\infty}}.
\]
Therefore, we conclude \eqref{decay2improve} 
for all $\sigma_1<\sigma\leq \frac d2-1$ in the limiting case $\alpha=-1$.

\subsection{Enhanced stability of the time-dependent diffusion profile}

This section is devoted to the proof of the enhanced stability of the
time-dependent diffusion profile under the additional structural assumptions
\eqref{structural2}. Let
\[
\delta Z_1:=Z_1-\mathcal{N},
\]
where $\mathcal{N}$ is the solution to \eqref{ZL}. Since
$N_1^k(Z)=0$ and $Q(Z)=0$, it follows from \eqref{eqz1} that
$\delta Z_1$ satisfies
\begin{equation}\label{errorZ}
\left\{
\begin{aligned}
&\partial_t \delta Z_1
-b(t)\sum_{k=1}^d\sum_{l=1}^d
A^k_{1,2}L_2^{-1}A^l_{2,1}
\partial_k\partial_l \delta Z_1
=
\sum_{k=1}^d\partial_k\widetilde{G}^k(Z),\\
&\delta Z_1(0,x)=0,
\end{aligned}
\right.
\end{equation}
where
\begin{equation}\label{Gtilde-def}
\begin{aligned}
\widetilde{G}^k(Z)
:={}&
-A^k_{1,2}W
+b(t)\sum_{l=1}^d
A^k_{1,2}L_2^{-1}A^l_{2,2}\partial_l Z_2
-b(t)\sum_{l=1}^d
A^k_{1,2}L_2^{-1}\partial_l N^l_2(Z).
\end{aligned}
\end{equation}
The source term in \eqref{errorZ} consists of the damped mode $W$, a
linear term involving one spatial derivative of $Z_2$, and a nonlinear
term involving $N_2(Z)$. The weighted estimates established previously
for $W$ and $Z_2$, together with the divergence structure of the
right-hand side of \eqref{errorZ}, yield one additional spatial
derivative and hence an additional diffusive decay factor.

For later use, let $M\gg1$ be fixed and define the time weight
\[
S_M(t):=
\begin{cases}
(1+t)^M, & -1<\alpha\leq1,\\[0.3em]
(\log(1+t))^M, & \alpha=-1,
\end{cases},\qquad 
\Lambda_\alpha(t):=
\begin{cases}
1, & -1\leq\alpha\leq0,\\[0.3em]
b(t)^{-1}, & 0<\alpha\leq1.
\end{cases}
\]
We also set
\[
R_M(t):=
\begin{cases}
t^{M-\frac{1+\alpha}{2}
(\frac d2-\sigma_1)},
& -1<\alpha\leq1,\\[0.3em]
(\log(1+t))^{M-\frac12
(\frac d2-\sigma_1)},
& \alpha=-1.
\end{cases}
\]

\begin{lem}\label{lem:deltaZ1-weighted}
Let $-1\leq\alpha\leq1$. For any $M\gg1$ and all $t\geq1$, it holds that
\begin{align}
&\quad
\|S_M(\tau)\Lambda_\alpha(\tau)\delta Z_1\|
_{\widetilde{L}^{\infty}_t
(\dot{\B}^{\frac d2-1}_{2,1})}^{\ell}
+
\|S_M(\tau)b(\tau)\Lambda_\alpha(\tau)\delta Z_1\|
_{L^1_t
(\dot{\B}^{\frac d2+1}_{2,1})}^{\ell}
\notag\\
&\lesssim
R_M(t)
\Big(
\|Z_0\|_{\dot{\B}^{\sigma_1}_{2,\infty}
\cap\dot{\B}^{\frac d2+1}_{2,1}}
+
\|\Lambda_\alpha(\tau)\delta Z_1\|
_{\widetilde{L}^{\infty}_t
(\dot{\B}^{\sigma_1-1}_{2,\infty})}^{\ell}
\Big).
\label{deltaZ1-weighted}
\end{align}
\end{lem}

\begin{proof}
We first deal with the high-frequency part. By \eqref{inter},
\eqref{low:a}, and the corresponding estimate for the linear diffusion
equation \eqref{ZL}, we have
\begin{align}
\|S_M(\tau)\Lambda_\alpha(\tau)\delta Z_1\|
_{\widetilde{L}^{\infty}_t
(\dot{\B}^{\frac d2-1}_{2,1})}^{h}
\lesssim
R_M(t)
\|Z_0\|_{\dot{\B}^{\sigma_1}_{2,\infty}
\cap\dot{\B}^{\frac d2+1}_{2,1}}.
\label{deltaZ1-high}
\end{align}
Indeed, for the diffusion profile, the localized energy estimate gives
\[
\|\dot{\Delta}_j\mathcal{N}(t)\|_{L^2}
\lesssim
\exp\left(
-c\,2^{2j}\int_0^t b(\tau)\,d\tau
\right)
\|\dot{\Delta}_j Z_{1,0}\|_{L^2},
\]
which yields the same high-frequency bound.

It remains to consider the low frequencies. Applying $\dot{\Delta}_j$
to \eqref{errorZ} and taking the $L^2$ inner product with
$\dot{\Delta}_j\delta Z_1$, we obtain, for all $j\leq J_t$,
\begin{align}
\frac12\frac{d}{dt}
\|\dot{\Delta}_j\delta Z_1\|_{L^2}^2
+\kappa_0 2^{2j}b(t)
\|\dot{\Delta}_j\delta Z_1\|_{L^2}^2
\lesssim
2^j\sum_{k=1}^d
\|\dot{\Delta}_j\widetilde{G}^k(Z)\|_{L^2}
\|\dot{\Delta}_j\delta Z_1\|_{L^2}.
\label{deltaZ1L2-new}
\end{align}
The localized estimate is integrated on $I_{j,t}^{\ell}$.
If this interval starts at $t_j$, the boundary term at $t_j$ is
controlled by the high-frequency estimate \eqref{deltaZ1-high}.

Multiplying \eqref{deltaZ1L2-new} by
$S_M(t)^2\Lambda_\alpha(t)^2$, we obtain
\begin{align*}
&\frac12\frac{d}{dt}
\left(
S_M^2\Lambda_\alpha^2
\|\dot{\Delta}_j\delta Z_1\|_{L^2}^2
\right)
+\kappa_0 2^{2j}b(t)S_M^2\Lambda_\alpha^2
\|\dot{\Delta}_j\delta Z_1\|_{L^2}^2
\\
&\lesssim
2^jS_M^2\Lambda_\alpha^2
\sum_{k=1}^d
\|\dot{\Delta}_j\widetilde{G}^k(Z)\|_{L^2}
\|\dot{\Delta}_j\delta Z_1\|_{L^2}
+
\big((S_M\Lambda_\alpha)'\big)_+S_M\Lambda_\alpha
\|\dot{\Delta}_j\delta Z_1\|_{L^2}^2.
\end{align*}
Here and below, $r_+:=\max\{r,0\}$. The negative part of
$(S_M\Lambda_\alpha)'$ has a favorable sign and is therefore discarded.

Dividing by $\sqrt{
S_M^2\Lambda_\alpha^2
\|\dot{\Delta}_j\delta Z_1\|_{L^2}^2+\varepsilon^2}$, integrating in time, and then letting $\varepsilon\to0$, we infer, after
multiplying by $2^{j(\frac d2-1)}$ and summing over the low-frequency
indices, that
\begin{align}
&\quad
\|S_M(\tau)\Lambda_\alpha(\tau)\delta Z_1\|
_{\widetilde{L}^{\infty}_t
(\dot{\B}^{\frac d2-1}_{2,1})}^{\ell}
+
\|S_M(\tau)b(\tau)\Lambda_\alpha(\tau)\delta Z_1\|
_{L^1_t
(\dot{\B}^{\frac d2+1}_{2,1})}^{\ell}
\notag\\
&\lesssim
\sum_{k=1}^d
\|S_M(\tau)\Lambda_\alpha(\tau)\widetilde{G}^k(Z)\|
_{L^1_t
(\dot{\B}^{\frac d2}_{2,1})}^{\ell}
+
\left\|
\big((S_M\Lambda_\alpha)'(\tau)\big)_+\delta Z_1
\right\|_{L^1_t
(\dot{\B}^{\frac d2-1}_{2,1})}^{\ell}.
\label{deltaZ1-basic}
\end{align}

We first estimate the source term. From \eqref{Gtilde-def}, we have
\begin{align}
&\sum_{k=1}^d
\|S_M(\tau)\Lambda_\alpha(\tau)\widetilde{G}^k(Z)\|
_{L^1_t
(\dot{\B}^{\frac d2}_{2,1})}^{\ell}
\notag\\
&\lesssim
\|S_M(\tau)\Lambda_\alpha(\tau)W\|
_{L^1_t
(\dot{\B}^{\frac d2}_{2,1})}^{\ell}
+
\|S_M(\tau)b(\tau)\Lambda_\alpha(\tau)Z_2\|
_{L^1_t
(\dot{\B}^{\frac d2+1}_{2,1})}^{\ell}
\notag\\
&\quad
+
\sum_{l=1}^d
\|S_M(\tau)b(\tau)\Lambda_\alpha(\tau)N_2^l(Z)\|
_{L^1_t
(\dot{\B}^{\frac d2+1}_{2,1})}^{\ell}.
\label{source-Gtilde}
\end{align}
We now estimate the three terms on the right-hand side of
\eqref{source-Gtilde}. First, by the definition of
$\Lambda_\alpha$, we have
\begin{equation}\label{balpha-uniform}
b(\tau)\Lambda_\alpha(\tau)\lesssim 1,
\qquad
b(\tau)^{\frac12}\Lambda_\alpha(\tau)
\lesssim b(\tau)^{-\frac12},
\qquad -1\leq\alpha\leq1.
\end{equation}
Indeed, both inequalities are identities when $\alpha>0$, whereas
they follow from the uniform boundedness of $b$ when $\alpha\leq0$.

For the damped mode, we have
\begin{align}
\|S_M(\tau)\Lambda_\alpha(\tau)W\|_{L^1_t
(\dot{\B}^{\frac d2}_{2,1})}^{\ell}
\lesssim
\left\|S_M(\tau)\frac{W}{b(\tau)}\right\|_{L^1_t
(\dot{\B}^{\frac d2}_{2,1})}^{\ell}.
\label{weighted-W-error}
\end{align}
In fact, this is an equality when $\alpha>0$, while for
$\alpha\leq0$ it follows from the uniform boundedness of $b$.

For the linear term involving $Z_2$, the first inequality in
\eqref{balpha-uniform} gives
\begin{align}
\|S_M(\tau)b(\tau)\Lambda_\alpha(\tau)Z_2\|_{L^1_t
(\dot{\B}^{\frac d2+1}_{2,1})}^{\ell}
\lesssim
\|S_M(\tau)Z_2\|_{L^1_t
(\dot{\B}^{\frac d2+1}_{2,1})}.
\label{weighted-Z2-error}
\end{align}

It remains to estimate the nonlinear term. The structural condition
$N_2^l(Z_1,0)=0$ implies that
\[
N_2^l(Z)=g_2^l(Z)Z_2,
\]
where $g_2^l$ is smooth and $g_2^l(0)=0$. Hence, by the product
estimates,
\begin{align}
&\sum_{l=1}^d
\|S_M(\tau)b(\tau)\Lambda_\alpha(\tau)N_2^l(Z)\|_{L^1_t
(\dot{\B}^{\frac d2+1}_{2,1})}^{\ell}
\notag\\
&\lesssim
\|Z\|_{L^\infty_t(\dot{\B}^{\frac d2}_{2,1})}
\|S_M(\tau)b(\tau)\Lambda_\alpha(\tau)Z_2\|_{L^1_t
(\dot{\B}^{\frac d2+1}_{2,1})}
\notag\\
&\quad+
\|S_M(\tau)b(\tau)^{\frac12}Z\|_{\widetilde L^2_t
(\dot{\B}^{\frac d2+1}_{2,1})}
\|b(\tau)^{\frac12}\Lambda_\alpha(\tau)Z_2\|_{\widetilde L^2_t
(\dot{\B}^{\frac d2}_{2,1})}
\notag\\
&\lesssim
\|Z\|_{\mathcal E_t}
\left(
\|S_M(\tau)Z_2\|_{L^1_t
(\dot{\B}^{\frac d2+1}_{2,1})}
+
\|S_M(\tau)b(\tau)^{-\frac12}Z_2\|_{\widetilde L^2_t
(\dot{\B}^{\frac d2}_{2,1})}
\right),
\label{weighted-N2-error}
\end{align}
where \eqref{balpha-uniform} has been used in the last inequality.

For $-1<\alpha\leq1$, the contribution on $[0,1]$ is controlled by
\eqref{Xbound}, while for $\tau\geq1$ one has
$S_M(\tau)\lesssim\tau^M$ and hence the required bounds follow from
\eqref{weight:ee}. In the case $\alpha=-1$, we use the corresponding
logarithmically weighted estimate established above. Combining
\eqref{source-Gtilde}, \eqref{weighted-W-error},
\eqref{weighted-Z2-error}, and \eqref{weighted-N2-error}, we obtain
\begin{align}
\sum_{k=1}^d
\|S_M(\tau)\Lambda_\alpha(\tau)\widetilde{G}^k(Z)\|
_{L^1_t(\dot{\B}^{\frac d2}_{2,1})}^{\ell}
\lesssim
R_M(t)
\|Z_0\|_{\dot{\B}^{\sigma_1}_{2,\infty}
\cap\dot{\B}^{\frac d2+1}_{2,1}}.
\label{Gtilde-bound-new}
\end{align}

We now handle the last term in \eqref{deltaZ1-basic}. Note that
\[
\frac{\big((S_M\Lambda_\alpha)'(t)\big)_+}
{S_M(t)\Lambda_\alpha(t)}
\lesssim
\begin{cases}
(1+t)^{-1},&-1<\alpha\leq1,\\[0.3em]
\big((1+t)\log(1+t)\big)^{-1},&\alpha=-1.
\end{cases}
\]
Let $\theta_*\in(0,1)$ be defined by $\frac d2-1
=
\theta_*(\sigma_1-1)
+(1-\theta_*)\Big(\frac d2+1\Big)$. Using the interpolation inequality between
$\dot{\B}^{\sigma_1-1}_{2,\infty}$ and
$\dot{\B}^{\frac d2+1}_{2,1}$, followed by Young's inequality,
and arguing as in Subsection~\ref{subsection:inter}, we obtain, for
any $\zeta_*>0$,
\begin{align}
&\left\|
\big((S_M\Lambda_\alpha)'(\tau)\big)_+
\delta Z_1
\right\|_{L^1_t
(\dot{\B}^{\frac d2-1}_{2,1})}^{\ell}
\notag\\
&\leq
\zeta_*
\|S_M(\tau)b(\tau)\Lambda_\alpha(\tau)\delta Z_1\|
_{L^1_t
(\dot{\B}^{\frac d2+1}_{2,1})}^{\ell}
+
C_{\zeta_*}R_M(t)
\|\Lambda_\alpha(\tau)\delta Z_1\|
_{\widetilde{L}^{\infty}_t
(\dot{\B}^{\sigma_1-1}_{2,\infty})}^{\ell}.
\label{lower-order-delta}
\end{align}
Combining \eqref{deltaZ1-basic}, \eqref{Gtilde-bound-new} and
\eqref{lower-order-delta}, and taking $\zeta_*>0$ sufficiently small,
we obtain \eqref{deltaZ1-weighted}. This completes the proof.
\end{proof}

We also need the following lower-regularity bound.

\begin{lem}\label{lem:deltaZ1-low}
It holds that
\begin{align}
\|\Lambda_\alpha(\tau)\delta Z_1\|
_{\widetilde{L}^{\infty}_t
(\dot{\B}^{\sigma_1-1}_{2,\infty})}^{\ell}
\lesssim
\|Z_0\|_{\dot{\B}^{\sigma_1}_{2,\infty}
\cap\dot{\B}^{\frac d2+1}_{2,1}}.
\label{deltaZ1-low}
\end{align}
\end{lem}

\begin{proof}
If $-1\leq\alpha\leq0$, then $\Lambda_\alpha=1$, and hence no
additional zeroth-order term appears in the equation for $\Lambda_\alpha\delta Z_1$.
If $0<\alpha\leq1$, then $\Lambda_\alpha=b^{-1}$ and
\[
\frac{\Lambda_\alpha'(t)}{\Lambda_\alpha(t)}
=
-\frac{b'(t)}{b(t)}
\leq0.
\]
Thus, the zeroth-order term generated by differentiating
$\Lambda_\alpha$ has a favorable sign and may be discarded in the
localized energy estimate.

As in the proof of Lemma~\ref{lem:deltaZ1-weighted}, the localized
estimate is performed on $I_{j,t}^{\ell}$. If this interval starts
at $t_j$, the boundary term at $t_j$ is controlled by
\eqref{deltaZ1-high} together with the uniform lower-regularity
bounds for $Z$ and $\mathcal N$.

Repeating the localized estimate \eqref{deltaZ1L2-new} for
$\Lambda_\alpha\delta Z_1$ at the level
$\dot{\B}^{\sigma_1-1}_{2,\infty}$ yields
\begin{align}
&\quad
\|\Lambda_\alpha(\tau)\delta Z_1\|
_{\widetilde{L}^{\infty}_t
(\dot{\B}^{\sigma_1-1}_{2,\infty})}^{\ell}
+
\|b(\tau)\Lambda_\alpha(\tau)\delta Z_1\|
_{L^1_t
(\dot{\B}^{\sigma_1+1}_{2,\infty})}^{\ell}\lesssim
\sum_{k=1}^d
\|\Lambda_\alpha(\tau)\widetilde{G}^k(Z)\|
_{L^1_t
(\dot{\B}^{\sigma_1}_{2,\infty})}^{\ell}.
\label{deltaZ1-low-energy}
\end{align}
Using \eqref{Gtilde-def}, we obtain
\begin{align}
&\sum_{k=1}^d
\|\Lambda_\alpha(\tau)\widetilde{G}^k(Z)\|
_{L^1_t
(\dot{\B}^{\sigma_1}_{2,\infty})}^{\ell}
\notag\\
&\lesssim
\|\Lambda_\alpha(\tau)W\|_{L^1_t
(\dot{\B}^{\sigma_1}_{2,\infty})}^{\ell}
+
\|b(\tau)\Lambda_\alpha(\tau)Z_2\|_{L^1_t
(\dot{\B}^{\sigma_1+1}_{2,\infty})}^{\ell}
+
\sum_{k=1}^d
\|b(\tau)\Lambda_\alpha(\tau)N_2^k(Z)\|_{L^1_t
(\dot{\B}^{\sigma_1+1}_{2,\infty})}^{\ell}.
\label{Gtilde-low-1}
\end{align}
By the definition of $\Lambda_\alpha$ and the uniform boundedness
of $b$ when $\alpha\leq0$, we have
\begin{align}
\|\Lambda_\alpha(\tau)W\|_{L^1_t
(\dot{\B}^{\sigma_1}_{2,\infty})}^{\ell}
&\lesssim
\left\|\frac{W}{b(\tau)}\right\|_{L^1_t
(\dot{\B}^{\sigma_1}_{2,\infty})}^{\ell},
\label{low-W-sigma}\\
\|b(\tau)\Lambda_\alpha(\tau)Z_2\|_{L^1_t
(\dot{\B}^{\sigma_1+1}_{2,\infty})}^{\ell}
&\lesssim
\|Z_2\|_{L^1_t
(\dot{\B}^{\sigma_1+1}_{2,\infty})}^{\ell}.
\label{low-Z2-sigma}
\end{align}
Furthermore, as $N_2^k(Z)=g_2^k(Z)Z_2$, the product estimates give
\begin{align}
&\sum_{k=1}^d
\|b(\tau)\Lambda_\alpha(\tau)N_2^k(Z)\|_{L^1_t
(\dot{\B}^{\sigma_1+1}_{2,\infty})}^{\ell}
\notag\\
&\lesssim
\|Z\|_{L^\infty_t(\dot{\B}^{\frac d2}_{2,1})}
\|b(\tau)\Lambda_\alpha(\tau)Z_2\|_{L^1_t
(\dot{\B}^{\sigma_1+1}_{2,\infty})}^{\ell}
\notag\\
&\quad+
\|b(\tau)^{\frac12}Z\|_{\widetilde L^2_t
(\dot{\B}^{\frac d2+1}_{2,1})}
\|b(\tau)^{\frac12}\Lambda_\alpha(\tau)Z_2\|_{\widetilde L^2_t
(\dot{\B}^{\sigma_1}_{2,\infty})}
\notag\\
&\lesssim
\|Z\|_{\mathcal E_t}\|Z\|_{\mathcal Y_t},
\label{N2-low-product}
\end{align}
where we employed $b(\tau)^{\frac12}\Lambda_\alpha(\tau)
\lesssim b(\tau)^{-\frac12}$.

Combining \eqref{Gtilde-low-1}, \eqref{low-W-sigma},
\eqref{low-Z2-sigma}, \eqref{N2-low-product},
\eqref{Xbound}, and \eqref{low:a}, we obtain
\begin{align}
\sum_{k=1}^d
\|\Lambda_\alpha(\tau)\widetilde{G}^k(Z)\|
_{L^1_t
(\dot{\B}^{\sigma_1}_{2,\infty})}^{\ell}
\lesssim
\|Z_0\|_{\dot{\B}^{\sigma_1}_{2,\infty}
\cap\dot{\B}^{\frac d2+1}_{2,1}}.
\label{Gtilde-low}
\end{align}
Combining \eqref{deltaZ1-low-energy} and \eqref{Gtilde-low} proves
\eqref{deltaZ1-low}.
\end{proof}

We now conclude the proof of the enhanced stability of the diffusion
profile. From \eqref{deltaZ1-high}, \eqref{deltaZ1-weighted}, and
\eqref{deltaZ1-low}, we have
\[
\|\Lambda_\alpha(t)\delta Z_1(t)\|
_{\dot{\B}^{\frac d2-1}_{2,1}}
\lesssim
\begin{cases}
(1+t)^{-\frac{1+\alpha}{2}
(\frac d2-\sigma_1)},
& -1<\alpha\leq1,\\[0.4em]
(\log(e+t))^{-\frac12
(\frac d2-\sigma_1)},
& \alpha=-1.
\end{cases}
\]
Interpolating the low-frequency part of the above estimate with
\eqref{deltaZ1-low}, we obtain, for
$\sigma_1<\sigma\leq\frac d2-1$,
\begin{align}
\|\Lambda_\alpha(t)\delta Z_1(t)\|
_{\dot{\B}^{\sigma}_{2,1}}^{\ell}
\lesssim
\begin{cases}
(1+t)^{-\frac{1+\alpha}{2}
(\sigma+1-\sigma_1)},
&-1<\alpha\leq1,\\[0.4em]
(\log(e+t))^{-\frac12
(\sigma+1-\sigma_1)},
&\alpha=-1.
\end{cases}
\label{deltaZ1-Lambda-decay}
\end{align}
For the high-frequency part, the definition of $J_t$ gives
\[
\|\Lambda_\alpha(t)\delta Z_1(t)\|
_{\dot{\B}^{\sigma}_{2,1}}^h
\lesssim
b(t)^{\frac d2-1-\sigma}
\|\Lambda_\alpha(t)\delta Z_1(t)\|
_{\dot{\B}^{\frac d2-1}_{2,1}}^h.
\]
For $-1<\alpha\leq1$, using
$b(t)\sim(1+t)^\alpha$ and
$\sigma\leq\frac d2-1$, we have $-\frac{1+\alpha}{2}\Big(\frac d2-\sigma_1\Big)
+\alpha\Big(\frac d2-1-\sigma\Big)
\leq
-\frac{1+\alpha}{2}(\sigma+1-\sigma_1)$. Hence, the high-frequency contribution decays at least as fast as
the right-hand side of \eqref{deltaZ1-Lambda-decay}. When
$\alpha=-1$, the factor
$b(t)^{\frac d2-1-\sigma}$ provides an additional algebraic decay.
Recalling the definition of $\Lambda_\alpha$, we finally obtain
\eqref{decayimprove2}.
This finishes the proof of Theorem \ref{thm3}.

\section{Application to compressible Euler system with time-dependent damping}\label{sect:7}

Our analysis can be applied to the compressible Euler equations with time-dependent damping in
$\mathbb{R}^{d}$ ($d\geq1$):
\begin{equation}\label{euler1}
\left\{
    \begin{aligned}
    &\partial_{t}\rho+\div(\rho u)=0,\\
    &\partial_{t}(\rho u)+\div (\rho u\otimes u)
      +\nabla P(\rho)+\frac{\rho u}{b(t)}=0,
    \end{aligned}
\right.
\end{equation}
supplemented with the initial condition
\begin{align}
&(\rho, \rho u)|_{t=0}=(\rho_0, \rho_0 u_0).\label{euler0}
\end{align}
We consider solutions $(\rho,u)$ in a neighborhood of the equilibrium state
$(\bar{\rho},0)$, where $\bar{\rho}>0$ is a constant background density.
Here $\rho=\rho(t,x)\ge0$ denotes the density and
$u=u(t,x)\in\mathbb{R}^{d}$ the velocity field.
The time-dependent friction coefficient is given by $1/b(t)$ with $b(t)>0$.
The pressure function $P(\rho)$ is assumed to satisfy
$P\in C^{\infty}(\mathbb{R}_{+})$ and $P'(\bar{\rho})>0$.

Let $V=(\rho,m)$ with $m=\rho u$.
System \eqref{euler1} admits the convex mechanical entropy
\begin{equation}\label{entropy-euler}
\eta(V)
=\frac{|m|^{2}}{2\rho}+h(\rho)
\quad\text{with}\quad
h(\rho):=\int_{\bar\rho}^{\rho}\int_{\bar\rho}^{s}\frac{P'(s')}{s'}\,ds'\,ds .
\end{equation}
The associated entropy variable is defined by
$U:=\big(D_V\eta(V)\big)^{\top}=(U_1,U_2)^{\top}$, that is,
\[
U=
\begin{pmatrix}
-\dfrac{|m|^2}{2\rho^{2}}+h'(\rho)
\\[2pt]
\dfrac{m}{\rho}
\end{pmatrix}.
\]
Since $\eta$ is strictly convex, the mapping $V\mapsto U$ is locally invertible,
and we may write $V=V(U)$.

In terms of the entropy variable $U$, system \eqref{euler1} can be rewritten in
the symmetric hyperbolic form
\begin{equation}\label{sym-hyp-U}
\widetilde A^0(U)\,\partial_t U
+\sum_{j=1}^d \widetilde A^j(U)\,\partial_{x_j}U
=\frac{1}{b(t)}\,\widetilde H(U),
\end{equation}
where $\widetilde A^j(U)$ $(j=0,\dots,d)$ are symmetric matrices and
$\widetilde H(U)$ is given by
{\small
\begin{equation*}
\renewcommand{\arraystretch}{1.2}
\begin{aligned}
\widetilde A^0(U)
&=
\begin{pmatrix}
\dfrac{\rho}{P'(\rho)} &
\dfrac{m^{\!\top}}{P'(\rho)}
\\[6pt]
\dfrac{m}{P'(\rho)} &
\rho I_d+\dfrac{m\otimes m}{\rho\,P'(\rho)}
\end{pmatrix},
\\[2mm]
\widetilde A^j(U)
&=
\begin{pmatrix}
\dfrac{m_j}{P'(\rho)} &
\rho\,e_j^{\!\top}+\dfrac{m_j}{\rho\,P'(\rho)}\,m^{\!\top}
\\[8pt]
\rho\,e_j+\dfrac{m_j}{\rho\,P'(\rho)}\,m &
m_j I_d
+\;m\otimes e_j
+\;e_j\otimes m
+\;\dfrac{m_j}{\rho\,P'(\rho)}\,m\otimes m
\end{pmatrix},
\qquad j=1,\dots,d,
\\[2mm]
\widetilde H(U)
&=
\begin{pmatrix}
0\\
-m
\end{pmatrix}.
\end{aligned}
\end{equation*}}
\noindent
Here $e_j$ denotes the $j$-th canonical basis vector of $\mathbb{R}^d$.

Let $\bar V=(\bar\rho,0)$ and $\bar U:=U(\bar V)$.
Denote the perturbation $r:=\rho-\bar\rho$.
Instead of freezing the variable coefficients in \eqref{sym-hyp-U}, we
linearize \eqref{euler1} around $(\bar\rho,0)$ while keeping all nonlinear
terms in divergence form.
Recalling that $m=\rho u$, the system \eqref{euler1} is equivalent to the linearized constant-coefficient system with
divergence-form nonlinear remainders:
\begin{equation}\label{lin-euler-rm}
\left\{
\begin{aligned}
&\partial_t r+\div m=0,\\
&\partial_t m+P'(\bar\rho)\nabla r+\frac{1}{b(t)}\,m
= -\sum_{j=1}^d \partial_{x_j}\Big(\frac{m_j}{\rho}\,m+\Pi(\rho)\,e_j\Big).
\end{aligned}
\right.
\end{equation}
with $\Pi(\rho):=P(\rho)-P(\bar\rho)-P'(\bar\rho)(\rho-\bar\rho)$.

Next, we normalize \eqref{lin-euler-rm} by a constant symmetrizer at the
equilibrium state. Define the new unknown
\begin{equation}\label{defZ-euler}
Z:=
\begin{pmatrix}
Z_1\\ Z_2
\end{pmatrix}
:=
\begin{pmatrix}
\sqrt{\dfrac{P'(\bar\rho)}{\bar\rho}}\;r\\[6pt]
\dfrac{1}{\sqrt{\bar\rho}}\;m
\end{pmatrix}.
\end{equation}
Then \eqref{lin-euler-rm} is equivalent to the normalized constant-coefficient
symmetric system
\begin{align}\label{eulerre}
\partial_t Z
+\sum_{k=1}^d A^k\,\partial_{x_k}Z
+\frac{1}{b(t)}\,B\,Z
=
-\sum_{k=1}^d \partial_{x_k}\mathbf N^{k}(Z),
\end{align}
where
\begin{equation*}
\begin{aligned}
A^k
&=
\begin{pmatrix}
0 & \sqrt{P'(\bar\rho)}\,e_k^{\!\top}\\
\sqrt{P'(\bar\rho)}\,e_k & 0
\end{pmatrix},
\qquad
B
=
\begin{pmatrix}
0&0\\
0&I_d
\end{pmatrix},
\\[2mm]
\mathbf N^{k}(Z)
&:=
\begin{pmatrix}
0\\[3pt]
\dfrac{1}{\sqrt{\bar\rho}}
\Big(\dfrac{m_k}{\rho}\,m+\Pi(\rho)\,e_k\Big)
\end{pmatrix},
\qquad k=1,\dots,d.
\end{aligned}
\end{equation*}
%Here %$m=\sqrt{\bar\rho}\,Z_2$ and
%$\rho=\bar\rho+\sqrt{\dfrac{\bar\rho}{P'(\bar\rho)}}\,Z_1$, so that
%$\mathbf N^{k}(Z)$ is a smooth function of $Z$ in a neighborhood of $Z=0$.

According to Theorems \ref{thm0}-\ref{thm3}, we have the global existence and decay estimates of solutions for the system \eqref{euler1}.

\begin{thm}\label{thm:euler1}
Assume that $b(t)$ fulfills $b(t)>0$ and $b'(t)\leq 0$ or  $b'(t)>0$ and $\limsup_{t\to\infty}|b'(t)|\leq c^*$,
where $c^*>0$ depends only on $d$, $\bar{\rho}$ and $P'(\bar{\rho})$. If $(\rho_0, m_0)$ with $m_0=\rho_0 u_0$ satisfies
\begin{align}
\|(\rho_0-\bar{\rho},m_0)\|_{\dot{\B}^{\frac{d}{2}}_{2,1}\cap \dot{\B}^{\frac{d}{2}+1}_{2,1}}\leq \beta_0
\end{align}
for some constant $\beta_0>0$, then the system \eqref{euler1} has a global-in-time classical solution $(\rho, m)$ with $m=\rho u$ satisfying 
\begin{alignat*}{3}
(\rho-\bar{\rho},m)
&\in \mathcal{C}_b(\mathbb{R}_+;\dot{\B}_{2,1}^{\frac{d}{2}}),\qquad&
b(t)(\rho-\bar{\rho},m)
&\in \mathcal{C}_b(\mathbb{R}_+;\dot{\B}_{2,1}^{\frac{d}{2}+1}),\qquad&
b(t)^{\frac{1}{2}}(\rho-\bar{\rho},m)
&\in L^2(\mathbb{R}_+;\dot{\B}_{2,1}^{\frac{d}{2}+1}),\\
b(t)(\rho-\bar{\rho},m)^{\ell}
&\in L^1(\mathbb{R}_+;\dot{\B}_{2,1}^{\frac{d}{2}+2}),\qquad&
m
&\in L^1(\mathbb{R}_+;\dot{\B}_{2,1}^{\frac{d}{2}+1}),\qquad&
(\rho-\bar{\rho},m)^h
&\in L^1(\mathbb{R}_+;\dot{\B}_{2,1}^{\frac{d}{2}+1}),\\
\frac{m}{\sqrt{b(t)}}
&\in L^2(\mathbb{R}_+;\dot{\B}_{2,1}^{\frac{d}{2}}),\qquad&
\partial_t\rho,\partial_t m
&\in L^1(\mathbb{R}_+;\dot{\B}_{2,1}^{\frac{d}{2}}),\qquad&
\nabla P(\rho)+\frac{m}{b(t)}
&\in L^1(\mathbb{R}_+;\dot{\B}_{2,1}^{\frac d2}).
\end{alignat*}

and
\begin{equation*}
\begin{aligned}
&\quad \|(\rho-\bar{\rho},m)\|_{\widetilde{L}_t^\infty(\dot{\B}_{2,1}^{\frac d2})}
+\|b(\tau)(\rho-\bar{\rho},m)\|_{\widetilde{L}_t^\infty(\dot{\B}_{2,1}^{\frac d2+1})}+\|b(\tau)(\rho-\bar{\rho})\|^\ell_{L_t^1(\dot{\B}_{2,1}^{\frac d2+2})}+\|(\rho-\bar{\rho},m)\|^h_{L_t^1(\dot{\B}_{2,1}^{\frac d2+1})}\\
&\quad\quad+\|m\|_{L_t^1(\dot{\B}_{2,1}^{\frac d2+1})}+\|b(\tau)^{-\frac12} m\|_{\widetilde{L}_t^2(\dot{\B}_{2,1}^{\frac d2})}+\|b(\tau)^{\frac{1}{2}}(\rho-\bar{\rho},m)\|_{\widetilde{L}_t^2(\dot{\B}_{2,1}^{\frac d2+1})}\\
&\quad\quad+\|\partial_t (\rho-\bar{\rho},m)\|_{L_t^1(\dot{\B}_{2,1}^{\frac d2})}+\Big\|\nabla P(\rho)+\frac{m}{b(t)}\Big\|_{L_t^1(\dot{\B}_{2,1}^{\frac d2})}\\
&\lesssim \|(\rho_0-\bar{\rho},m_0)\|_{\dot{\B}^{\frac{d}{2}}_{2,1}\cap \dot{\B}^{\frac{d}{2}+1}_{2,1}}.
\end{aligned}
\end{equation*}
\end{thm}

\begin{thm}\label{thm:euler2}
Let $(\rho,m)$ be the global solution to the system \eqref{euler1}--\eqref{euler0} with initial data $(\rho_0,m_0)$ obtained in Theorem~\ref{thm:euler1}. Consider $b(t)=\frac{(1+t)^{\alpha}}{K}$ with $-1\leq \alpha\leq 1$, where  $K>0$ if $-1\leq \alpha<1$, and $K>K_0$ for some sufficiently large constant $K_0>0$ when $\alpha=1$.
Furthermore, assume that $
(\rho_0-\bar{\rho},m_0)\in \dot{\B}^{\sigma_1}_{2,\infty}$ with $-\frac{d}{2}\leq \sigma_1<\frac{d}{2}$. 
Then, it holds that
\begin{itemize}
\item \textbf{Case 1: $-1<\alpha\leq 1$.}
For any $\sigma_1<\sigma\leq \frac{d}{2}$ and $\sigma_1<\sigma'\leq \frac{d}{2}-1$,
\begin{align*}
\|(\rho-\bar{\rho},u)(t)\|_{\dot{\B}^{\sigma}_{2,1}}
&\lesssim (1+t)^{-\frac{1+\alpha}{2}(\sigma-\sigma_1)},\\
\|u(t)\|_{\dot{\B}^{\sigma'}_{2,1}}
&\lesssim (1+t)^{-\frac{1+\alpha}{2}(\sigma'-\sigma_1)-\frac{1-\alpha}{2}}.
\end{align*}

\item \textbf{Case 2: $\alpha=-1$.}
For any $\sigma_1<\sigma\leq \frac{d}{2}$ and $\sigma_1<\sigma'\leq \frac{d}{2}-1$,
\begin{align*}
\|(\rho-\bar{\rho},u)(t)\|_{\dot{\B}^{\sigma}_{2,1}}
&\lesssim \,\log(e+t)^{-\frac{1}{2}(\sigma-\sigma_1)},\\
\|u(t)\|_{\dot{\B}^{\sigma'}_{2,1}}
&\lesssim \,\log(e+t)^{-\frac{1}{2}(\sigma'-\sigma_1)-\frac{1}{2}}(1+t)^{-1}.
\end{align*}
\end{itemize}
%Here $C>0$ denotes a constant independent of time.
\end{thm}

Finally, we establish the asymptotic equivalence of solutions between the Euler equations \eqref{eulerre} (with $P(\rho)=A\rho$) and the linear diffusion equation with a time-dependent coefficient:
\begin{align}
\partial_t\rho^*-b(t)A\Delta \rho^*=0,\quad \rho^*|_{t=0}= \rho_0. \label{ZL:euler}
\end{align}

\begin{thm}\label{thmEuler}
Let $(\rho,m)$ be the global solution to the system \eqref{euler1}--\eqref{euler0} with initial data $(\rho_0,m_0)$ obtained in Theorem~\ref{thm:euler1}. Consider $P(\rho)=A\rho$ and $b(t)=\frac{(1+t)^{\alpha}}{K}$ with $-1\leq \alpha\leq 1$.
Assume $
(\rho_0-\bar{\rho},m_0)\in \dot{\B}^{\sigma_1}_{2,\infty}$ with $-\frac{d}{2}\leq \sigma_1<\frac{d}{2}$. Then, for all $\sigma_1<\sigma\leq \frac{d}{2}-1$, we have
\begin{align*}
 \|(\rho-\rho^*)(t)\|_{\dot{\B}^{\sigma}_{2,1}}&\leq C
\begin{cases}
(1+t)^{-\frac{1+\alpha}{2}(\sigma-\sigma_1)
-\frac{1-\alpha}{2}},
& 0<\alpha\leq 1,\\[0.5em]
(1+t)^{-\frac{1+\alpha}{2}(\sigma+1-\sigma_1)},
& -1<\alpha\leq 0,\\[0.5em]
(\log(e+t))^{-\frac12(\sigma-\sigma_1+1)},
& \alpha=-1.
\end{cases}
 \end{align*}
where the profile $\rho^*$ is the solution to \eqref{ZL:euler}. 
\end{thm}

\begin{rem}\normalfont
The pressure law \(P(\rho)=A\rho\) is required to match the structural condition \(N_2^k(Z_1,0)=0\) in \eqref{structural2}. It should be possible to consider more general pressure laws; see \cite{SXZ} for the case of constant damping coefficients. However, a more refined analysis is needed to handle the nonlinear nature of the pressure.
\end{rem}

\begin{rem}\normalfont
Compared with our previous work \cite{CBPSZ}, where the analysis was devoted to the global existence and relaxation limit in the case of a  power-law damping
coefficient, in Theorems \ref{thm:euler2} and \ref{thmEuler}, we address more general 
coefficient \(b(t)\) and also describe the corresponding time decay rates and
asymptotic stability.
\end{rem}

\appendix

\renewcommand{\thesection}{Appendix~\Alph{section}}

\section{Normal form}\label{AppendixNF}
In this section, we derive the constant-coefficient formulation \eqref{eq0}
from the original system \eqref{system}. We first recall the normal-form
change of variables from \cite{KY} and then linearize at an equilibrium while
retaining the nonlinear terms as quadratic remainders.

\begin{defi}\label{dfn:normal}
{\rm (Normal form)} A symmetric dissipative system written in a variable
$\mathcal V=(\mathcal V_1,\mathcal V_2)^{\top}$ is said to be in normal form
if its time matrix is block diagonal with respect to the orthogonal
decomposition $\mathbb R^n=\mathcal M\oplus\mathcal M^{\perp}$.
\end{defi}

As in \cite{KY}, write $V=(V_1,V_2)^{\top}$ and
$U=(U_1,U_2)^{\top}$ according to
$\mathbb R^n=\mathcal M\oplus\mathcal M^{\perp}$, where
$U_2=\big(D_{V_2}\eta(V)\big)^{\top}$. Introduce the mapping
$V\mapsto\mathcal V$ by
\begin{equation}\label{2.8}
\begin{aligned}
  \mathcal{V}= \begin{pmatrix}
   \mathcal{V}_1 \\ \mathcal{V}_2
    \end{pmatrix}:=
    \begin{pmatrix}
   V_1 \\ U_2
    \end{pmatrix}.
    \end{aligned}
\end{equation}
Locally, this mapping is a diffeomorphism. Denote its inverse by
$V=V(\mathcal V)$, set $U(\mathcal V):=U(V(\mathcal V))$, and write
$T(\mathcal V):=D_{\mathcal V}U(\mathcal V)$. Multiplying
\eqref{entropeq} by $T(\mathcal V)^{\top}$ yields
\begin{equation}
\begin{aligned}
&\mathcal{A}^{0}(\mathcal{V})\partial_{t}\mathcal{V}+\sum_{k=1}^{d}\mathcal{A}^{k}(\mathcal{V})\partial_{x_{k}}\mathcal{V}=\frac{1}{b(t)}\mathcal{H}(\mathcal{V})
\label{normal}
\end{aligned}\end{equation}
with
\begin{equation}\label{2.10}
\begin{aligned}
\mathcal{A}^{0}(\mathcal{V})&=T(\mathcal V)^{\top}
 \widetilde{A}^0(U(\mathcal V))T(\mathcal V),\\
\mathcal{A}^{i}(\mathcal{V})&=T(\mathcal V)^{\top}
 \widetilde{A}^i(U(\mathcal V))T(\mathcal V),\\
\mathcal{H}(\mathcal{V})&=T(\mathcal V)^{\top}
 \widetilde{H}(U(\mathcal V)).
\end{aligned}
\end{equation}
According to \cite[Theorem 2.4]{KY}, system \eqref{normal} is a symmetric
dissipative system in normal form. Let $\bar V\in\mathcal E$ satisfy
$G(\bar V)=0$, and let $\bar U$ and $\bar{\mathcal V}$ be the corresponding
states. Since $\widetilde H(\bar U)=0$, differentiation of \eqref{2.10}
at $\bar{\mathcal V}$ gives
\begin{equation}\label{2.11}
\begin{aligned}
\mathcal L(\bar{\mathcal V})
:=-D_{\mathcal V}\mathcal H(\bar{\mathcal V})
=\bar T^{\top}\widetilde B(\bar U)\bar T,
\qquad \bar T:=T(\bar{\mathcal V}).
\end{aligned}
\end{equation}
Moreover, the normal-form structure implies
\begin{equation}\label{2.12}
\begin{aligned}
\mathcal A^0(\bar{\mathcal V})
=\begin{pmatrix}
\mathcal A^0_{11}&0\\0&\mathcal A^0_{22}
\end{pmatrix}>0,
\qquad
\mathcal L(\bar{\mathcal V})
=\begin{pmatrix}
0&0\\0&\mathcal L_{22}
\end{pmatrix},
\qquad \mathcal L_{22}>0.
\end{aligned}
\end{equation}

We now pass from this variable-coefficient normal form to the
constant-coefficient perturbation system. Set
\begin{equation*}
H_0:=D_UV(\bar U)=\widetilde A^0(\bar U),
\qquad V=\bar V+H_0\mathbf V,
\end{equation*}
and define the Taylor remainders
\begin{equation*}
\begin{aligned}
\mathcal R_G(V)&:=G(V)-G(\bar V)-D_VG(\bar V)(V-\bar V),\\
\mathcal R_F^k(V)&:=F^k(V)-F^k(\bar V)
 -D_VF^k(\bar V)(V-\bar V).
\end{aligned}
\end{equation*}

Because
\begin{equation*}
D_VF^k(\bar V)H_0=\widetilde A^k(\bar U),
\qquad
-D_VG(\bar V)H_0=\widetilde B(\bar U),
\end{equation*}
the Taylor expansion of \eqref{system} at $\bar V$ gives
\begin{equation}\label{tildeV111}
\begin{aligned}
H_0\partial_t\mathbf V
+\sum_{k=1}^d\widetilde A^k(\bar U)\partial_{x_k}\mathbf V
+\frac{1}{b(t)}\widetilde B(\bar U)\mathbf V
=\frac{1}{b(t)}\mathcal R_G(V)
-\sum_{k=1}^d\partial_{x_k}\mathcal R_F^k(V).
\end{aligned}
\end{equation}
Set $\mathbf V=\bar T\mathcal Z$ and multiply \eqref{tildeV111} from the
left by $\bar T^{\top}$. By \eqref{2.10} and \eqref{2.11}, we obtain
\begin{equation}\label{tildeV}
\begin{aligned}
\mathcal A^0(\bar{\mathcal V})\partial_t\mathcal Z
+\sum_{k=1}^d\mathcal A^k(\bar{\mathcal V})\partial_{x_k}\mathcal Z
+\frac{1}{b(t)}\mathcal L(\bar{\mathcal V})\mathcal Z
=\frac{1}{b(t)}\bar T^{\top}\mathcal R_G(V)
-\sum_{k=1}^d\partial_{x_k}
 \big(\bar T^{\top}\mathcal R_F^k(V)\big).
\end{aligned}
\end{equation}

Let
\begin{equation*}
S:=\mathcal A^0(\bar{\mathcal V})^{1/2},
\qquad Z:=S\mathcal Z.
\end{equation*}
Multiplying \eqref{tildeV} by $S^{-1}$ and substituting
$\mathcal Z=S^{-1}Z$ give
\begin{equation*}
\partial_tZ+\sum_{k=1}^d A^k\partial_{x_k}Z
+\frac{1}{b(t)}BZ
=\frac{1}{b(t)}\mathbf Q(Z)
+\sum_{k=1}^d\partial_{x_k}\mathbf N^k(Z),
\end{equation*}
where
\begin{align*}
A^k&:=S^{-1}\mathcal A^k(\bar{\mathcal V})S^{-1},\\
B&:=S^{-1}\mathcal L(\bar{\mathcal V})S^{-1},\\
\mathbf Q(Z)&:=S^{-1}\bar T^{\top}\mathcal R_G(V(Z)),\\
\mathbf N^k(Z)&:=-S^{-1}\bar T^{\top}\mathcal R_F^k(V(Z)),
\end{align*}
with
\begin{equation*}
V(Z)=\bar V+H_0\bar T S^{-1}Z.
\end{equation*}
The matrices $A^k$ are symmetric. By \eqref{2.12}, $S$ is block
diagonal and
\begin{equation*}
B=\begin{pmatrix}0&0\\0&L_2\end{pmatrix},
\qquad L_2>0.
\end{equation*}
%The normal-form source structure also gives $\mathbf Q(Z)=(0,Q_2(Z))^{\top}$. 
Since $G$ takes values in $\mathcal M^\perp$, so does its Taylor
remainder $\mathcal R_G$. Moreover, $\bar T^{\top}$ preserves
$\mathcal M^\perp$ because $\mathcal V_2=U_2$, while $S^{-1}$ preserves
$\mathcal M^\perp$ because it is block diagonal. Therefore,
\[
\mathbf Q(Z)
=S^{-1}\bar T^{\top}\mathcal R_G(V(Z))
\in\mathcal M^\perp,
\qquad
\mathbf Q(Z)=
\begin{pmatrix}
0\\ Q(Z)
\end{pmatrix}.
\]
Finally,
$\mathcal R_G$ and $\mathcal R_F^k$ vanish to second order at $\bar V$;
hence $\mathbf Q$ and $\mathbf N^k$ are quadratic near $Z=0$ and have
the representations stated in \eqref{R0Ri}. This is precisely system
\eqref{eq0}.

\section{Construction of the corrector}\label{sec:appendixBZ}

We explain precisely how to define the corrector appropriately with the choice of constants. The proof follows similar arguments in \cite{BZ,c2}. Assume that $z$ satisfies the linear system \eqref{eqlin}.
Our goal is to establish inequality \eqref{xianxing}.
The key point is to control the right-hand side terms appearing in
\eqref{eqcorrec}.

\noindent
$\bullet$ \emph{Terms}
\(
\mathcal{I}_k^1
:=\frac{1}{b(t)}(BA^{k-1}Bz\cdot BA^kz),
\quad k\in\{1,\dots,n-1\}.
\)

Since $A$ is uniformly bounded on $\SSS^{d-1}$, we have
\begin{align*}
\varepsilon_k|\mathcal{I}_k^1|
&\le C\frac{1}{b(t)}\varepsilon_k |Bz|\,|BA^kz|\le \frac{\varepsilon_0}{16nr\,b^2(t)}|Bz|^2
   +Cr\,\frac{\varepsilon_k^2}{\varepsilon_0}|BA^kz|^2,
\end{align*}
where Young's inequality has been used and $\varepsilon_0$ will be fixed later.

\noindent
$\bullet$ \emph{Terms}
\(
\frac{1}{b(t)}(BA^{k-1}z\cdot BA^kBz),
\quad k\in\{2,\dots,n-1\}.
\)

These terms can be bounded in the same way as $\mathcal{I}_k^1$ and hence
lead to estimates of the same form.

\noindent
$\bullet$ \emph{The term}
\(
\frac{1}{b(t)}\varepsilon_1(Bz\cdot BABz).
\)

It is directly controlled by
\[
\frac{1}{b(t)}\varepsilon_1 |Bz\cdot BABz|
\le C\frac{\varepsilon_1}{b(t)}|Bz|^2.
\]

\noindent
$\bullet$ \emph{Terms}
\(
\mathcal{I}_k^2
:= r(BA^{k-1}z\cdot BA^{k+1}z),
\quad k\in\{1,\dots,n-2\}.
\)

Using Young's inequality, we obtain
\begin{align*}
\varepsilon_k|\mathcal{I}_k^2|
&\le \varepsilon_k r |BA^{k-1}z|\,|BA^{k+1}z| \le \frac{\varepsilon_{k-1}}{16}r|BA^{k-1}z|^2
   +Cr\,\frac{\varepsilon_k^2}{\varepsilon_{k-1}}|BA^{k+1}z|^2.
\end{align*}
In order for the second term to be absorbed by the left-hand side,
we impose the condition
\begin{align}
4C\varepsilon_k^2\le \varepsilon_{k-1}\varepsilon_{k+1}.
\label{condi1}
\end{align}

\noindent
$\bullet$ \emph{The last term}
\(
\mathcal{I}_{n-1}^2
:= \varepsilon_{n-1}r(BA^{n-2}z\cdot BA^nz).
\)

By the Cayley--Hamilton theorem, there exist coefficients $c_\omega^j$,
uniformly bounded on $\SSS^{d-1}$, such that $A_\omega^n=\sum_{j=0}^{n-1}c_\omega^j A_\omega^j$. Therefore,
\begin{align*}
|\mathcal{I}_{n-1}^2|
&\le C\varepsilon_{n-1}r\sum_{j=0}^{n-1}|BA^{n-2}z|\,|BA^jz| \le \sum_{j=0}^{n-1}
\left(
\frac{C\varepsilon_{n-1}^2 r}{\varepsilon_j}|BA^{n-2}z|^2
+\frac{\varepsilon_j r}{16}|BA^jz|^2
\right).
\end{align*}
Hence, we further require
\begin{align}
4Cn\,\varepsilon_{n-1}^2
\le \varepsilon_j\varepsilon_{n-2},
\qquad j=0,\dots,n-1.
\label{condi2}
\end{align}
It is possible to choose the parameters
$\varepsilon_k$ so that both \eqref{condi1} and \eqref{condi2} hold.
More precisely, we take $\varepsilon_k=\varepsilon^{m_k}$ with $\varepsilon>0$
small enough, $m_0=0$ (corresponding to $\varepsilon_0=1$),
and exponents $m_1,\dots,m_{n-1}$ satisfying, for some sufficiently small
$\delta>0$,
\begin{align*}
m_k &\ge \frac{m_{k-1}+m_{k+1}}{2}+\delta,
\qquad k=1,\dots,n-2,\\
m_{n-1} &\ge \frac{m_k+m_{n-2}}{2}
+\delta.
\end{align*}

Collecting all the above estimates, we finally obtain
\begin{align*}
\frac{d}{dt}\mathcal{I}
+\sum_{k=1}^{n-1}\varepsilon_k r|BA^kz|^2
\leqslant\;
&\Big(\frac{(2n-3)\varepsilon_0}{16nr\,b^2(t)}
      +\frac{C\varepsilon_1}{b(t)}
      +\frac{2\varepsilon_0}{16}r\Big)|Bz|^2 \\
&+\sum_{j=1}^{n-3}r
\Big(C\frac{\varepsilon_j^2+\varepsilon_{j+1}^2}{\varepsilon_0}
     +\frac{2\varepsilon_j}{16}
     +C\frac{\varepsilon_j^2}{\varepsilon_{j-1}}\Big)|BA^jz|^2 \\
&+r\Big(\sum_{j=0}^{n-1}C\frac{\varepsilon_{n-1}^2}{\varepsilon_j}
     +C\frac{\varepsilon_{n-2}^2+\varepsilon_{n-1}^2}{\varepsilon_0}
     +\frac{2\varepsilon_{n-2}}{16}
     +C\frac{\varepsilon_{n-2}^2}{\varepsilon_{n-3}}\Big)|BA^{n-2}z|^2 \\
&+r\Big(C\frac{\varepsilon_{n-1}^2}{\varepsilon_0}
     +\frac{2\varepsilon_{n-1}}{16}
     +C\frac{\varepsilon_{n-1}^2}{\varepsilon_{n-2}}\Big)|BA^{n-1}z|^2.
\end{align*}
Using conditions \eqref{condi1} and \eqref{condi2},
all the right-hand side terms can be absorbed,
and inequality \eqref{xianxing} follows.

\section{Technical tools}
We have the following Gronwall-type lemma.
\begin{lem}\label{lem2}
    Let $X:[T_0,T_1]\rightarrow \R_+$ be a continuous function such that $X^2$ is differentiable. Assume that there exists a constant $c\geqslant 0$ and  $C^1$ functions $f,g$ with $g>0$, $f>0,0\leqslant f'\leqslant  \frac{c}{2}g$ on $[T_0,T_1]$ and a measurable function $A:[T_0,T_1]\rightarrow \R_+$ such that 
    \begin{align*}
        \frac{d}{dt}(f(t)X^2)+c g(t)X^2\leqslant AX \quad  a.e.\  on \ [T_0,T_1].
    \end{align*}
    Then, for all $t\in [T_0,T_1]$, we have 
    \begin{align*}
     2f(t)X(t)+c'\int_{T_0}^t g(\tau)X(\tau)d\tau\leqslant 2f(T_0)X(T_0)+\int_{T_0}^tA(\tau)d\tau,
    \end{align*}
    with $c'=\frac{c}{2}$.
\end{lem}
\begin{proof}
The proof of the case where $f$ is constant can be found in the appendix of \cite{c1}. Here we only need to note that we have 
\begin{align*}
\frac{d}{dt}\bigl(f(t)X^2\bigr)+cg(t)X^2
&=f'(t)X^2+2f(t)X\frac{dX}{dt}+cg(t)X^2\\
&=2X\frac{d}{dt}(fX)+(cg(t)-f'(t))X^2\\
&\geq 2X\frac{d}{dt}(fX)+\frac c2 g(t)X^2,
\end{align*}
where we used $0\leqslant f'\leqslant  \frac{c}{2}g$. Then, we can divide by $X$ on both sides and integrate in time.  The general case follows from a standard regularization argument.
\end{proof}

The following elementary factorization lemma will be used to exploit the structural vanishing conditions in \eqref{conditionAVbar}; see also \cite[Proposition~3.1]{KY}.

\begin{lem}\label{lem:factorization-Z2}
Let $n=n_1+n_2$ and write $Z=(Z_1,Z_2)\in\R^{n_1}\times\R^{n_2}$.
Let $G$ be a $C^{1}$ vector-valued map defined in a neighborhood $\mathcal O$ of
$0\in\R^{n}$.
Assume that
\[
G(Z_1,0)=0 \qquad \text{for all }(Z_1,0)\in\mathcal O.
\]
Then there exists a continuous matrix-valued map $G_1$
such that
\[
G(Z)=G_1(Z)\,Z_2
\qquad \text{for all $Z\in \mathcal O$.}
\]
\end{lem}

\begin{proof}
Fix $(Z_1,Z_2)\in\mathcal O$ and define $\gamma:[0,1]\to\mathbb{R}^n$ by $\gamma(s):=G(Z_1,sZ_2)$. Owing to $\gamma(0)=G(Z_1,0)=0$, one has
\[
G(Z_1,Z_2)=\gamma(1)-\gamma(0)=\int_0^1 \gamma'(s)\,ds.
\]
Since $G\in C^{1}$, the chain rule yields $\gamma'(s)=D_{Z_2}G(Z_1,sZ_2)\,Z_2$. Therefore, we have
\[
G(Z_1,Z_2)=G_1(Z)Z_2\quad\text{with}\quad G_1(Z)=\left(\int_0^1 D_{Z_2}G(Z_1,sZ_2)\,ds\right).
\]
\end{proof}

The proof of the following  inequality may be found in, e.g., \cite[Chap. 2]{Bahouri2011}.
\begin{lem} \label{Commutateur1}
There exists a constant $C$ such that  for all
$1\leq p,q,r\leq\infty$ such that $\frac{1}{p}+\frac{1}{q}=\frac{1}{r},$ 
all functions $a$ with a gradient  in $L^p,$ and $b$ in $L^q,$ 
we have
$$\norme{[\dot{\Delta}_j,a]b}_{L^r}\leq C2^{-j}\norme{\nabla a}_{L^p}\norme{b}_{L^q}
\quad\hbox{for all }\ j\in\Z.$$
\end{lem}
The following result is  proved, for instance, in \cite[Chap. 2]{Bahouri2011}.  %and \cite{RLG}, respectively.
\begin{prop}\label{C1}    %The following inequality holds for 
For all $1\leq p\leq\infty$
%\begin{itemize}\item  
and  $-\min(d/p,d/p')<s\leq d/p,$ we have
 \begin{equation}\label{eq:com1}
2^{js}\norme{[w,\dot{\Delta}_j]\nabla v}_{L^p}\leq Cc_j\norme{\nabla w}_{\dot{\B}^{\frac{d}{p}}_{p,1}}\norme{v}_{\dot{\B}^{s}_{p,1}}  \with\sum_{j\in\mathbb{Z}}c_j=1.
\end{equation}
%\item   If $-d/2\leq s<{d}/{2}+1$, then
%\begin{equation}\label{eq:com3} 
%\sup_{j\in\Z} 2^{js}\|[w,\ddj]\nabla v\|_{L^2}\leq C\|\nabla w\|_{\dot\B^{\frac d2}_{2,1}} \|v\|_{\dot\B^{s}_{2,\infty}}.\end{equation}
%\item If  $s\in \left]-1-\frac{d}{2},\frac{d}{2}\right]$, then we have
 %\begin{equation}\label{eq:com2}
% \norme{\nabla ([w,\dot{\Delta}_j]v)}_{L^2}\leq Cc_j2^{-js}\norme{\nabla  w}_{\dot{\B}^{\frac{d}{2}}_{2,1}}\norme{v}%_{\dot{\B}^{s}_{2,1}} \with\sum_{j\in\mathbb{Z}}c_j=1.\end{equation}
 % \end{itemize}
 \end{prop}
 
 The following product laws in Besov spaces have been used several times.
 \begin{prop} \label{LP} Let $(s,p,r)$ be in $]0,\infty[\times[1,\infty]^2.$ Then, 
 $\dot{\B}^{s}_{p,r}\cap L^\infty$ is an algebra and we have
\begin{equation}\label{eq:prod1}
\norme{ab}_{\dot{\B}^{s}_{p,r}}\leq C\bigl(\norme{a}_{L^\infty}\norme{b}_{\dot{\B}^{s}_{p,r}}+\norme{a}_{\dot{\B}^{s}_{p,r}}\norme{b}_{L^\infty}\bigr)\cdotp
\end{equation}
If, furthermore, $-\min(d/p,d/p')<s\leq d/p,$ then the following inequality holds:
\begin{equation}\label{eq:prod2}
\|ab\|_{\dot\B^{s}_{p,1}}\leq C\|a\|_{\dot\B^{\frac dp}_{p,1}}\|b\|_{\dot\B^{s}_{p,1}}.
\end{equation}
Finally,  if  $-d/p<\sigma_1\leq \min(d/p,d/p')$, then the following inequality holds true: 
\begin{equation}\label{eq:prod3} 
\norme{ab}_{\dot{\B}^{-\sigma_1}_{p,\infty}}\leq C  \norme{a}_{\dot{\B}^{\frac{d}{p}}_{p,1}}\norme{b}_{\dot{\B}^{-\sigma_1}_{p,\infty}}.
\end{equation}
\end{prop}
The following result for left composition can be found in \cite{Bahouri2011}.
\begin{prop}\label{Composition}
Let $p\geq 1$ and let $f$ be a function in $\mathcal{C}^\infty(\mathbb{R})$ such that $f(0)=0$. Let $(s_1,s_2)\in]0,\infty[^2$ and $(r_1,r_2)\in[1,\infty]^2$. We assume that $s_1<{d}/{p}$ or that $s_1={d}/{p}$ and $r_1=1$.

Then, for every real-valued function $u$ in $\dot{\B}^{s_1}_{p,r_1}\cap\dot{\B}^{s_2}_{p,r_2}\cap L^\infty$, the function $f\circ u$ belongs to $\dot{\B}^{s_1}_{p,r_1}\cap\dot{\B}^{s_2}_{p,r_2}\cap L^\infty$ and we have
$$\norme{f\circ u}_{\dot{\B}^{s_k}_{p,r_k}}\leq C\left(f',\norme{u}_{L^\infty}\right)\norme{u}_{\dot{\B}^{s_k}_{p,r_k}}\quad\hbox{for}\  k\in\{1,2\}.$$
\end{prop}
As a consequence (see \cite[Cor. 2.66]{Bahouri2011}), if $g$ is a $\cC^\infty(\R)$ function such that $g'(0)=0$, then,  
for all $u,v$ in $\dot\B^s_{p,1}\cap L^\infty$ with $s>0,$ we have
\begin{equation}\label{eq:compo}
\|g(v)-g(u)\|_{\dot\B^s_{p,1}} \leq C\Bigl(\|v-u\|_{L^\infty}\|(u,v)\|_{\dot\B^s_{p,1}} + 
\|v-u\|_{\dot\B^s_{p,1}} \|(u,v)\|_{L^\infty}\Bigr)\cdotp
\end{equation}
We also need more involved product laws to handle 
the high frequencies of some non-linear terms.
 \begin{prop} \label{prohigh} 
For all $\sigma\geq s>0$, we have, for the threshold $J$ between low and high frequencies:
\begin{equation}\label{eq:prod4}
\|ab\|^h_{\dot\B^{s}_{2,1}}\lesssim  \|a\|_{\dot\B^{\frac d2}_{2,1}}\|b\|^h_{\dot\B^{s}_{2,1}}
+\|b\|_{\dot\B^{\frac d2}_{2,1}}\|a\|^h_{\dot\B^{s}_{2,1}}+ 2^{(s-\sigma)J}(\|a\|^\ell_{\dot\B^{\frac d2}_{2,1}}\|b\|^\ell_{\dot\B^{\sigma}_{2,1}}
+ \|b\|^\ell_{\dot\B^{\frac d2}_{2,1}}\|a\|^\ell_{\dot\B^{\sigma}_{2,1}}).
\end{equation}
\end{prop}
\begin{proof} 
Recall the  following so-called Bony decomposition 
(first introduced by J.-M. Bony in \cite{Bony})
for the product 
of two tempered distributions $f$ and $g$:
$$fg=T_fg+T'_gf\with T_fg\triangleq\sum_{j\in\Z}\dot S_{j-1}f\,\ddj g\andf
T'_gf\triangleq \sum_{j\in\Z}\dot S_{j+2}g\,\ddj f.$$
Using this decomposition 
 and further splitting  $a$ and $b$ into low and high frequencies,
 we get
 $$ ab=T_{a^\ell}b^\ell+T'_{b^\ell}a^\ell+T'_{b}a^h+T_{a}b^h+T'_{b^h}a^\ell+T_{a^h}b^\ell.$$
 All the terms in the right-hand side, except for the last two ones, may be bounded
 by means of standard results of continuity for operators $T$ and $T'$
   (see again \cite[Chap. 2]{Bahouri2011}). Provided $\sigma\geq s>0,$ we get,
 $$\begin{aligned}
 \|T'_{b^\ell}a^\ell\|^h_{\dot\B^{s}_{2,1}}& \lesssim 2^{(s-\sigma)J}\|T'_{b^\ell}a^\ell\|^h_{\dot\B^{\sigma}_{2,1}}\lesssim 2^{(s-\sigma)J}\|b^\ell\|_{L^{\infty}}\|a^\ell\|_{\dot\B^{\sigma}_{2,1}},\\
 \|T'_{b}a^h\|_{\dot\B^{s}_{2,1}}&\lesssim\|b\|_{L^\infty}\|a^h\|_{\dot\B^{s}_{2,1}}.
\end{aligned}$$ 
Let $J$ be the integer corresponding to the threshold 
between low and high frequencies. Since $a^\ell=\dot S_{J+1}a$ and $b^h=({\rm Id}-\dot S_{J+1})b,$ we see that 
$$
T'_{b^h} a^\ell=  \dot S_{J+2}b^h\,\dot\Delta_{J+1} a^\ell.$$
Consequently, as 
$\dot S_{J+2}b^h=(\dot\Delta_{J-1}+\dot\Delta_{J}+\dot\Delta_{J+1})b^h,$
$$\|T'_{b^h} a^\ell\|_{\dot\B^{s}_{2,1}}\lesssim
\|\dot\Delta_{J+1}a^\ell\|_{L^\infty} \|\dot S_{J+2}b^h\|_{L^2}\lesssim \|a\|_{L^\infty} 
 \|b\|^h_{\dot\B^{s}_{2,1}}.$$
Adding this latter inequality to the previous one and to the symmetric ones (with just operator $T$ instead of $T'$), together with the embeddings
 $\dot\B^{\frac d2}_{2,1}\hookrightarrow L^\infty$, completes the proof of \eqref{eq:prod4}. 
\end{proof}
To handle commutators  in high frequencies, we need the following lemma.

  \begin{lem}
\label{commhigh} Let $s>0$. For $j\in\Z,$ 
denote $\mathfrak{R}_j\triangleq [w,\ddj] z=w\ddj z-\ddj(wz)$. There exists  a constant $C$ depending only on  the threshold 
number $J$ between low and high frequencies and on $s,$ $d,$ such that

$$\displaylines{
\sum_{j\geq J}\left(2^{js}\norme{\mathfrak{R}_j}_{L^2}\right)\leq C\Bigl(\norme{\nabla w}_{\dot{\B}^{\frac d2}_{2,1}}\norme{z}^h_{\dot{\B}^{s-1}_{2,1}}
+ 2^{(s-\sigma_1)J}\norme{z}^\ell_{\dot{\B}^{\frac{d}{2}}_{2,1}}\norme{w}^\ell_{\dot{\B}^{\sigma_1}_{2,1}}\hfill\cr\hfill+\norme{z}_{\dot{\B}^{\frac d2-k}_{2,1}}\norme{w}^h_{\dot{\B}^{s+k}_{2,1}}
+ 2^{(s-\sigma_2-1)J}\norme{z}^\ell_{\dot{\B}^{\sigma_2}_{2,1}}\norme{ \nabla  w}^\ell_{\dot{\B}^{\frac{d}{2}}_{2,1}}\Bigr),}$$
for any $k\geq0$, $\sigma_1 \geq s$ and $\sigma_2\in\R.$

\end{lem}
\begin{proof}
From Bony's decomposition recalled above
and the fact that $\dot\Delta_j\dot\Delta_{j'}=0$ for $|j-j'|\geq2,$
we deduce that
$$\begin{aligned}\mathfrak{R}_j&=-\ddj(T'_{z}w)-\sum_{|j'-j|\leq4} [\ddj,\dot S_{j'-1}w]\dot\Delta_{j'}z
- \sum_{|j'-j|\leq 1}\left(\dot{S}_{j'-1}w-w\right)\dot{\Delta}_j\dot{\Delta}_{j'} z\\
&\triangleq\mathcal{R}^1_j + \mathcal{R}^2_j +  \mathcal{R}^3_j. 
\end{aligned}$$
To estimate  $\mathcal{R}^1_j$, we  use the decomposition
$$
T'_zw=T'_{z^\ell}w^\ell +T'_{z^h}w^\ell+T'_zw^h$$
and proceed as in the proof of Proposition \ref{prohigh}. In the end, we get
$$\|T'_{z}w\|^h_{\dot\B^{s}_{2,1}}\lesssim\|z\|_{\dot\B^{-k}_{\infty,1}}\|w\|^h_{\dot\B^{s+k}_{2,1}}+2^{(s-\sigma_1)J} \|z\|^\ell_{\dot\B^{\frac d2}_{2,1}}\|w\|^\ell_{\dot\B^{\sigma_1}_{2,1}}.$$
Therefore, since $\dot\B^{\frac d2-k}_{2,1}\hookrightarrow\dot\B^{-k}_{\infty,1},$
\begin{equation}\label{eq:Rj1}
\sum_{j\in\Z}\left(2^{js}\norme{\cR_j^1}_{L^2}\right)\lesssim
\|z\|_{\dot\B^{\frac{d}{2}-k}_{2,1}}\|w\|^h_{\dot\B^{s+k}_{2,1}}+ 2^{(s-\sigma_1)J}\|z\|^\ell_{\dot\B^{\frac d2}_{2,1}}\|w\|^\ell_{\dot\B^{\sigma_1}_{2,1}}.
\end{equation}
Next, taking advantage of Lemma \ref{C1}, we see that
if $j'\geq J$ and $|j-j'|\leq 4,$ then we have 
$$
2^{js} \|[\ddj,\dot S_{j'-1}w] \dot\Delta_{j'}z\|_{L^2}
\lesssim \|\nabla \dot S_{j'-1}w\|_{L^{\infty}} \, 2^{j'(s-1)}\| \dot\Delta_{j'}z\|_{L^2}
$$
while, if $j'<J,$ $j\geq J$ and $|j-j'|\leq 4,$ 
\begin{align*}
2^{js} \|[\ddj,\dot S_{j'-1}w]\dot\Delta_{j'}z\|_{L^2}
&\lesssim 2^{J(s-\sigma_2-1)}2^{j(\sigma_2+1)} \|[\ddj,\dot S_{j'-1}w]\dot\Delta_{j'}z\|_{L^2}\\
&\lesssim 2^{J(s-\sigma_2-1)}\|\nabla\dot S_{j'-1}w\|_{L^{\infty}} \, 2^{j'\sigma_2}\| \dot\Delta_{j'}z\|_{L^2}.
\end{align*}
Therefore, 
\begin{align*}
\sum_{j\geq J}\left(2^{js}\norme{\cR_j^2}_{L^2}\right)&\lesssim 
\norme{z}^h_{\dot{\B}^{s-1}_{2,1}}\norme{\nabla w}_{L^\infty}+2^{(s-\sigma_2-1)J}\norme{z}^\ell_{\dot{\B}^{\sigma_2}_{2,1}}\norme{\nabla w}_{L^{\infty}}^\ell\cdotp
\end{align*}
Then, with suitable embeddings, one gets
\begin{align}\label{eq:Rj22}
\sum_{j\geq J}\left(2^{js}\norme{\cR_j^2}_{L^2}\right)&\lesssim 
\norme{z}^h_{\dot{\B}^{s-1}_{2,1}}\norme{\nabla w}_{\dot{\B}^{\frac{d}{2}}_{2,1}}+2^{(s-\sigma_2-1)J}\norme{z}^\ell_{\dot{\B}^{\sigma_2}_{2,1}}\norme{\nabla w}_{\dot{\B}^{\frac{d}{2}}_{2,1}}^\ell\cdotp
\end{align}
Finally, for all $j\geq J$ and $|j'-j|\leq 1,$ we have
$$\begin{aligned}
2^{js}\|(\dot{S}_{j'-1}w-w)\dot{\Delta}_j\dot{\Delta}_{j'}z\|_{L^2}
&\leq 2^j\|\dot{S}_{j'-1}w-w\|_{L^\infty}\, 2^{j(s-1)}\| \dot\Delta_{j'}\ddj z\|_{L^2}\\
&\leq C\| \nabla (\dot{S}_{j'-1}w-w)\|_{L^\infty}\, 2^{j(s-1)}\| \ddj z\|_{L^2}.\end{aligned}
$$
Hence
\begin{equation}\label{eq:Rj3}
\sum_{j\geq J}\left(2^{js}\norme{\cR_j^3}_{L^2}\right)\leq 
C\|\nabla w\|_{L^\infty}\|z\|^h_{\dot\B^{s-1}_{2,1}}.\end{equation}
Putting \eqref{eq:Rj1}, \eqref{eq:Rj22} and \eqref{eq:Rj3} together  yields the desired estimate.
\end{proof}

\begin{lem}\label{compositionlp}
Let $J\in\mathbb Z$ be the threshold between low and high frequencies.
Let $F:\mathbb R^n\to\mathbb R^n$ be a smooth function, and let
$w=w(x)$ be an $\mathbb R^n$-valued function on $\mathbb R^d$. Then
\begin{align}
&\|F(w)-F(0)-DF(0)w\|_{\dot{\B}^{s}_{2,1}}^{\ell}
\leq C_{w} \|w\|_{\dot{\B}^{\frac{d}{2}}_{2,1}}
\big(
\|w\|_{\dot{\B}^{s}_{2,1}}^{\ell}
+2^{J(s-\sigma)}\|w\|_{\dot{\B}^{\sigma}_{2,1}}^{h}
\big),
\quad s>-\frac{d}{2},\quad \sigma\geq-\frac{d}{2},
\label{q1}\\
&\|F(w)-F(0)-DF(0)w\|_{\dot{\B}^{s}_{2,1}}^{h}
\leq C_{w}\|w\|_{\dot{\B}^{\frac{d}{2}}_{2,1}}
\big(
2^{J(s-\sigma)}\|w\|_{\dot{\B}^{\sigma}_{2,1}}^{\ell}
+\|w\|_{\dot{\B}^{s}_{2,1}}^{h}
\big),
\quad s>-\frac{d}{2},\quad \sigma\in\mathbb R,
\label{q2}
\end{align}
where $C_w>0$ depends only on $\|w\|_{L^\infty}$, $F''$, $s$, $\sigma$ and $d$.

Furthermore, one has
\begin{align}
&\|F(w)-F(0)-DF(0)w\|_{\dot{\B}^{s}_{2,\infty}}^{\ell}
\leq C_{w}\|w\|_{\dot{\B}^{\frac{d}{2}}_{2,1}}
\big(
\|w\|_{\dot{\B}^{s}_{2,\infty}}^{\ell}
+2^{J(s-\frac{d}{2}-1)}
\|w\|_{\dot{\B}^{\frac{d}{2}+1}_{2,1}}^{h}
\big),
\quad s\geq-\frac{d}{2}.
\label{q5}
\end{align}
\end{lem}

\begin{proof}
Set $R(w):=F(w)-F(0)-DF(0)w $. Since $F:\mathbb R^n\to\mathbb R^n$ is smooth, Taylor's formula gives, for $1\leq k\leq n$,
\[
R_k(w)=\sum_{i=1}^n w_i r_i^k(w)\quad\text{
where}\quad 
r_i^k(w)
:=
\sum_{j=1}^n
w_j
\int_0^1(1-\theta)\,
\partial_{ij}F_k(\theta w)\,d\theta .
\]
For each $i,k$, $r_i^k$ is a smooth scalar function of $w$ and satisfies
$r_i^k(0)=0$. Hence, by Proposition \ref{Composition}, $\|r_i^k(w)-r_i^k(0)\|_{\dot{\B}^{\frac d2}_{2,1}}
\leq C_w \|w\|_{\dot{\B}^{\frac d2}_{2,1}}.$
If $-\frac{d}{2}<s\leq \frac{d}{2}$, then the classical product estimate \eqref{eq:prod2} yields
\begin{align*} 
\|w_ir_i^k(w)\|_{\dot{\B}^{s}_{2,1}}^\ell 
&\lesssim \|r_i^k(w)\|_{\dot{\B}^{\frac d2}_{2,1}} \|w_i\|_{\dot{\B}^{s}_{2,1}} \leq C_w \|w\|_{\dot{\B}^{\frac d2}_{2,1}} \left( \|w_i\|_{\dot{\B}^{s}_{2,1}}^\ell + 2^{J(s-\sigma)} \|w_i\|_{\dot{\B}^{\sigma}_{2,1}}^h \right), \end{align*} 
where we used $\sigma\geq s$ in the last line. If $\sigma>\frac{d}{2}$, then one can replace \eqref{eq:prod2} by \eqref{eq:prod1} and using the embedding $\dot{\B}^{\frac{d}{2}}_{2,1}\hookrightarrow L^\infty$.  If $-\frac{d}{2}\leq  \sigma<s\leq \frac{d}{2}$, then we decompose $w_i=w_i^\ell+w_i^h$ and use \eqref{eq:prod2}--\eqref{eq:prod3} to get
 \begin{align*}
\|w_ir_i^k(w)\|_{\dot{\B}^{s}_{2,1}}^\ell &\leq \|w_i^\ell r_i^k(w)\|_{\dot{\B}^{s}_{2,1}}^\ell + 2^{J(s-\sigma)}\|w_i^h r_i^k(w)\|_{\dot{\B}^{\sigma}_{2,\infty}}^\ell \\ 
&\leq C_w \|w\|_{\dot{\B}^{\frac d2}_{2,1}} \left( \|w_i\|_{\dot{\B}^{s}_{2,1}}^\ell + 2^{J(s-\sigma)} \|w_i\|_{\dot{\B}^{\sigma}_{2,1}}^h \right),
 \end{align*}
If $-\frac{d}{2}<\sigma$ and $\sigma>\frac{d}{2}$ holds, then replacing \eqref{eq:prod2} by \eqref{eq:prod1}leads to the same estimate. Consequently, we obtain \eqref{q1}.

For the high-frequency estimate, we decompose $w_i=w_i^\ell+w_i^h$. Similarly, for $-\frac{d}{2}<s\leq \frac{d}{2}$ and $s\geq \sigma$,
 \begin{align*} \|w_ir_i^k(w)\|_{\dot{\B}^{s}_{2,1}}^h &\lesssim 2^{J(s-\sigma)}\|r_i^k(w)\|_{\dot{\B}^{\frac d2}_{2,1}} \|w_i\|_{\dot{\B}^{s}_{2,1}} \\ &\lesssim \|w\|_{\dot{\B}^{\frac d2}_{2,1}} \left( 2^{J(s-\sigma)} \|w_i\|_{\dot{\B}^{\sigma}_{2,1}}^\ell + \|w_i\|_{\dot{\B}^{s}_{2,1}}^h \right), \end{align*}
and for $-\frac{d}{2}<s< \sigma\leq \frac{d}{2}$,
 \begin{align*} \|w_ir_i^k(w)\|_{\dot{\B}^{s}_{2,1}}^h &\lesssim 2^{J(s-\sigma)}\|w_i^\ell r_i^k(w)\|_{\dot{\B}^{\sigma}_{2,1}}^h + \|w_i^h r_i^k(w)\|_{\dot{\B}^{s}_{2,1}}^h \\ &\lesssim 2^{J(s-\sigma)} \|r_i^k(w)\|_{\dot{\B}^{\frac d2}_{2,1}} \|w_i\|_{\dot{\B}^{\sigma}_{2,1}}^\ell + \|r_i^k(w)\|_{\dot{\B}^{\frac d2}_{2,1}} \|w_i\|_{\dot{\B}^{s}_{2,1}}^h \\ &\lesssim \|r_i^k(w)\|_{\dot{\B}^{\frac d2}_{2,1}} \left( 2^{J(s-\sigma)} \|w_i\|_{\dot{\B}^{\sigma}_{2,1}}^\ell + \|w_i\|_{\dot{\B}^{s}_{2,1}}^h \right). \end{align*}
 The higher-order indices can be addressed using \eqref{eq:prod1} instead of \eqref{eq:prod2}. The third estimate \eqref{q5} is similar to \eqref{q1}, but one uses \eqref{eq:prod3} carefully. This finishes the proof of Lemma \ref{compositionlp}. 
\end{proof}

\section*{Acknowledgments}

T. Crin-Barat is supported by the project ANR-24-CE40-3260 – Hyperbolic Equations, Approximations $\&$ Dynamics (HEAD) and the project ANR-25-CE40-5565 (Cookie). L.-Y. Shou is supported by the National Natural Science
Foundation of China under Grant No. 12671258. Q. Zhu would like to thank his Ph.D. supervisor Rapha\"el Danchin for some helpful discussions.

 % is currently a PhD student,

\section*{Data availability statement}

Data sharing is not applicable to this article, as no datasets were generated or analyzed during the current study.

\section*{Conflict of interest statement}

The authors declare that they have no conflict of interest.

\printbibliography
\vfill

\bigbreak
Timothée Crin-Barat \hfill\break\indent
{\sc Université de Toulouse, Institut de Mathématiques de Toulouse, Route de Narbonne 118, 31062 CEDEX 9 Toulouse, France, \hfill\break\indent
{\it Email address}: {\tt timothee.crin-barat@math.univ-toulouse.fr}}

\bigbreak
Ling-Yun Shou\hfill\break\indent
{\sc School of Mathematical Sciences, Ministry of Education Key Laboratory of NSLSCS, and Key Laboratory of Jiangsu Provincial Universities of FDMTA, Nanjing Normal University, Nanjing 210023, China\hfill\break\indent
{\it Email address}: {\tt shoulingyun11@gmail.com}}

\bigbreak
Qimeng Zhu \hfill\break\indent
{\sc Laboratoire d'Analyse et Math\'ematiques Appliqu\'ees (LAMA UMR8050), \hfill\break\indent
Université Paris-Est Cr\'eteil, Cr\'eteil 94010,
France.\hfill\break\indent
{\it Email address}: {\tt qimeng.zhu@u-pec.fr }}

\end{document}